\documentclass[a4paper,12pt]{amsart}
\usepackage{amsmath}
\usepackage{amsmath}
\usepackage{amssymb}
\usepackage{amscd}
\usepackage{mathrsfs}
\usepackage[all]{xy}
\usepackage{yhmath}

\def\mathcal{\mathscr}
\everymath{\displaystyle}
\newtheorem{thm}{Theorem}[section]
\newtheorem{lem}[thm]{Lemma}
\newtheorem{cor}[thm]{Corollary}
\newtheorem{prop}[thm]{Proposition}

\theoremstyle{definition}

\newtheorem{rem}[thm]{Remark}

\newtheorem{defn}[thm]{Definition}

\newcommand{\mf}[1]{{\mathfrak{#1}}}

\newcommand{\mca}[1]{{\mathcal{#1}}}
\newcommand{\mtrg}[1]{{\mathring{#1}}}
\newcommand{\fbp}[2]{ \,_{#1}\!\times_{#2}}

\def\Z{{\mathbb Z}}
\def\C{{\mathbb C}}

\def\R{{\mathbb R}}

\def\Aut{\text{\rm Aut}\,}

\def\con{\text{\rm con}\,}

\def\dim{\text{\rm dim}\,}

\def\dR{\text{\rm dR}}

\def\ep{\varepsilon} 

\def\ev{\text{\rm ev}\,}
\def\evl{\ev^{\mca{L}}}

\def\exterior{\text{\rm ext}\,}

\def\Hom{\text{\rm Hom}}

\def\id{\text{\rm id}}
\def\Image{\text{\rm Im}\,}
\def\interior{\text{\rm int}\,}

\def\Ker{\text{\rm Ker}\,}

\def\ol{\overline}

\def\part{\partial}
\def\ph{\varphi}

\def\pr{\text{\rm pr}}

\def\rk{\text{\rm rk}\,}

\def\sgn{\text{\rm sgn}}
\def\sing{\text{\rm sing}}

\def\supp{\text{\rm supp}\,}
\def\sym{\mathbb{S}\,}

\def\wh{\widehat}
\def\wt{\widetriangle}

\begin{document}
\pagestyle{plain}
\thispagestyle{plain}

\title[Chain level loop bracket and pseudo-holomorphic disks]
{Chain level loop bracket and pseudo-holomorphic disks}

\author[Kei Irie]{Kei Irie}
\address{Research Institute for Mathematical Sciences, Kyoto University,
Kyoto 606-8502, Japan}
\email{iriek@kurims.kyoto-u.ac.jp}


\begin{abstract}
Let $L$ be a Lagrangian submanifold in a symplectic vector space which is closed, oriented and spin. 
Using virtual fundamental chains of moduli spaces of nonconstant pseudo-holomorphic disks with boundaries on $L$, 
one can define a Maurer-Cartan element of a Lie bracket operation in string topology (the loop bracket) defined at chain level. 
This observation is due to Fukaya, who also pointed out its important consequences in symplectic topology. 
The goal of this paper  is to work out details of this observation. 
Our argument is based on a string topology chain model previously introduced by the author, 
and the theory of Kuranishi structures on moduli spaces of pseudo-holomorphic disks, 
which has been developed by Fukaya-Oh-Ohta-Ono. 
\end{abstract}

\begin{center} 
\large Link to an erratum for the published version 
\end{center} 
\begin{center} 
Kei Irie 
\end{center} 

The following article is identical to v2, which does \textit{not} reflect corrections and revisions made in the published version. 
The published version of this paper is available on the following webpage of J. Topology: 

https://doi.org/10.1112/topo.12140

The aim of this update is to post a link to an erratum to the published version: 

https://sites.google.com/view/kei-irie-math/errata

As explained in the erratum, all errors can be fixed and none of them affect the main result of this paper. 
Nonetheless, I think this erratum is important especially for those using technical details of this paper in their own work. 
If further errors in the published version are found, I will update the erratum on the above webpage. 

\newpage 

\maketitle

\section{Introduction} 

The study of Lagrangian submanifolds is one of central topics in symplectic topology. 
In the monumental paper \cite{Gromov_pseudoholomorphic}, 
Gromov proved that the first Betti number of 
a closed Lagrangian submanifold in $\C^n$ (with the standard symplectic structure) is nonzero, 
using moduli spaces of (perturbed) pseudo-holomorphic disks with boundaries on the Lagrangian submanifold. 
On the other hand, string topology is the study of algebraic structures on (homology of) loop spaces, 
introduced by Chas-Sullivan \cite{ChSu_99}. 

In this paper we discuss an application of \textit{chain level} string topology operations to the pseudo-holomorphic curve theory in symplectic topology. 
This idea is due to Fukaya \cite{Fukaya_06}, who also pointed out its important consequences, 
including a proof of Audin's conjecture for (closed, oriented and spin) aspherical Lagrangian submanifolds in $\C^n$, 
and a complete classification of 
orientable, closed, prime three-manifolds admitting Lagrangian embeddings into $\C^3$. 

Let us briefly sketch the key argument in \cite{Fukaya_06}. 
Let $L$ be a closed, oriented and spin Lagrangian submanifold in $\C^n$, and 
$\mca{L}L:= C^\infty(S^1, L)$ be the space of free loops on $L$. 
Let $D:= \{ z \in \C \mid |z| \le 1\}$, 
and let $\mca{M}$ denote the (compactified) moduli space of nonconstant holomorphic maps 
$(D, \partial D) \to (\C^n, L)$ modulo $\text{Aut}(D, 1)$. 
Then one can define a map $\ev_{\mca{M}}: \mca{M} \to \mca{L}L$ by $\ev_{\mca{M}}(u):= u|_{\partial D}$
(strictly speaking, this ``definition'' has an ambiguity up to parametrizations of loops, however we omit this issue for the moment). 
Considering virtual fundamental chain of the moduli space $\mca{M}$, 
the pair $(\mca{M}, \ev_{\mca{M}})$ defines a ``chain'' $x \in \mca{C}_*(\mca{L}L)$. 
Here $\mca{C}_*(\mca{L}L)$ denotes the complex of ``chains'' on $\mca{L}L$, 
on which the \textit{loop bracket} is defined and makes $\mca{C}_*(\mca{L}L)$ a dg Lie algebra. 
Since the codimension $1$ boundary of $\mca{M}$ consists of configurations of two disks glued at a point, one sees that 
the chain $x$ satisfies the \textit{Maurer-Cartan equation} 
\begin{equation}\label{170627_4} 
\partial x  - \frac{1}{2} [x,x] = 0
\end{equation} 
where $[\, , \,]$ denotes the loop bracket defined at \textit{chain level}. 
Nextly, we take a time-dependent Hamiltonian $H$ on $\C^n$ which displaces $L$, 
and define a moduli space $\mca{N}$ which consists of solutions of the Cauchy-Riemann equation perturbed by $H$. 
Then, the associated chain $y:= (\mca{N}, \ev_{\mca{N}}) \in \mca{C}_*(\mca{L} L)$ satisfies 
\begin{equation}\label{170627_5} 
\partial y - [x, y] = z 
\end{equation}
where $z$ is another chain whose ``symplectic area zero part'' 
is a cycle representing the fundamental class of $L$. 
Once we obtain chains $x$, $y$, $z$ satisfying equations (\ref{170627_4}) and (\ref{170627_5}), 
using the homotopy transfer theorem for $L_\infty$-algebras, 
one can formulate an equivalent result on homology of the free loop space (Theorem \ref{161011_1}). 
This result has the following remarkable consequences: 
\begin{itemize}
\item[(i):] 
If $L$ is aspherical, then 
there exists  $a \in H_1(L: \Z)$ with Maslov index $2$ and positive symplectic action 
(Corollary \ref{170421_1}). 
\item[(ii):] 
If $n=3$ and $L$ is prime as a three-manifold, then $L$ is diffeomorphic to $S^1$ times a closed surface
(Corollary \ref{170622_1}). 
\end{itemize} 
(i) in particular confirms Audin's conjecture for Lagrangian tori in $\C^n$, 
and (ii) gives 
a complete classification of 
orientable, closed, prime three-manifolds admitting Lagrangian embeddings into $\C^3$; 
see Section 3 for details and previous related works. 

In the above argument, 
$\mca{C}_*(\mca{L}L)$
denotes the complex of ``chains'' on $\mca{L}L$, 
on which the loop bracket is defined and makes it a dg Lie algebra. 
It is a highly nontrivial technical problem to define such chain models of the free loop space, in particular those compatible with virtual techniques in the pseudo-holomorphic curve theory. 
Partly due to this issue, 
in spite of the importance of its consequences, 
full details of the above argument have not been available so far. 

In \cite{Irie_17}, the author developed foundations for part of chain level algebraic structures (specifically, Batalin-Vilkovisky structure) in string topology, 
using de Rham chains on spaces of Moore loops with arbitrarily many marked points. 
The goal of this paper is to combine techniques from \cite{Irie_17} with the theory of Kuranishi structures \cite{FOOO_Kuranishi} 
to work out details of the argument sketched above. 

Now we describe the structure of this paper. 
The goal of the first part (Sections 2--6) is to state the main result and introduce our setup in string topology. 
Section 2 explains some preliminaries on $L_\infty$-algebras, in particular the homotopy transfer theorem. 
Section 3 states the main result (Theorem \ref{161011_1}) on homology of the free loop space.  
We also recall a few applications in symplectic topology from \cite{Fukaya_06}. 
In Sections 4--6, we introduce our setup in string topology, following \cite{Irie_17} with minor modifications, 
and reduce Theorem \ref{161011_1} to a chain-level statement (Theorem \ref{161215_1}). 
Further details will be explained in the last paragraph of Section 3. 

The goal of the second part (Sections 7--10) is to prove Theorem \ref{161215_1}. 
The plan of the proof will be explained at the beginning of Section 7. 
Our proof uses the theory of Kuranishi structures on moduli spaces of (perturbed) pseudo-holomorphic disks. 
In particular, our arguments heavily rely on \cite{FOOO_Kuranishi} by Fukaya-Oh-Ohta-Ono. 
Section 10 very briefly explains some notions in the theory of Kuranishi structures, mainly to fix notations. 

\textbf{Conventions.}
Throughout this paper all manifolds are assumed to be of $C^\infty$. 
All vector spaces are over $\R$, unless otherwise specified. 

\textbf{Acknowledgements.}
The author appreciates Kenji Fukaya for sharing his time and insights into virtual techniques in 
the pseudo-holomorphic curve theory, 
and his comments on an early version of this paper. 
The author also appreciates the Simons Center for Geometry and Physics for a great work environment.
This work is supported by JSPS Postdoctoral Fellowship for Research Abroad. 

\section {Preliminaries on $L_\infty$-algebras} 

We briefly recall basics of $L_\infty$-algebras, partially following \cite{Latschev_15}. 

\subsection{Bar construction} 

Let $C = \bigoplus_{i \in \Z} C_i$ be a $\Z\,$-graded vector space. 
For every integer $k \ge 1$, 
let $\sym_k$ denote the $k$-th symmetric group, 
and 
let us define an $\sym_k$\,-action on $C^{\otimes k}$ by 
\[ 
\rho \cdot (c_1 \otimes \cdots \otimes c_k) := \ep(\rho: c_1, \ldots, c_k) \cdot c_{\rho(1)} \otimes \cdots \otimes c_{\rho(k)}
\]
where $\ep(\rho: c_1, \ldots, c_k):= \prod_{\substack{i<j \\ \rho(i)>\rho(j)}} (-1)^{|c_i||c_j|}$. 
Let $S^kC$ denote the quotient of $C^{\otimes k}$ by the $\sym_k$-action, 
and set $SC:= \bigoplus_{1 \le k \le \infty} S^kC$.

We define a coproduct $\Delta: SC \to SC^{\otimes 2}$ by 
\[ 
\Delta(c_1 \ldots c_k) := \sum_{\substack{k_1+k_2 = k \\ \rho \in \sym_k}} \frac{\ep(\rho:c_1,\ldots,c_k)}{k_1!k_2!} \cdot c_{\rho(1)} \cdots c_{\rho(k_1)} \otimes c_{\rho(k_1+1)} \cdots c_{\rho(k)}. 
\]
Then $\Delta$ is coassociative, namely $(1 \otimes \Delta) \circ \Delta = (\Delta \otimes 1) \circ \Delta$. 
We denote $\Delta$ by $\Delta_C$ when we need to specify $C$.

\subsection{$L_\infty$-algebras and $L_\infty$-homomorphisms}

For any $\Z$-graded vector space $C$ and $n \in \Z$, 
we define a shifted complex $C[n]$ by 
$C[n]_d:= C_{n+d} \,(\forall d \in \Z)$. 

\begin{defn}
\begin{enumerate}
\item[(i):] 
An $L_\infty$-algebra is a pair of a graded vector space $C$ and 
a linear map $l: S(C[-1])  \to S(C[-1])$ such that 
\begin{itemize}
\item $|l|=-1$.
\item $l$ is a coderivation; $\Delta \circ l = (l \otimes 1 + 1 \otimes l) \circ \Delta$.
\item $l^2=0$.
\end{itemize}
For each integer $k \ge 1$, we define $l_k: S^k(C[-1])\to C[-1]$ by 
$l_k:= \pr_1 \circ l|_{S^k(C[-1])} $, where 
$\pr_1: S(C[-1]) \to S^1(C[-1]) \cong C[-1]$
denotes the projection. 
\item[(ii):]
Let $(C, l)$ and $(C', l')$ be $L_\infty$-algebras. 
An $L_\infty$-homomorphism from $(C, l)$ to $(C', l')$ is a linear map 
$f: S(C[-1]) \to S(C'[-1]) $ such that 
\begin{itemize}
\item $|f|=0$. 
\item $f$ is a coalgebra map; $\Delta_{C'[-1]} \circ f = (f \otimes f) \circ \Delta_{C[-1]}$. 
\item $\l' \circ f = f \circ l$. 
\end{itemize}
For each integer $k \ge 1$, we define $f_k: S^k(C[-1]) \to C'[-1]$ by 
$f_k:= \pr_1 \circ f|_{S^k(C[-1]) }$. 
\end{enumerate}
\end{defn}

Here is another definition of $L_\infty$-algebras and $L_\infty$-homomorphisms using exterior products. 
For each integer $k \ge 1$, let $\Lambda^k C$ denote the quotient of $C^{\otimes k}$ by the $\sym_k$-action defined by 
\[ 
\rho \cdot (c_1 \otimes \cdots \otimes c_k) : = \sgn(\rho) \cdot \ep(\rho: c_1, \ldots, c_k) \cdot c_{\rho(1)} \otimes \cdots \otimes c_{\rho(k)}. 
\] 
Then there exists a natural isomorphism 
\[
\sigma_k: (\Lambda^k C)[-k] \to S^k (C[-1]); \quad c_1 \wedge \cdots \wedge c_k \mapsto (-1)^{ \sum_i (k-i) |c_i|} c_1 \cdots c_k. 
\]
We define $\lambda_k: \Lambda^k C \to C$ by $\lambda_k:= \sigma_1^{-1} \circ l_k \circ \sigma_k$. 

Then one can define an $L_\infty$-structure on $C$ as a sequence $(\lambda_k)_{k \ge 1}$ such that 
each $\lambda_k: \Lambda^k C \to C$ is of degree $k-2$ and satisfies the equation 
\[ 
\sum_{\substack{k_1+k_2 = k+1 \\ \rho \in \sym_k}} \pm \frac{1}{k_1 ! (k-k_1)!} \lambda_{k_2} ( \lambda_{k_1}(c_{\rho(1)} \wedge \cdots \wedge c_{\rho(k_1)}) \wedge c_{\rho(k_1+1)} \wedge \cdots \wedge c_{\rho(k)}) = 0 
\]
where $\pm$ stands for appropriate signs. 
Similarly, one can define an $L_\infty$-homomorphism from $(C, (\lambda_k)_k)$ to $(C', (\lambda'_k)_k)$ as a sequence $(\ph_k)_{k \ge 1}$ such that 
each $\ph_k: \Lambda^k C \to C'$ is of degree $k-1$ and satisfies the equation 
\begin{align*} 
&\sum_{\substack{k_1+k_2=k+1 \\ \rho \in \sym_k}} \pm \frac{1}{k_2! (k_1-1)!} \cdot \ph_{k_1} (\lambda_{k_2} (c_{\rho(1)} \wedge \cdots \wedge c_{\rho(k_2)}) \wedge c_{\rho(k_2+1)} \wedge \cdots \wedge c_{\rho(k)}) \\
= &\sum_{\substack{k_1+\cdots +k_r = k \\ \rho \in \sym_k}} \pm \frac{1}{r! k_1! \cdots k_r!} \cdot \lambda'_r ( \ph_{k_1}(c_{\rho(1)} \wedge \cdots \wedge c_{\rho(k_1)}) \wedge \cdots \wedge \ph_{k_r}(c_{\rho(k-k_r+1)} \wedge \cdots \wedge c_{\rho(k)})) \\
\end{align*} 
where $\pm$ stands for appropriate signs. 

\begin{rem}\label{170829_1} 
For later purposes we need to specify signs for dg Lie algebras. 
A dg Lie algebra is an $L_\infty$-algebra such that $\lambda_k = 0$ for every $k \ge 3$. 
Setting $\partial x := \lambda_1(x)$ and $[x,y]:= \lambda_2(x, y)$, there holds 
\begin{align*} 
& [x,y] + (-1)^{|x||y|} [y,x] = 0, \\ 
& \partial [x, y] = [\partial x, y] + (-1)^{|x|} [x, \partial y], \\ 
& [x, [y,z]] + (-1)^{|x|(|y|+|z|)} [y, [z,x]] + (-1)^{|z|(|x|+|y|)} [z, [x,y]] = 0. 
\end{align*} 
\end{rem} 

We also introduce the following notions for later purposes. 

\begin{defn} 
\begin{enumerate} 
\item[(i):] Let $V$ be a vector space and $A$ be a commutative semigroup. 
A decomposition of $V$ over $A$ is a decomposition $V = \bigoplus_{a \in A} V(a)$, 
where each $V(a)$ is a subspace of $V$, 
and $V(a) + V(a') \subset V(a+a')$ for every $a, a' \in A$. 
\item[(ii):] Let $V$ and $W$ be vector spaces with decompositions over $A$. 
A linear map $f: V^{\otimes k} \to W$ respects these  decompositions if 
\[ 
f (V(a_1) \otimes \cdots \otimes V(a_k)) \subset W(a_1+ \cdots + a_k)
\] 
for every $a_1, \ldots, a_k \in A$. 
\item[(iii):] An $L_\infty$-algebra structure 
$l = (l_k)_k$ on $V$ respects the decomposition of $V$
if $l_k$ respects the decomposition for every $k \ge 1$. 
\end{enumerate} 
\end{defn} 

\subsection{Homotopy transfer theorem} 

Finally we state the homotopy transfer theorem for $L_\infty$-algebras.
The proof is only sketched since it is now standard 
(perhaps goes back to the paper by Kadeishvili \cite{Kadeishvili} on $A_\infty$-algebras). 

\begin{thm}\label{170623_2} 
Let $(C, l)$ be an $L_\infty$-algebra, and 
$H(C) := \Ker l_1/ \Image l_1$. 
Suppose that there are linear maps
\[
\iota: H_*(C) \to C_*, \quad
\pi: C_* \to H_*(C), \quad
\kappa: C_* \to C_{*+1}
\]
such that 
\[
l_1 \circ \iota  = 0, \quad
\pi \circ l_1 = 0, \quad
\pi \circ \iota = \id_{H(C)}, \quad
\id_C  - \iota \circ \pi  = l_1 \circ \kappa + \kappa \circ l_1. 
\]
Then there exist
an $L_\infty$-algebra structure $l^H$ on $H(C)$ 
and an $L_\infty$-homomorphism $p: (C, l) \to (H(C), l^H)$, 
such that $l^H_1 = 0$ and $p_1 = \pi$. 

When $C$ has a decomposition over a commutative semigroup $A$
(thus $H(C)$ also has a decomposition over $A$), 
and linear maps $\iota$, $\pi$ and $\kappa$ respect these decompositions, 
then one can take $l^H$ and $p$ so that they respect the decompositions over $A$.  
\end{thm} 
\begin{proof}
It is sufficient to define sequences $(l^H_{\le k})_{k \ge 1}$ and $(p_{\le k})_{k \ge 1}$ satisfying the following conditions: 
\begin{itemize}
\item $p_{\le k}$ is a coalgebra map of degree $0$ from $S(C[-1])$ to $S(H[-1])$ which respects decompositions over $A$. 
\item $l^H_{\le k}$ is a coderivation of degree $-1$ from $S(H[-1])$ to $S(H[-1])$ which respects decompositions over $A$. 
\item $l^H_{\le k} \circ l^H_{\le k} = 0$ on $S^{\le k+1} (H[-1]) := \bigoplus_{1 \le i \le k+1} S^i(H[-1])$. 
\item $l^H_{\le k} \circ p_{\le k} - p_{\le k} \circ l = 0$ on $S^{\le k} (C[-1]) := \bigoplus_{1 \le i \le k} S^i(C[-1])$. 
\item $p_{\le k} = p_{\le k+1}$ on $S^{\le k} (C[-1])$. 
\item $l^H_{\le k} = l^H_{\le k+1}$ on $S^{\le k}(H[-1])$. 
\item $p_{\le 1} = \pi$, $l^H_{\le 1} =0$.
\end{itemize} 
Once we obtain these sequences, the limits 
$p:= \lim_{k \to \infty} p_{\le k}$ and 
$l^H:= \lim_{k \to \infty} l^H_{\le k}$
satisfy the conditions in the theorem. 

We can define $(l^H_{\le k})_{k \ge 1}$ and $(p_{\le k})_{k \ge 1}$ by upward induction on $k$. 
For $k=1$, $p_{\le 1}$ and $l^H_{\le 1}$ are defined by the last condition. 
We assume that we have defined $p_{\le k}$, $l^H_{\le k}$ and 
are going to define 
$p_{\le k+1}$, $l^H_{\le k+1}$. 
Let $\Hom_A(S^{k+1}C, H)$ denote the space of linear maps from $S^{k+1}C$ to $H$ preserving decompositions over $A$. Namely: 
\begin{align*} 
\Hom_A(S^{k+1}C, H)&:= \{ f \in \Hom(S^{k+1}C, H) \mid  \\ 
&f(C(a_1) \cdots C(a_{k+1} )) \subset H(a_1+\cdots+a_{k+1} ) \, (a_1, \ldots, a_{k+1} \in A)\}. 
\end{align*} 
We define a boundary operator $\partial$ on $\Hom(S^{k+1}C, H)$ by $\partial f:= f \circ l_1$, 
then the homology is isomorphic to $\Hom_A(S^{k+1}H, H)$. 
Our induction assumption shows that $l^H_{\le k} \circ p_{\le k} - p_{\le k} \circ l$ is a cycle in $\Hom_A(S^{k+1}C, H)$, thus one can define 
\[ 
l^H_{k+1}:= [p_{\le k} \circ l - l^H_{\le k} \circ p_{\le k}] \in \Hom_A(S^{k+1} H, H). 
\] 
Then 
$ l^H_{k+1} \circ p_{\le 1} + l^H_{\le k} \circ p_{\le k} - p_{\le k} \circ l$
is a null-homologous cycle, thus there exists
$p_{k+1} \in \Hom_A(S^{k+1}C, H)$ such that 
\[ 
p_{k+1} \circ l_1 = l^H_{k+1} \circ p_{\le 1} + l^H_{\le k} \circ p_{\le k} - p_{\le k} \circ l. 
\] 
Then we define a coderivation $l^H_{\le k+1}$ so that $l^H_{\le k+1} = l^H_{\le k}$ on $S^{\le k}H$, and 
\[ 
l^H_{\le k+1}|_{S^iH \to H} = \begin{cases} l^H_{k+1} &(i=k+1) \\ 0 &(i>k+1). \end{cases} 
\] 
Similarly, we define a coalgebra map $p_{\le k+1}$ so that $p_{\le k+1} = p_{\le k}$ on $S^{\le k}C$, and
\[ 
p_{\le k+1}|_{S^i C \to H} = \begin{cases}  p_{k+1} &(i=k+1) \\ 0 &(i>k+1). \end{cases}
\]
It is easy to check $l^H_{\le k+1} \circ l^H_{\le k+1} =0$ on $S^{\le k+2}H$. 
\end{proof} 

\section{Main result} 

We state the main result (Theorem \ref{161011_1}) and 
explain a few applications to symplectic topology of Lagrangian submanifolds. 
Let $\omega_n$ denote the standard symplectic form on $\C^n$, namely $\omega_n:=\sum_{j=1}^n dx_j \wedge dy_j$, 
and $L$ be a Lagrangian submanifold in $(\C^n, \omega_n)$. 
We assume that $L$ is closed (compact and $\partial L = \emptyset$), 
connected, oriented and spin. 
Let $\mu \in H^1(L: \Z)$ denote the Maslov class. 
Since $L$ is oriented, $\mu (H_1(L:\Z)) \subset 2\Z$. 

\begin{rem} 
Let us explicitly define the Maslov class $\mu$ as follows. 
Let $\Lambda(n):= U(n)/O(n)$ denote the unoriented Lagrangian Grassmannian, 
and consider maps  
\[ 
\tau: L \to \Lambda(n); \, x \mapsto T_x L, \qquad 
{\det}^2: \Lambda(n) = U(n)/O(n) \to U(1). 
\] 
Then we define $\mu: =(\text{det}^2 \circ \tau)^*[U(1)]$, 
where $U(1)$ is identified with $\{ e^{\sqrt{-1} \theta}| \theta \in \R/2\pi \Z\}$, 
and $[U(1)]$ is defined as $[U(1)] : = [d\theta]/2\pi$. 
\end{rem} 

Let $S^1:= \R/\Z$, and $\mca{L}L := C^\infty(S^1, L)$. 
We will often abbreviate $\mca{L}L$ by $\mca{L}$. 
For every $a \in H_1(L: \Z)$, we set
$\mca{L}(a):= \{ \gamma \in \mca{L} \mid [\gamma] = a\}$. 
Obviously $\mca{L} = \bigsqcup_{a \in H_1(L: \Z)} \mca{L}(a)$. 
For each $a \in H_1(L: \Z)$, 
we consider the $C^\infty$-topology on $\mca{L}(a)$ and set 
\[
H^{\mca{L}}(a)_*: = H^\sing_{*+n+\mu(a)-1}(\mca{L}(a): \R), 
\]
where the RHS is the singular homology with respect to the $C^\infty$-topology on $\mca{L}(a)$
(in the following we often omit the superscript ``$\sing$'' ). 
Now we consider the direct product
\[ 
H^\mca{L}_*:= \bigoplus_{a \in H_1(L: \Z)} H^\mca{L}(a)_* 
\]
equipped with the energy filtration; 
for each $E \in \R$, we set
\[ 
F^E H^\mca{L}_*: = \bigoplus_{\omega_n(\bar{a}) >E} H^\mca{L}(a)_* 
\] 
where $\bar{a}$ denotes the unique element in $H_2(\C^n, L)$ satisfying $\partial \bar{a}=a$.
Finally $\wh{H}^\mca{L}_*$ denotes the completion by the energy filtration: 
\[ 
\wh{H}^\mca{L}_*:= \varprojlim_{E \to \infty} H^\mca{L}_* / F^E H^\mca{L}_*. 
\]

Now let us state the main result of this paper: 

\begin{thm}\label{161011_1} 
Let $L$ be a Lagrangian submanifold in $(\C^n, \omega_n)$ 
which is closed, oriented and spin. 
Then, there exist
an $L_\infty$-structure $(l^H_k)_{k \ge 1}$ on $H^\mca{L}$ and 
$X \in \wh{H}^\mca{L}_{-1}$, 
$Y \in \wh{H}^\mca{L}_2$, 
satisfying the following conditions: 
\begin{enumerate}
\item[(i):] $l^H_1=0$. 
\item[(ii):] The $L_\infty$-structure $(l^H_k)_{k \ge 1}$ respects the decomposition of $H^\mca{L}$ over $H_1(L: \Z)$. 
In particular, the $L_\infty$-structure extends to the completion $\wh{H}^\mca{L}$. 
\item[(iii):]  There exists $c>0$ such that $X \in F^c \wh{H}^\mca{L}_{-1}$. 
\item[(iv):]
$X$ and $Y$ satisfy the following equations: 
\begin{equation}\label{170827_1} 
\sum_{k \ge 2}  \frac{1}{k!} l^H_k(X, \ldots, X) = 0, 
\end{equation}
\begin{equation}\label{170827_2} 
\biggl( \sum_{k \ge 2} \frac{1} {(k-1)!} l^H_k(Y, X, \ldots, X) \biggr)_{a=0} = (-1)^{n+1} [L]. 
\end{equation} 
Note that infinite sums in the LHS make sense by the condition (iii). 
$[L]$ in the RHS of (\ref{170827_2}) denotes the image of the fundamental class $[L] \in H_n(L: \R)$ 
by the embedding map $H_*(L: \R) \to H_*(\mca{L}(0): \R)$
which is induced by 
\[ 
L \to \mca{L}(0); \quad x \mapsto \text{constant loop at $x$}. 
\]
\end{enumerate}
\end{thm} 
\begin{rem} 
It will be possible to show that $l^H_2$ coincides (up to sign) with the Chas-Sullivan loop bracket \cite{ChSu_99}, 
and the full $L_\infty$-algebra structure is homotopy equivalent to 
the dg Lie algebra defined by the chain level loop bracket in \cite{Irie_17}. 
However, we do not give complete proofs of these claims in this paper. 
\end{rem} 

\begin{rem} 
If $l^H_k=0$ for $k \ge 3$, assuming that $l^H_2$ coincides with the loop bracket up to sign, 
(\ref{170827_2}) implies $[Y(-a), X(a)] \ne 0$ for some $a \in H_1(L: \Z)$. 
When $L$ is diffeomorphic to $S^1 \times S^2$, this equation implies nonvanishing of the Maslov class $\mu$, 
contradicting a result by Ekholm-Eliashberg-Murphy-Smith (Corollary 1.6 in \cite{EEMS}). 
Therefore, if $L$ is diffeomorphic to $S^1 \times S^2$, there exists at least one nonvanishing higher term in the LHS of (\ref{170827_2}). 
\end{rem}  

Let us quickly recall two applications of Theorem \ref{161011_1} from \cite{Fukaya_06}. 
For further results, consult the original paper \cite{Fukaya_06}. 

\begin{cor}\label{170421_1} 
Suppose that $L$ is aspherical. Then there exists $a \in H_1(L: \Z)$ such that $\mu(a)=2$, 
$\omega_n(\bar{a})>0$ and $H_n(\mca{L}(a): \R) \ne 0$. 
\end{cor}
\begin{proof} 
By the equation (\ref{170827_2}), 
there exist $a_1, \ldots, a_{k-1} \in H_1(L: \Z)$ such that 
\[ 
l^H_k ( Y(- (a_1+ \cdots + a_{k-1})), X(a_1), \ldots, X(a_{k-1})) \ne 0. 
\] 
On the other hand, the assumption that $L$ is aspherical implies that 
$H_i(\mca{L}) \ne 0$ only if $0 \le i \le n $ (Lemma 12.11 in \cite{Fukaya_06}). 
Therefore we obtain  
\[ 
1 \le \mu(a_1 + \cdots + a_{k-1}) \le n+1, \qquad
2-n \le \mu(a_j) \le 2 \quad (1 \le \forall j  \le k-1). 
\] 
Since $\mu(a_1+ \cdots + a_{k-1})>0$, there exists $j$ such that $\mu(a_j)>0$. 
Since $\mu$ takes values in $2\Z$, we obtain $\mu(a_j)=2$. 
Since $X(a_j) \ne 0$ we obtain $\omega_n(\bar{a_j})>0$ and $H_n(\mca{L}(a_j): \R) \ne 0$. 
\end{proof}

Corollary \ref{170421_1} in particular confirms Audin's conjecture: 
every Lagrangian torus in $\C^n$ bounds a disk with positive symplectic area and Maslov index $2$. 
Note that Cieliebak-Mohnke \cite{Cieliebak_Mohnke}
proved Audin's conjecture 
by an approach different from ours. 
For other previous results on this conjecture see \cite{Cieliebak_Mohnke}. 

Another important application is 
a complete classification of 
orientable, closed, prime three-manifolds admitting Lagrangian embeddings into $\C^3$: 

\begin{cor}\label{170622_1} 
A closed, connected, orientable and prime three-manifold $M$ admits a Lagrangian embedding into $(\C^3, \omega_3)$ 
if and only if $M$ is diffeomorphic to $S^1 \times \Sigma$ 
where $\Sigma$ is a closed orientable two-manifold. 
\end{cor} 

The ``if'' part in Corollary \ref{170622_1} is elementary and classically known. 
The ``only if'' part follows from Corollary \ref{170421_1} and some classical results in three-dimensional topology. 
See \cite{Fukaya_06} Section 11 or \cite{Latschev_15} Section 5 for details. 
Note that Evans-K\c{e}dra \cite{Evans_Kedra} and Damian \cite{Damian} proved the same conclusion for \textit{monotone} Lagrangian submanifolds in $\C^3$ which are orientable but not necessarily prime. 

The proof of Theorem \ref{161011_1} occupies the rest of this paper. 
In Sections 4--6 we reduce Theorem \ref{161011_1} to Theorem \ref{161215_1} (see Section 6), 
which will be proved in Sections 7--9 using the pseudo-holomorphic curve theory. 
In Section 4 we introduce the space of Moore loops with marked points, and the notion of de Rham chains on these spaces, 
following \cite{Irie_17} with minor modifications. 
Then we define the chain complex of de Rham chains and study its basic properties. 
In Section 5, we reduce Theorem \ref{161011_1} to Theorem \ref{161214_2}, which asserts the existence of a solution of (\ref{170827_1}), (\ref{170827_2}) at chain level. 
In Section 6, we reduce Theorem \ref{161214_2} to Theorem \ref{161215_1}, 
which asserts the existence of a sequence of approximate solutions connected by ``gauge equivalences''.

\section{de Rham chains on the space of loops with marked points} 

In Section 4.1, we introduce the space of Moore loops with $k+1$ marked points (where $k \in \Z_{\ge 0}$) which we denote by $\mca{L}_{k+1}$. 
In Section 4.2 we fix our conventions on signs. 
In Section 4.3, we define the chain complex $C^\dR_*(\mca{L}_{k+1})$ which consists of ``de Rham chains'' on $\mca{L}_{k+1}$. 
In Section 4.4, we introduce a chain model of $[-1, 1] \times \mca{L}_{k+1}$. 
In Section 4.5, we introduce a natural dg Lie algebra by taking direct products of de Rham chain complexes introduced in Sections 4.3 and 4.4. 

\subsection{Space of Moore loops with marked points} 

First we consider the space of Moore paths 
\[
\Pi:= \{ (T, \gamma) \mid T \in \R_{>0}, \, \gamma \in C^\infty([0, T], L), \, 
\partial_t^m \gamma(0) = \partial_t^m \gamma(T)=0 \, (\forall m \ge 1) \}.
\]
We define evaluation maps $\ev_0, \ev_1: \Pi \to L$ by 
\[ 
\ev_0(T, \gamma):= \gamma(0), \qquad  
\ev_1(T, \gamma):= \gamma(T)
\] 
and a concatenation map 
\[ 
\Pi \fbp{\ev_1}{\ev_0} \Pi \to \Pi ;\quad (\Gamma_0, \Gamma_1) \mapsto \Gamma_0 * \Gamma_1
\] 
by 
\[ 
(T_0, \gamma_0) * (T_1, \gamma_1) := (T_0 + T_1, \gamma_0 * \gamma_1)
\] 
where 
\[
(\gamma_0 * \gamma_1) (t) := \begin{cases} \gamma_0(t)  &(0 \le t \le T_0), \\ \gamma_1(t-T_0) &(T_0 \le t \le T_0+T_1). \end{cases}
\]
Next we consider the space of Moore loops with marked points. 
For every $k \in \Z_{\ge 0}$, 
we define the space $\mca{L}_{k+1}$ 
which consists of $(T, \gamma,  t_1, \ldots, t_k)$ such that 
\begin{itemize}
\item $T>0$ and $\gamma \in C^\infty(\R/T \Z, L)$. 
\item $0 < t_1 < \cdots < t_k < T$. We set $t_0:= 0 = T \in \R/T\Z$. 
\item $\partial_t^m \gamma(t_j)=0$ for every $m \in \Z_{\ge 1}$ and $j \in \{0, \ldots, k\}$. 
\end{itemize}
For every $j \in \{0, \ldots, k\}$, we define $\evl_j: \mca{L}_{k+1}  \to L$ by 
\[ 
\evl_j (T, \gamma, t_1, \ldots, t_k) := \gamma(t_j). 
\]
$\evl_j$ will be abbreviated as $\ev_j$ when there is no risk of confusion. 

For $k \in \Z_{\ge 1}$, $k' \in \Z_{\ge 0}$ and $j \in \{1, \ldots, k\}$, 
we define a concatenation map 
\[ 
\con_j:   \mca{L}_{k+1} \fbp{\ev^{\mca{L}}_j}{\ev^{\mca{L}}_0} \mca{L}_{k'+1}  \to \mca{L}_{k+k'}
\] 
as follows. 
Notice that one can identify $\mca{L}_{k+1}$ with 
\[ 
\{ (\Gamma_0, \ldots, \Gamma_k) \in  \Pi^{k+1} \mid \ev_1(\Gamma_i) = \ev_0(\Gamma_{i+1}) \,(0 \le i \le k-1), \, \ev_1(\Gamma_k) = \ev_0(\Gamma_0)\}.
\]
Then we define $\con_j $ by 
\begin{align*} 
&\con_j ((\Gamma_0, \ldots, \Gamma_k), (\Gamma'_0, \ldots, \Gamma'_{k'})) \\
&:= \begin{cases} 
(\Gamma_0, \ldots, \Gamma_{j-2}, \Gamma_{j-1}*\Gamma'_0, \Gamma'_1, \ldots, \Gamma'_{k'-1}, \Gamma'_{k'}*\Gamma_j, \Gamma_{j+1}, \ldots, \Gamma_k) &(k' \ge 1) \\
(\Gamma_0, \ldots, \Gamma_{j-2}, \Gamma_{j-1}*\Gamma'_0*\Gamma_j, \Gamma_{j+1}, \ldots, \Gamma_k ) &(k' = 0).
\end{cases}
\end{align*} 

For every $a \in H_1(L: \Z)$, 
let $\mca{L}_{k+1}(a)$ denote the subset of $\mca{L}_{k+1}$ 
which consists of $(T, \gamma, t_1, \ldots, t_k)$ such that 
$[\gamma]=a$.
Obviously $\mca{L}_{k+1} = \bigsqcup_{a \in H_1(L: \Z)} \mca{L}_{k+1}(a)$, 
and the concatenation map $\con_j$ satisfies 
\[
\con_j  ( \mca{L}_{k+1}(a)  \fbp{\ev^{\mca{L}}_j}{\ev^{\mca{L}}_0} \mca{L}_{k'+1}(a')) \subset \mca{L}_{k+k'}(a+a'). 
\]

\subsection{Signs} 
Here we summarize some conventions on signs. 
Our sign conventions for direct/fiber products follow \cite{FOOO_09} Section 8.2, 
and for pushout of differntial forms we follow \cite{FOOO_Kuranishi} Section 7.1. 
Note that these conventions are different from those in \cite{Irie_17}. 

\textbf{Direct and fiber products of manifolds} 

Let $X_1$ and $X_2$ be oriented manifolds. 
Their direct product $X_1 \times X_2$ is oriented so that 
\[ 
T (X_1 \times X_2) \cong TX_1 \oplus TX_2
\] 
preserves orientations. 

Next we consider fiber product. 
Let $M$ be an oriented manifold and 
$\pi_i: X_i \to M\,(i=1,2)$ be $C^\infty$-maps. 
We assume that $\pi_2$ is a submersion. 
$\ker d\pi_2$ is oriented so that the isomorphism 
\[ 
T X_2 \cong TM \oplus \ker d \pi_2
\] 
preserves orientations. Then we orient $X_1 \fbp{\pi_1}{\pi_2} X_2$ so that
\[ 
T (X_1 \fbp{\pi_1}{\pi_2} X_2) \cong T X_1 \oplus \ker d\pi_2
\] 
preserves orientations. 

\textbf{Direct and fiber products of K-spaces} 

For later use we also fix sign conventions for direct and fiber products of K-spaces
(see Section 10 for basic notions in the theory of Kuranishi structures). 
Let $X_1$, $X_2$ be topological spaces with K-structures, 
and $\mca{U}_i = (U_i, \mca{E}_i, s_i, \psi_i)$ be a K-chart on $X_i$, for each $i=1, 2$. 

Then the direct product $\mca{U}_1 \times \mca{U}_2$, 
which is a K-chart of $X_1 \times X_2$, 
is oriented by 
\[ 
\mca{U}_1 \times \mca{U}_2 := (-1)^{\rk \mca{E}_2 (\dim U_1 - \rk \mca{E}_1)}  (U_1 \times U_2, \mca{E}_1 \times \mca{E}_2, s_1 \times s_2, \psi_1 \times \psi_2), 
\] 
where $U_1 \times U_2$ is oriented as before, and $\mca{E}_1 \times \mca{E}_2$
is oriented so that the isomorphism 
\[ 
(\mca{E}_1 \times \mca{E}_2)_{(x_1, x_2)} \cong (\mca{E}_1)_{x_1} \oplus (\mca{E}_2)_{x_2}  \qquad  (x_1 \in U_1, \, x_2 \in U_2)
\] 
preserves orientations. 

Next we consider the fiber product. 
Let $M$ be an oriented $C^\infty$-manifold and 
$\pi_i: X_i \to M\,(i=1,2)$ be strongly smooth maps. 
We assume that $\pi_2$ is weakly submersive. 
Then the fiber product 
$\mca{U}_1 \fbp{\pi_1}{\pi_2} \mca{U}_2$, 
which is a K-chart of 
$X \fbp{\pi_1}{\pi_2} X_2$, is oriented by 
\[ 
\mca{U}_1 \fbp{\pi_1}{\pi_2} \mca{U}_2:= (-1)^{\rk \mca{E}_2 (\dim U_1 - \dim M - \rk \mca{E}_1)} 
(U_1 \fbp{\pi_1}{\pi_2} U_2, \mca{E}_1 \times \mca{E}_2, s_1 \times s_2, \psi_1 \times \psi_2), 
\] 
where $U_1 \fbp{\pi_1}{\pi_2} U_2$ is oriented as before. 

\textbf{Pushout of differential forms} 

For any manifold $X$ and $j \in \Z$, 
let $\mca{A}^j(X)$ denote the space of degree $j$ differential forms on $X$, and 
$\mca{A}^j_c(X)$ denote its subspace which consists of compactly supported differential forms. 
We set $\mca{A}^j(X) = 0$ when $j<0$ or $j > \dim X$. 

Suppose $X$ and $Y$ are oriented manifolds 
and $\pi: X \to Y$ is a $C^\infty$-submersion. 
We define a pushout (or integration along fibers) 
\[ 
\pi_!: \mca{A}^*_c(X) \to \mca{A}^{*- \dim \pi}_c(Y)
\] 
(here $\dim \pi:= \dim X - \dim Y$) so that the formula 
\[ 
\int_Y \pi_! \omega \wedge \eta = \int_X \omega \wedge \pi^* \eta
\] 
holds for any $\omega \in \mca{A}^*_c(X)$ and $\eta \in \mca{A}^*(Y)$. 
Simple computations show 
\begin{align*} 
&d (\pi_! \omega) = (-1)^{\dim \pi} \pi_!(d\omega) \qquad\qquad ( \omega \in \mca{A}^*_c(X)),  \\
&\pi_! (\omega \wedge \pi^* \omega') = \pi_! \omega \wedge \omega' \qquad\qquad\,  ( \omega \in \mca{A}^*_c(X), \, \omega' \in \mca{A}^*(Y)). 
\end{align*}

\subsection{de Rham chain complex of $\mca{L}_{k+1}$}

Let us define the ``de Rham chain complex'' of $\mca{L}_{k+1}(a)$. 
First we need the following definition. 

\begin{defn}\label{171205_1} 
Let $U$ be a $C^\infty$-manifold and 
$\ph: U \to \mca{L}_{k+1}$. 
We set 
\[
\ph(u)= (T(u), \gamma(u), t_1(u), \ldots, t_k(u)). 
\]
We say that $\ph$ is of $C^\infty$, if the map 
\[ 
U \to \R^{k+1}; \, u \mapsto (T(u), t_1(u), \ldots, t_k(u))
\]
is of $C^\infty$ and 
\[
\{ (u,t) \mid u \in U, \, 0 \le t \le T(u) \} \to L; \, (u,t) \mapsto \gamma(u)(t)
\]
is of $C^\infty$, namely it extends to a $C^\infty$-map from an open neighborhood of the LHS in $U \times \R$ to $L$. 
We say that $\ph$ is \textit{smooth}, if $\ph$ is of $C^\infty$ and 
$\ev^{\mca{L}}_0 \circ \ph: U \to L$ is a submersion. 
\end{defn} 

For every $N \in \Z_{\ge 1}$, let $\mf{U}_N$ denote the set of  oriented submanifolds in $\R^N$, and let $\mf{U}:= \bigsqcup_{N \ge 1} \mf{U}_N$. 
Let $\mca{P}(\mca{L}_{k+1}(a))$ denote the set of pairs $(U, \ph)$ such that 
$U \in \mf{U}$ and $\ph: U \to \mca{L}_{k+1}(a)$ is a smooth map. 

For every $N \in \Z$, let us consider the vector space 
\begin{equation}\label{170619_1} 
\bigoplus_{(U,\ph) \in \mca{P}(\mca{L}_{k+1}(a))}  \mca{A}^{\dim U -N}_c(U). 
\end{equation} 
For any $(U,\ph) \in \mca{P}(\mca{L}_{k+1}(a))$ and $\omega \in \mca{A}^{\dim U- N}_c(U)$, 
let $(U, \ph, \omega)$ denote the vector in (\ref{170619_1})
such that its $(U,\ph)$-component is $\omega$ and the other components are $0$. 

Let $Z_N$ denote the subspace of (\ref{170619_1}) which is generated by 
\begin{align*} 
&\{ (U, \ph, \pi_! \omega) - (U', \ph \circ \pi, \omega) \mid (U, \ph) \in \mca{P}(\mca{L}_{k+1}(a)), \quad U' \in \mf{U}, \\
& \quad  \omega \in \mca{A}_c^{\dim U' -N} (U'), \quad \pi: U' \to U \, \text{is a $C^\infty$-submersion} \}. 
\end{align*} 
Then we define 
\[ 
C^\dR_N(\mca{L}_{k+1}(a)) :=  \biggl( \bigoplus_{(U,\ph) \in \mca{P}(\mca{L}_{k+1}(a))}  \mca{A}^{\dim U -N}_c(U) \biggr) / Z_N.
\] 
We often abbreviate $[(U, \ph, \omega)] \in C^\dR_N(\mca{L}_{k+1}(a))$ by $(U, \ph, \omega)$. 

We define a boundary operator 
$\partial: C^\dR_*(\mca{L}_{k+1}(a)) \to C^\dR_{*-1}(\mca{L}_{k+1}(a))$ by  
\begin{equation}\label{170828_1} 
\partial (U, \ph, \omega):= (-1)^{|\omega|+1} (U, \ph, d\omega).
\end{equation} 
It is easy to check that $\partial$ is well-defined and $\partial^2=0$. 
We call this chain complex the \textit{de Rham chain complex} of $\mca{L}_{k+1}(a)$, 
and denote its homology by $H^\dR_*(\mca{L}_{k+1}(a))$. 

\begin{rem} 
Here are slight differences between the presentation in this section and that in \cite{Irie_17}. 
\begin{itemize} 
\item 
The sign for the boundary operator in (\ref{170828_1}) is different from that in \cite{Irie_17}. 
\item 
In the definition of $\mca{P}(\mca{L}_{k+1}(a))$ we only require that 
$\ev^{\mca{L}}_0 \circ \ph: U \to L$ is a submersion. 
On the other hand, 
in Section 7.2 in \cite{Irie_17}, 
we consider maps $\ph: U \to \mca{L}_{k+1}$ such that 
$\ev^{\mca{L}}_j \circ \ph: U \to L$ are submersions for all $j \in \{0, \ldots, k\}$. 
However the resulting chain complexes are quasi-isomorphic. 
\end{itemize} 
\end{rem}

\begin{lem}\label{170619_2} 
\begin{enumerate} 
\item[(i):] The forgetting map 
\[ 
\mca{L}_{k+1}(a) \to \mca{L}_1(a); \qquad (T, \gamma, t_1, \ldots, t_k) \mapsto (T, \gamma) 
\] 
induces an isomorphism $H^\dR_*(\mca{L}_{k+1}(a)) \cong H^\dR_*(\mca{L}_1(a))$. 
In particular $H^\dR_*(\mca{L}_{k+1}(a))$ does not depend on $k$. 
\item[(ii):] $H^\dR_*(\mca{L}_1(a) ) \cong H^\sing_*(\mca{L}(a): \R)$, where the RHS denotes the singular homology with respect to the $C^\infty$-topology on $\mca{L}(a)$. 
\end{enumerate}
\end{lem} 
\begin{proof}
For each $k \in \Z_{\ge 0}$, let $\Delta^k$ denote the $k$-dimensional simplex: 
\[ 
\Delta^k:= \begin{cases} \R^0 &(k=0), \\ \{ (t_1,\ldots, t_k) \in \R^k \mid 0 \le t_1 \le \cdots \le t_k \le 1\} &(k \ge 1).  \end{cases} 
\] 
Then we define $\mca{P}(\mca{L}(a) \times \Delta^k)$ to be the set
which consists of $(U, \ph)$ such that 
$U \in \mf{U}$ and $\ph: U \to \mca{L}(a) \times \Delta^k$ is of $C^\infty$ (i.e. projections to each components are of $C^\infty$). 
Then we can define a chain complex $C^\dR_*(\mca{L}(a) \times \Delta^k)$ in exactly the same manner as $C^\dR_*(\mca{L}_{k+1}(a))$. 
Moreover, there exists a zig-zag of quasi-isomorphisms connecting $C^\dR_*(\mca{L}_{k+1}(a))$ and $C^\dR_*(\mca{L}(a) \times \Delta^k)$
(see the last part of Section 7.2 in \cite{Irie_17}, where $C^\dR_*(\mca{L}_{k+1}(a))$ is denoted by 
$C^\dR_*(\bar{\mca{L}}^a_{k, \text{reg}})$). 
Then, to prove (i) it is sufficient to show that the map 
\begin{equation}\label{170828_2} 
\mca{L}(a) \times \Delta^k \to \mca{L}(a) ; \qquad (\gamma, t_1, \ldots,t_k) \mapsto \gamma 
\end{equation} 
induces an isomorphism on $H^\dR_*$, 
which follows from homotopy invariance of $H^\dR_*$ (see Proposition 4.7 in \cite{Irie_17}). 
(ii) follows from $H^\dR_*(\mca{L}_1(a)) \cong H^\dR_*(\mca{L}(a))$ 
(apply the zig-zag mentioned above for $k=0$), 
and $H^\dR_*(\mca{L}(a)) \cong H^\sing_*(\mca{L}(a): \R)$, 
which follows from Theorem 6.1 in \cite{Irie_17}. 
\end{proof} 

Next we define the fiber product on de Rham chain complexes. 
For every $k \in \Z_{\ge 1}$, $k' \in \Z_{\ge 0}$, $j \in \{1,\ldots, k\}$ and $a, a' \in H_1(L:\Z)$, 
we define a linear map 
\begin{equation}\label{170619_3} 
\circ_j:  C^\dR_{n+d}(\mca{L}_{k+1}(a)) \otimes C^\dR_{n+d'}(\mca{L}_{k'+1}(a')) \to C^\dR_{n+d+d'}(\mca{L}_{k+k'}(a+a')) ; \, x \otimes y \mapsto x \circ_j y
\end{equation}
in the following way. 
Setting 
\[ 
x := (U, \ph, \omega), \qquad 
y:= (U', \ph', \omega')
\] 
let 
$\ph_j: = \evl_j \circ \ph$ and 
$\ph'_0:= \evl_0 \circ \ph'$. 
Then we define 
\[ 
x \circ_j y:= (-1)^{(\dim U - |\omega|-n)|\omega'|}  (U \fbp{\ph_j}{\ph'_0} U', \con_j  \circ (\ph_j \times \ph'_0), \omega \times \omega'), 
\] 
where 
the fiber product $U \fbp{\ph_j}{\ph'_0} U'$ is oriented as in Section 4.2, and 
$\con_j$ denotes the concatenation map defined in Section 4.1. 
It is straightforward to check that the fiber product (\ref{170619_3}) is a chain map, 
i.e. it satisfies the Leibniz rule 
\[
\partial (x \circ_j y) = \partial x \circ_j y + (-1)^d x \circ_j \partial y. 
\] 

It is also easy to check the associativity: given $x_i \in C^\dR_{n+d_i} (\mca{L}_{k_i+1}(a_i)) \,(i=1,2,3)$, 
there holds 
\begin{align*} 
 (x_1 \circ_{i_1} x_2) \circ_{k_2 + i_2 - 1} x_3 &= (-1)^{d_2d_3} (x_1 \circ_{i_2} x_3) \circ_{i_1} x_2 \quad\,   ( 1 \le i_1 < i_2 \le k_1), \\
 (x_1 \circ_{i_1} x_2) \circ_{i_1 + i_2 -1} x_3  &= x_1 \circ_{i_1} (x_2 \circ_{i_2} x_3) \qquad\qquad\qquad  (1 \le i_1 \le k_1,\, 1 \le i_2 \le k_2).
\end{align*} 

On homology level, 
the fiber product corresponds to the Chas-Sullivan loop product, 
which was originally defined in \cite{ChSu_99}: 

\begin{lem}\label{170623_1} 
The fiber product (\ref{170619_3}) induces a linear map
\[ 
H^\dR_{n+d} ( \mca{L}_{k+1}(a) ) \otimes H^\dR_{n+d'} (\mca{L}_{k'+1}(a') ) \to H^\dR_{n+d+d'}(\mca{L}_{k+k'} (a+a')). 
\] 
Via isomorphisms in Lemma \ref{170619_2}, this map corresponds to the loop product 
\[
H_{n+d}(\mca{L}(a)) \otimes H_{n+d'} (\mca{L}(a')) \to H_{n+d+d'}(\mca{L}(a+a')). 
\]
\end{lem} 
\begin{proof}
By Lemma \ref{170619_2} (i), it is sufficient to prove the case $k=1$ and $k'=0$, 
which follows from Proposition 8.7 in \cite{Irie_17}. 
\end{proof} 

\subsection{Chain model of $[-1, 1] \times \mca{L}_{k+1}$}

In this subsection, we define another chain complex $\bar{C}^\dR_*(\mca{L}_{k+1}(a))$. 
Roughly speaking, it consists of chains on $[-1,1] \times \mca{L}_{k+1}(a)$ relative to $\{-1,1\} \times \mca{L}_{k+1}(a)$. 
In Section 6, we use this chain complex to define ``gauge-equivalence'' of (approximate) solutions of the Maurer-Cartan equation of loop bracket. 

Let $\bar{\mca{P}}$ denote the set consists of tuple $(U, \ph, \tau_+, \tau_-)$ such that the following conditions are satisfied: 

\begin{itemize} 
\item $U \in \mf{U}$ and $\ph: U \to \R \times \mca{L}_{k+1}(a)$. We denote $\ph:= (\ph_\R, \ph_{\mca{L}})$, 
and for every interval $I \subset \R$ we denote $U_I:= (\ph_\R)^{-1}(I)$.
\item 
$\ph_\R$ and $\ph_{\mca{L}}$ are of $C^\infty$. Moreover, 
$U \to \R \times L; \, u \mapsto (\ph_\R(u), \ev_0 \circ \ph_{\mca{L}}(u))$ is a submersion. 

\item $\tau_+: U_{\ge 1} \to \R_{\ge 1} \times U_1$ is a diffeomorphism ($U_{\ge 1}$ is an abbreviation of $U_{\R_{\ge 1}}$) such that 
\[
\ph|_{U_{\ge 1}} = (i _{\ge 1} \times \ph_{\mca{L}}|_{U_1}) \circ \tau_+
\]
where $i_{\ge 1}: \R_{\ge 1} \to \R$ is the inclusion map. 
\item $\tau_-: U_{\le -1} \to \R_{\le -1} \times U_{-1}$ is a diffeomorphism $(U_{\le -1}$ is an abbreviation of $U_{\R_{\le -1}}$) such that 
\[
\ph|_{U_{\le -1}} = (i _{\le -1} \times \ph_{\mca{L}}|_{U_{-1}}) \circ \tau_-
\]
where $i_{\le -1}: \R_{\le -1} \to \R$ is the inclusion map. 
\end{itemize} 
\begin{rem}
$U_{\ge 1}$ and $U_{\le -1}$ may be the empty set. 
\end{rem} 

For any $(U, \ph, \tau_+, \tau_-) \in \bar{\mca{P}}$ and $N \in \Z$, 
let $\mca{A}^N(U, \ph, \tau_+, \tau_-)$ 
denote the vector space which consists of $\omega \in \mca{A}^N(U)$ satisfying the following conditions: 
\begin{itemize} 
\item $\omega|_{U_{[-1, 1]}}$ is compactly supported. 
\item $\omega|_{U_{\ge 1}} = (\tau_+)^* (1 \times  \omega|_{U_1})$. 
\item $\omega|_{U_{\le -1}} = (\tau_-)^* (1 \times \omega|_{U_{-1}})$. 
\end{itemize} 

For $N \in \Z$, let us define 
\[
\bar{C}^\dR_N (\mca{L}_{k+1}(a))  := \biggl( \bigoplus_{(U,\ph, \tau_+, \tau_-) \in \bar{\mca{P}}} \mca{A}^{\dim U-N-1} (U, \ph, \tau_+, \tau_-) \biggr)/Z_N
\] 
where $Z_N$ is a subspace generated by vectors 
\[ 
(U, \ph, \tau_+, \tau_-, \omega) - 
(U', \ph', \tau'_+, \tau'_-,  \omega')
\] 
such that there exists a submersion $\pi: U' \to U$ satisfying 
\begin{align*} 
\ph' &= \ph \circ \pi, \\ 
\omega &= \pi_! \omega', \\ 
\tau_+ \circ \pi|_{U'_{\ge 1}} &= (\id_{\R_{\ge 1}} \times \pi|_{U'_1}) \circ \tau'_+, \\ 
\tau_- \circ \pi|_{U'_{\le -1}} &= (\id_{\R_{\le -1}} \times \pi|_{U'_{-1}}) \circ \tau'_-.
\end{align*} 
Let us define $\partial: \bar{C}^\dR_*(\mca{L}_{k+1}(a)) \to \bar{C}^\dR_{*-1}(\mca{L}_{k+1}(a))$
by 
\[ 
\partial(U, \ph, \tau_+, \tau_-, \omega): = (-1)^{|\omega|+1} (U, \ph, \tau_+, \tau_-, d\omega).
\] 
It is easy to check that $\partial$ is well-defined and $\partial^2=0$, thus we obtain a chain complex. 

In the following argument of this subsection, 
we abbreviate $\bar{C}^\dR_*(\mca{L}_{k+1}(a))$ and $C^\dR_*(\mca{L}_{k+1}(a))$ by 
$\bar{C}_*$ and $C_*$, respectively. 
Let us define $i: C_* \to \bar{C}_*$ by
\[ 
i( U, \ph, \omega) := (-1)^{\dim U}  (\R \times U, \id_\R \times \ph, \tau_+, \tau_-, 1 \times \omega)
\] 
where $\tau_+$ and $\tau_-$ are defined in the obvious way. 
Also, we define 
$e_+: \bar{C}_* \to C_*$ and $e_-: \bar{C}_* \to C_*$ by 
\begin{align*} 
e_+ (U, \ph, \tau_+, \tau_-, \omega) &:= (-1)^{\dim U - 1} (U_1, \ph|_{U_1}, \omega|_{U_1}), \\ 
e_- (U, \ph, \tau_+, \tau_-, \omega) &:=  (-1)^{\dim U - 1}(U_{-1}, \ph|_{U_{-1}}, \omega|_{U_{-1}}), 
\end{align*} 
where $U_1$ (resp. $U_{-1}$) is oriented so that 
$\tau_+: U_{\ge 1} \to \R_{\ge 1} \times U_1$ 
(resp. $\tau_-: U_{\le -1} \to \R_{\le -1} \times U_{-1}$)
is orientation-preserving, where 
$\R_{\ge 1}$ (resp. $\R_{\le -1}$)
is oriented so that 
$\partial/\partial t$ 
is of positive direction 
($t$ denotes the standard coordinate on $\R$). 
Then it is easy to see that $i$, $e_+$, $e_-$ are well-defined chain maps, 
and there holds $e_+ \circ i  = e_- \circ i = \id_C$. 

\begin{lem}
$(e_+, e_-): \bar{C}_* \to C_* \oplus C_*$ is surjective.
\end{lem}
\begin{proof} 
Let us take $\chi \in C^\infty(\R, [0, 1])$ so that 
$\chi(t) = 1$ for every $t \ge 1$ and $\chi(t)=0$ for every $t \le -1$. 
Given $x = (U, \ph, \omega) \in C_*$, let 
\[ 
\bar{x} := (-1)^{\dim U} (\R \times U, \id_\R \times \ph, \tau_+, \tau_-, \chi \times \omega), 
\]
where $\tau_+$ and $\tau_-$ are defined in the obvious manner. 
Then $e_+(\bar{x}) = x$ and $e_-(\bar{x})=0$, thus we have proved 
that the image of $(e_+, e_-)$ contains $C_* \oplus 0$. 
A similar argument shows that the image contains $0 \oplus C_*$. 
This completes the proof. 
\end{proof}

\begin{lem} 
$i \circ e_+$ and $i \circ e_-$ are chain homotopic to $\id_{\bar{C}}$. 
\end{lem} 
\begin{proof}
We only prove that $i \circ e_+$ is chain homotopic to $\id_{\bar{C}}$
by explicitly defining a linear map $K: \bar{C}_* \to \bar{C}_{*+1}$ which satisfies 
\[ 
K \partial + \partial K =  \id_{\bar{C}} - i \circ e_+. 
\]
The proof for $e_-$ is completely parallel.

\textbf{Step 1.} 
Let us take $C^\infty$-functions $\alpha: \R^2 \to \R$ and $\chi: \R^2 \to [0,1]$
so that the following conditions are satisfied: 
\begin{itemize} 
\item $x \le 0 \implies \alpha(x,y)=y$. 
\item $x \ge 1, \, y \ge -1 \implies \alpha(x,y)=-x$. 
\item $\nabla \alpha(x,y) \ne 0$ for every $(x,y) \in \R^2$.
\item $\nu_\alpha:= \nabla \alpha/|\nabla \alpha|$ is proper. 
Namely, for any $p \in \R^2$, there exists $c: \R \to \R^2$ such that
$c(0)=p$ and $\dot{c}(t) = \nu_\alpha(c(t))$ for any $t \in \R$. 
\item $\{ \alpha \ge 1\} \subset \{ y \ge 1\}$. 
\item $\{ \alpha \le -1\} \subset \{y \le -1\} \cup \{ x \ge 1\}$. 
\item $\chi \equiv 1$ on a neighborhood of $\{x=0\} \cup \{x \ge 0, y \le 1\}$. 
\item $d\chi(\nabla \alpha)=0$ on $\{ \alpha \ge 1\}  \cup \{ \alpha \le -1\}$. 
\item $\supp \chi \cap \{ -1 \le \alpha \le 1\}$ is compact. 
\end{itemize}

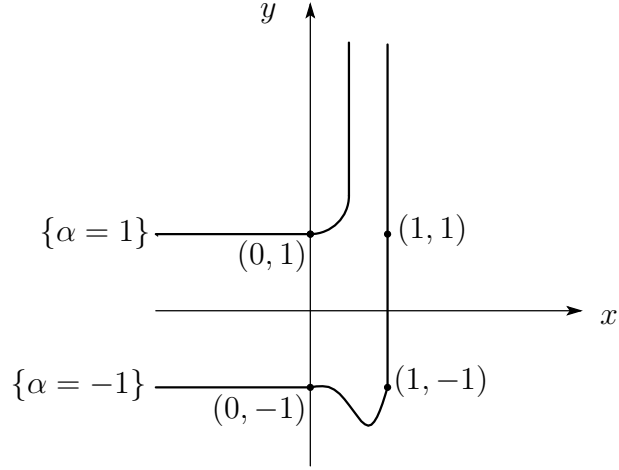
\begin{center}
\input{alpha.tpc}
\end{center} 

\textbf{Step 2.} 
We define a linear map $K: \bar{C}_* \to \bar{C}_{*+1}$ by 
\[ 
K (U, \ph, \tau_+, \tau_-, \omega):= (-1)^{|\omega|+1} (\R \times U, \bar{\ph}, \bar{\tau}_+, \bar{\tau}_-, \bar{\omega})
\] 
where $\bar{\ph}$, $\bar{\tau}_{\pm}$ and $\bar{\omega}$ are defined as follows. 
$\bar{C}_*$ is defined by taking a quotient, however well-definedness of $K$ is easy to check. 

\begin{itemize}
\item Let us denote $\ph: U \to \R \times \mca{L}_{k+1}(a)$ by $\ph = (\ph_\R, \ph_{\mca{L}})$. 
We define 
$\tilde{\alpha}: \R \times U \to \R$ and $\bar{\ph}: \R \times U \to \R \times \mca{L}_{k+1}(a)$ by 
\[ 
\tilde{\alpha}(r,u):=\alpha(r, \ph_\R(u)), \qquad 
\bar{\ph}(r,u):= (\tilde{\alpha}(r,u), \ph_{\mca{L}}(u)). 
\]
We have to check that
\[ 
\R \times U \to \R \times L; \quad (r,u) \mapsto (\tilde{\alpha}(r,u) , \ev_0 \circ \ph_{\mca{L}}(u))
\]
is a submersion. 
This is because 
$U \to \R \times L; \, u \mapsto (\ph_\R(u), \ev_0 \circ \ph_{\mca{L}}(u))$ is a submersion, 
and $d\alpha(x,y) \ne 0$ for every $(x,y) \in \R^2$. 

\item To define $\bar{\tau}_+$ and $\bar{\tau}_-$, 
we first define a vector field
\[
V \in \mf{X}(\R \times (U_{\ge 1} \cup U_{\le -1}) \cup \R_{\ge 1} \times U_{[-1,1]})
\]
as follows (recall $\nu_\alpha:= \nabla \alpha/|\nabla \alpha|$ from Step 1): 
\begin{itemize}
\item On $\R \times U_{\ge 1} \cong \R \times \R_{\ge 1} \times U_1$, we set $V(x,y,u):= (\nu_\alpha(x,y), 0)$.
\item On $\R \times U_{\le -1} \cong \R \times \R_{\le -1} \times U_{-1}$, we set  $V(x,y,u):= (\nu_\alpha(x,y), 0)$. 
\item On $\R_{\ge 1} \times U_{[-1,1]}$, we set $V(x,u) := (-1,0)$.
\end{itemize} 
In particular, $V$ is defined on $\{ \tilde{\alpha} \ge 1\} \cup \{ \tilde{\alpha} \le -1\}$, and 
$V$ is forward (resp. backward) complete on $\{ \tilde{\alpha} \ge 1\}$ (resp. $\{ \tilde{\alpha} \le -1\}$). 
Then we define a diffeomorphism $\bar{\tau}_+: \{ \tilde{\alpha} \ge 1\} \cong \R_{\ge 1} \times \{\tilde{\alpha} = 1\}$ by 
$\bar{\tau}_+(p) := (\tilde{\alpha}(p), q)$, where $q$ is a unique point in $\{\tilde{\alpha}=1\}$ which is connected to $p$ by 
an integral curve of $V$. 
Similarly, we define a diffeomorphism $\bar{\tau}_-: \{ \tilde{\alpha} \le -1\} \cong \R_{\le -1} \times \{\tilde{\alpha} = -1\}$ by 
$\bar{\tau}_-(p) := (\tilde{\alpha}(p), q)$, where $q$ is a unique point in $\{\tilde{\alpha}= -1\}$ which is connected to $p$ by 
an integral curve of $V$. 
\item We define $\bar{\omega}$ by 
$\bar{\omega}(r,u): = \chi(r, \ph_\R(u)) \cdot  (\text{pr}_U)^* \omega$. 
\end{itemize}

\textbf{Step 3.}
We prove that $K \partial + \partial K = \id_{\bar{C}} - i \circ e_+$. 
It is easy to see that 
\[ 
(K \partial + \partial K) (U, \ph, \tau_+, \tau_-, \omega) 
= 
(\R \times U, \bar{\ph}, \bar{\tau}_+, \bar{\tau}_-, d \chi (r, \ph_\R(u)) \wedge (\pr_U)^*\omega). 
\] 
Let us denote $d\chi:= (d\chi)_+ + (d\chi)_-$ where 
$(d\chi)_+$ is supported on $\{x>0, y>1\}$ and $(d\chi)_-$ is supported on $\{x<0\}$. 
Then 
\begin{equation}\label{170622_2} 
(\R \times U, \bar{\ph}, \bar{\tau}_+, \bar{\tau}_-, (d\chi)_+ \wedge (\text{pr}_U)^* \omega) = -  (i \circ e_+)( U, \ph, \tau_+, \tau_-, \omega). 
\end{equation}
To check (\ref{170622_2}), let us consider a submersion 
\[ 
\R_{>0} \times U_{>1} \cong \R_{>0} \times (\R_{>1} \times U_1) \to \R \times U_1
\]
where the first map (diffeomorphism) is $\id_{\R_{>0}} \times \tau_+|_{U_{>1}}$, and the second map is $(x,y,u) \mapsto (\alpha(x,y), u)$. 
Then $\pi: \R_{>0} \times \R_{>1} \to \R; \, (x,y) \mapsto \alpha(x,y)$ satisfies $\pi_!((d\chi)_+) = -1$, 
and this shows (\ref{170622_2}). 
On the other hand
\[ 
(\R \times U, \bar{\ph}, \bar{\tau}_+, \bar{\tau}_-, (d\chi)_- \wedge  \text{pr}_U^* \omega) = (U, \ph, \tau_+, \tau_-, \omega)
\]
can be proved by the projection $\R_{<0} \times U \to U$, 
which pushes $(d\chi)_-$ to $1$. 
\end{proof} 

Let us define the fiber product on $\bar{C}_*$. 
For every $k \in \Z_{\ge 1}$, $k' \in \Z_{\ge 0}$, $j \in \{1, \ldots, k\}$ and $a, a' \in H_1(L:\Z)$, we define 
\begin{equation}\label{170928_1} 
\circ_j: \bar{C}^\dR_{n+d}(\mca{L}_{k+1}(a)) \otimes \bar{C}^\dR_{n+d'}(\mca{L}_{k'+1}(a')) \to \bar{C}^\dR_{n+d+d'}(\mca{L}_{k+k'}(a+a')); \, 
 x \otimes y \mapsto x \circ_j y
\end{equation}
in the following way. Setting 
\[
x = (U, \ph, \tau_+, \tau_-, \omega), \qquad
y = (U', \ph', \tau'_+, \tau'_-, \omega'), 
\] 
let us define $C^\infty$-maps 
$\ph_j, \, \ph'_0: U \to \R \times L$ by 
\[ 
\ph_j: = (\ph_\R, \ev^{\mca{L}}_j \circ \ph_{\mca{L}}), \qquad
\ph'_0:= (\ph'_\R, \ev^{\mca{L}}_0 \circ \ph'_{\mca{L}}). 
\]
Note that $\ph'_0$ is a submersion, thus the fiber product
$U \fbp{\ph_j}{\ph'_0} U'$ is a $C^\infty$-manifold. 
Then we define 
\[
x \circ_j y:= (-1)^{(\dim U - |\omega| - n-1) |\omega'| + n}
(U \fbp{\ph_j}{\ph'_0} U', \ph'', \tau''_+, \tau''_-, \omega \times \omega')
\]
where $\ph''$ is defined by 
\[ 
\ph''(u, u') := (\ph_\R(u), \con_j(\ph_{\mca{L}}(u), \ph'_{\mca{L}}(u'))), 
\] 
and $\tau''_+$, $\tau''_-$ are defined as follows: 
\begin{align*} 
\rho_+(u,u')&:= \pr_{\R_{\ge 1}} \circ \tau_+(u) = \pr_{\R_{\ge 1}} \circ \tau'_+(u'), \\
\tau''_+(u, u')&:= (\rho_+(u, u'),  ( \pr_{U_1} \circ \tau_+(u), \pr_{U'_1} \circ \tau'_+(u'))), \\
\rho_-(u,u')&:=\pr_{\R_{\le -1}} \circ \tau_-(u) = \pr_{\R_{\le -1}} \circ \tau'_-(u'), \\ 
\tau''_-(u, u')&:= (\rho_-(u,u') , ( \pr_{U_{-1}} \circ \tau_-(u), \pr_{U'_{-1}} \circ \tau'_-(u'))). 
\end{align*} 
It is easy to see that the fiber product (\ref{170928_1}) is a chain map, 
i.e. it satisfies the Leibniz rule: 
\[ 
\partial (x \circ_j y) = \partial x \circ_j y + (-1)^d x \circ_j \partial y. 
\] 
Moreover, chain maps $i$, $e_+$, $e_-$ intertwine fiber products on $C_*$ and $\bar{C}_*$.  

\subsection{dg Lie algebras $C^{\mca{L}}$, $\bar{C}^{\mca{L}}$ and their completions}

For every $a \in H_1(L:\Z)$ and $k \in \Z_{\ge 0}$, we define
\[ 
C^{\mca{L}}(a, k)_*:= C^\dR_{*+n+\mu(a)+k-1} (\mca{L}_{k+1}(a))
\] 
and set 
\[ 
C^{\mca{L}}_*: = \bigoplus_{\substack{a \in H_1(L: \Z) \\ k \in \Z_{\ge 0}}} C^{\mca{L}}(a, k)_*. 
\]
To define the action filatration on $C^{\mca{L}}_*$, 
for each $E \in \R$ we set 
\[
F^E C^{\mca{L}}_* := \bigoplus_{\substack{\omega_n(\bar{a}) > E \\ k \in \Z_{\ge 0}}} C^{\mca{L}}(a, k)_* 
\]
and take completion: 
\[
\wh{C}^{\mca{L}}_*:= \varprojlim_{E \to \infty} C^{\mca{L}}_* / F^E C^{\mca{L}}_*. 
\]
$C^{\mca{L}}_*$ has a dg Lie algebra structure
(see Remark \ref{170829_1} for our convention) 
defined as follows: 
\begin{align*} 
(\partial x)(a,k) &:= \partial  (x(a,k)),  \\ 
(x \circ y)(a,k) &:= \sum_{\substack{k'+k'' = k+1 \\ 1 \le i \le k' \\ a'+a''=a}}(-1)^{(i-1)(k''-1)+(k'-1)(|y|+1+k'')}
x(a', k') \circ_i  y(a'', k''),  \\
[x, y]&:= x \circ y - (-1)^{|x||y|}  y \circ x. 
\end{align*} 
For the sign in the RHS of the second formula, see Theorem 2.8 (ii) in \cite{Irie_17}. 
The Jacobi identity follows from the associativity of the fiber product.
Since this dg Lie algebra structure respects the decomposition $C^{\mca{L}}_* \cong \bigoplus_{(a,k) \in H_1(L: \Z) \times \Z_{\ge 0} } C^{\mca{L}}(a,k)_*$, 
it extends to the completion $\wh{C}^{\mca{L}}_*$. 

We also consider $\bar{C}^{\mca{L}}_* := \bigoplus_{(a,k) \in H_1(L: \Z) \times \Z_{\ge 0}}  \bar{C}^\dR_{*+n+\mu(a)+k-1} (\mca{L}_{k+1}(a))$
and its completion $\wh{\bar{C}}^{\mca{L}}_*$. 
One can naturally define a dg Lie algebra structure on $\bar{C}^{\mca{L}}_*$, 
$C^{\mca{L}}_*$, and it extends to the completion. 
One can define morphisms of dg Lie algebras 
\[ 
i: C ^{\mca{L}}_* \to \bar{C}^{\mca{L}}_*, \quad
e_+: \bar{C}^{\mca{L}}_* \to C^{\mca{L}}_*, \quad
e_-: \bar{C}^{\mca{L}}_* \to C^{\mca{L}}_*, 
\] 
so that the following properties hold: 
\begin{itemize} 
\item $e_+ \circ i = e_- \circ i = \id_{C^{\mca{L}}}$.
\item $i \circ e_+$ and $i \circ e_-$ are chain homotopic to $\id_{\bar{C}^{\mca{L}}}$.
One can take chain homotopies to respect decompositions over $H_1(L: \Z) \times \Z_{\ge 0}$. 
\item $(e_+, e_-): \bar{C}^{\mca{L}}_* \to C^{\mca{L}}_* \oplus  C^{\mca{L}}_*$ is surjective. 
\end{itemize} 

\section{Chain level statement} 

The goal of this section is to reduce Theorem \ref{161011_1} to Theorem \ref{161214_2}, 
which is formulated on the dg Lie algebra $\wh{C}^{\mca{L}}$. 
Throughout this section, $\R$-coefficient singular homology $H^\sing_*(\, \cdot \, : \R)$ is abbreviated as $H_*(\, \cdot \,)$. 

\begin{thm}\label{161214_2} 
There exist $x \in \wh{C}^{\mca{L}}_{-1}$, $y \in \wh{C}^{\mca{L}}_2$, $z \in \wh{C}^{\mca{L}}_1$ 
and $\ep>0$ such that: 
\begin{enumerate} 
\item[(i):] $\partial x  - \frac{1}{2} [x,x] = 0$, namely $x$ is a Maurer-Cartan element of $\wh{C}^{\mca{L}}$. 
\item[(ii):] $\partial y - [x,y] = z$. 
\item[(iii):] $x(a,k) \ne 0$ only if $\omega(\bar{a}) \ge 2\ep$ or $a=0$, $k \ge 2$. 
Moreover, $x(0, 2)  \in C^\dR_n(\mca{L}_3(0) )$ is a cycle such that 
$[x(0, 2) ] \in H^\dR_n(\mca{L}_3(0) ) \cong H_n(\mca{L}(0))$ corresponds to $(-1)^{n+1} [L]$. 
\item[(iv):] 
$z(a, k) \ne 0$ only if $\omega(\bar{a}) \ge 2\ep $ or $a=0$. 
Moreover, $z(0, 0) \in C^\dR_n(\mca{L}_1(0) )$ is a cycle such that 
$[z(0, 0) ] \in H^\dR_n(\mca{L}_1(0) ) \cong H_n(\mca{L}(0))$ corresponds to $(-1)^{n+1} [L]$. 
\end{enumerate}
In (iii) and (iv), $[L] \in H_n(\mca{L}(0))$ is defined as in Theorem \ref{161011_1}. 
\end{thm} 

Conditions (i) and (iii) imply that 
\[ 
x^0:= \sum_{k \ge 2} x(0, k) \in \wh{C}^{\mca{L}}_{-1} 
\] 
is a Maurer-Cartan element of $\wh{C}^{\mca{L}}$. 
For every $a \in H_1(L: \Z)$, let 
$C^{\mca{L}}(a)_*:= \prod_{k \ge 0} C^\mca{L}(a,k)_*$. 
We define 
$\partial_{x^0}: C^{\mca{L}}(a)_* \to C^{\mca{L}}(a)_{*-1}$ by 
$\partial_{x^0}(u) := \partial u - [x^0 , u]$. 

\begin{lem}\label{161225_1} 
$H_*(C^{\mca{L}}(a)_*, \partial_{x^0}) \cong H_{*+n+\mu(a)-1}(\mca{L}(a))$.
\end{lem}
\begin{proof}
Let us consider the exact sequence 
\[ 
0 \to \prod_{k \ge 1} C^{\mca{L}}(a,k)_* \to C^{\mca{L}}(a)_* \to C^{\mca{L}}(a,0)_* \to 0 
\] 
where the first map is inclusion and the second map is projection. 
Since $H_*(C^{\mca{L}}(a,0)) \cong H_{*+n+\mu(a)-1}(\mca{L}(a))$ by Lemma \ref{170619_2}, 
it is sufficient to show that $\prod_{k \ge 1} C^{\mca{L}}(a,k)_*$ is acyclic
with the boundary operator $\partial_{x^0}$. 

For every positive integer $N$, 
the operator $\partial_{x^0}$ preserves $\prod_{k > 2N} C^{\mca{L}}(a,k)_*$. 
Thus $\partial_{x^0}$ acts on $\prod_{1 \le k \le 2N} C^{\mca{L}}(a,k)_*$, 
preserving the filtration 
$\biggl( \prod_{i \le k \le 2N} C^{\mca{L}}(a,k)_* \biggr)_{1 \le i \le 2N}$. 

Then the $E_1$-term is 
\[ 
H_*(C^{\mca{L}}(a,k)) \cong H^\dR_{*+n+\mu(a)+k-1}(\mca{L}_{k+1}(a)) \cong H_{*+n+\mu(a)+k-1}(\mca{L}(a)). 
\] 
Let us compute $d_1: H_*(C^{\mca{L}}(a,k)) \to H_{*-1}(C^{\mca{L}}(a,k+1))$. 
By direct computations 
\begin{align*} 
&[x^0,y](a, k+1) = \\ 
&\qquad  (-1)^{|y|} \biggl (x(0,2) \circ_2 y(a,k) + \sum_{1 \le i \le k} (-1)^i y(a,k) \circ_i x(0,2) + (-1)^{k+1} x(0,2) \circ_1 y(a,k) \biggr). 
\end{align*} 
Since $[x(0,2)] \in H^\dR_n(\mca{L}_3(0))$ corresponds to $(-1)^{n+1} [L] \in H_n(\mca{L})$
and the fiber product on homology level corresponds to the loop product (Lemma \ref{170623_1}), 
$d_1$ coincides with $\pm \sum_{0 \le i \le k+1} (-1)^i$. 
Then all $E_2$-terms vanish, thus $\prod_{1 \le k \le 2N} C^{\mca{L}}(a,k)_*$ is acyclic. 
Finally, by taking an inverse limit (Theorem 3.5.8 in \cite{Weibel_94}) we have shown that 
$\prod_{k \ge 1} C^{\mca{L}}(a,k)_*$ is acyclic. 
\end{proof} 

\begin{lem} 
$x^+:= x - x^0$ satisfies equations 
\[ 
\partial_{x^0} x^+ - \frac{1}{2}[x^+, x^+] = 0, \qquad \partial_{x^0} y - [x^+, y] = z.
\]
\end{lem}
\begin{proof}
Straightforward from $\partial x - \frac{1}{2}[x,x]=0$, $\partial y - [x,y]=z$ and $x= x^0 + x^+$. 
\end{proof} 

Let $H^{\mca{L}}_* = \bigoplus_{a \in H_1(L: \Z)} H_{*+n+\mu(a)-1}(\mca{L}(a))$
as in Section 3. Then there exist linear maps 
\[ 
\iota: H^{\mca{L}}_* \to C^{\mca{L}}_*, \quad
\pi: C^{\mca{L}}_* \to H^{\mca{L}}_*, \quad
\kappa: C^{\mca{L}}_* \to C^{\mca{L}}_{*+1}
\]
such that 
\[
\partial_{x^0} \circ \iota = 0, \quad
\pi \circ \partial_{x^0} = 0, \quad
\pi \circ \iota = \id_{H^{\mca{L}}}, \quad
\kappa \circ \partial_{x^0}  + \partial_{x^0} \circ  \kappa = \id_{C^{\mca{L}}} - \iota \circ \pi.
\]
One can take these maps so that they preserve decompositions of $H^{\mca{L}}$ and $C^{\mca{L}}$ 
over $H_1(L: \Z)$, thus they extend to the completions. 
Also, one can take $\pi$ so that it maps $\sum_{k \ge 0} z(0,k)$ to $(-1)^{n+1} [L] \in H_n(\mca{L}(0))$.  

Using homotopy transfer theorem (Thoerem \ref{170623_2}), 
one can define an $L_\infty$-structure $l^H$ on $H^{\mca{L}}_*$ 
and an $L_\infty$-homomorphism
$p$ from $(C^{\mca{L}}_*, \partial_{x^0}, [\, , \,])$ to 
$(H^{\mca{L}}_*, l^H)$ so that $l^H_1=0$ and $p_1 = \pi$. 
We can take $l^H$ and $p$ so that they respect decompositions over $H_1(L:\Z)$, 
thus they extend to the completions. 
Applying Proposition 4.9 in \cite{Latschev_15}, 
elements in $\wh{H}^{\mca{L}}$ 
\begin{align*} 
X&:= \sum_{k=1}^\infty \frac{1}{k!} p_k(x^+, \ldots, x^+) \in  \wh{H}^{\mca{L}}_{-1} \\
Y&:= \sum_{k=1}^\infty \frac{1}{(k-1)!} p_k(y, x^+, \ldots, x^+) \in \wh{H}^{\mca{L}}_2, \\
Z&:= \sum_{k=1}^\infty  \frac{1}{(k-1)!} p_k(z, x^+, \ldots, x^+) \in \wh{H}^{\mca{L}}_1, 
\end{align*} 
satisfy 
\begin{align*} 
&\sum_{k=2}^\infty \frac{1}{k!} l^H_k(X, \ldots, X) = 0, \\
&\sum_{k=2}^\infty \frac{1}{(k-1)!} l^H_k(Y, X, \ldots, X) = Z. 
\end{align*} 
Notice that infinite sums in the definitions of $X$, $Y$, $Z$ make sense
since $x^+(a) \ne 0$ only if $\omega(\bar{a}) \ge 2\ep$ and 
$p$ preserves decompositions over $H_1(L: \Z)$. 
Moreover $X(a) \ne 0$ only if $\omega(\bar{a}) \ge 2\ep$, 
thus Theorem \ref{161011_1} (iii) holds with $c:= 2\ep$. 
To complete the proof of Theorem \ref{161011_1}, 
we need to show that $Z(0)  = (-1)^{n+1}  [L]$.
Since $p$ respects decompositions over $H_1(L: \Z)$ 
and $z(a,k) \ne 0$ only if $\omega(\bar{a}) \ge 2\ep$ or $a=0$, we obtain 
\[ 
Z(0) = \pi \biggl( \sum_{k=0}^\infty z(0,k) \biggr) = (-1)^{n+1} [L]. 
\]

\section{Sequence of approximate solutions} 

To prove Theorem \ref{161214_2}, we have to define chains $x$, $y$, $z$ satisfying the equations
\begin{equation}\label{170830_1} 
\partial x - \frac{1}{2} [x,x] = 0, \qquad \partial y - [x,y] = z.
\end{equation} 
These chains are defined from virtual fundamental chains of moduli spaces of pseudo-holomorphic disks with boundaries on $L$. 
However, it is difficult to get such chains in \textit{one step} since it will involve simultaneous perturbations of Kuranishi maps of infinitely many moduli spaces. 
Due to this difficulty, we first define a sequence $(x_i, y_i, z_i)_i$ such that the following conditions hold for every $i$: 
\begin{itemize}
\item The tuple $(x_i, y_i, z_i)$ satisfies equations (\ref{170830_1}) up to certain energy level which goes to $\infty$ as $i \to \infty$. 
\item Tuples $(x_i, y_i, z_i)$ and $(x_{i+1}, y_{i+1}, z_{i+1})$ are ``gauge equivalent'' up to certain energy level which  goes to $\infty$ as $i \to \infty$. 
\end{itemize} 
By the second condition, limits $x:= \lim_{i \to \infty} x_i$, $y:= \lim_{i \to \infty} y_i$, $z:= \lim_{i \to \infty} z_i$ exist, 
and by the first condition these limits satisfy (\ref{170830_1}). 
This procedure is similar to the construction of the $A_\infty$-structure in Lagrangian Floer theory as a limit of $A_K$-structures; 
see Section 7.2.3 in \cite{FOOO_09}, and Remark 22.27 in \cite{FOOO_Kuranishi}. 

The goal of this section is to explain details of the algebraic procedure sketched above; 
we reduce Theorem \ref{161214_2} to Theorem \ref{161215_1}, 
which asserts the existence of a sequence of  approximate solutions connected by gauge equivalences. 
The proof of Theorem \ref{161215_1} involves construction of virtual fundamental chains using the theory of Kuranishi structures, 
which is carried out in Sections 7--9. 

Let $J$ be the standard complex structure on $\C^n$, 
and take $\ep>0$ so that 
$2 \ep$ is less than the minimal symplectic area of nonconstant $J$-holomorphic disks with boundaries on $L$. 
We fix such $\ep$ in the following arguments. 

For each $m \in \Z$, we set 
\[ 
F^m C^{\mca{L}}_*:= \bigoplus_{\substack{ a \in H_1(L:\Z) \\ k \in \Z_{\ge 0} \\ \omega_n(\bar{a}) \ge \ep(m+1-k)}} C^{\mca{L}}(a,k)_*. 
\] 
Obviously $\partial F^m \subset F^m$. 
We can also show $[F^m, F^{m'}] \subset F^{m+m'}$ since 
\[ 
\omega_n(\bar{a}) \ge \ep (m+1-k) , \, 
\omega_n(\bar{a'}) \ge \ep (m'+1-k') \implies 
\omega_n(\bar{a} + \bar{a'}) \ge \ep (m+ m' + 1 - (k+k'-1)). 
\] 
The filtration on $\bar{C}^{\mca{L}}$ (see Section 4.5) 
is defined in a similar manner. 
These filtrations extend to the completions. 
In the following we abbreviate $C^{\mca{L}}$ and $\bar{C}^{\mca{L}}$ by $C$ and $\bar{C}$. 

\begin{thm}\label{161215_1} 
There exist integers $I, U \ge 2$ 
and a sequence $(x_i, y_i, z_i, \bar{x}_i, \bar{y}_i, \bar{z}_i)_{i \ge I}$
satisfying the following conditions for every $i \ge I$: 
\[ 
x_i \in F^1 C_{-1}, \, 
\bar{x}_i \in F^1 \bar{C}_{-1}, \, 
y_i \in F^{-U} C_2, \, 
\bar{y}_i \in F^{-U} \bar{C}_2, \, 
z_i \in F^{-1} C_1, \, 
\bar{z}_i \in F^{-1} \bar{C}_1.
\]
\[ 
x_i = e_-(\bar{x}_i), \quad
y_i = e_-(\bar{y}_i), \quad
z_i = e_-(\bar{z}_i). 
\] 
\[ 
\partial \bar{x}_i - \frac{1}{2} [\bar{x}_i, \bar{x}_i] \in F^i \bar{C}_{-2}, \quad
\partial \bar{y}_i - [\bar{x}_i, \bar{y}_i] - \bar{z}_i \in F^{i-U-1}\bar{C}_1, \quad
\partial \bar{z}_i - [\bar{x}_i, \bar{z}_i] \in F^{i-2} \bar{C}_0. 
\] 
\[ 
x_{i+1} - e_+(\bar{x}_i) \in F^iC_{-1}, \quad 
y_{i+1} - e_+(\bar{y}_i) \in F^{i-U-1}C_2, \quad 
z_{i+1} - e_+(\bar{z}_i) \in F^{i-2}C_1. 
\] 
\begin{itemize} 
\item 
$x_i(a,k) \ne 0$ only if $\omega(\bar{a}) \ge 2\ep$ or $a=0$, $k \ge 2$. 
Moreover, 
$x_i(0,2)$ is a cycle in $C^\dR_n(\mca{L}_3(0))$ 
such that $[x_i(0,2)]$ corresponds to $(-1)^{n+1}  [L]$
via the isomorphism $H^\dR_n(\mca{L}_3(0)) \cong H_n(\mca{L})$. 
\item
$z_i(a,k) \ne 0$ only if $\omega(\bar{a}) \ge 2\ep$ or $a=0$. 
Moreover, 
$z_i(0,0)$ is a cycle in $C^\dR_n(\mca{L}_1(0))$ 
such that $[z_i(0,0)]$ corresponds to $(-1)^{n+1} [L]$ 
via the isomorphism $H^\dR_n(\mca{L}_1(0)) \cong H_n(\mca{L})$. 
\end{itemize} 
\end{thm} 

\begin{rem} 
$\partial z - [x,z]=0$ follows from $\partial x - \frac{1}{2}[x,x]=0$ and $\partial y - [x,y]=z$, 
thus the assumption $\partial \bar{z}_i - [\bar{x}_i, \bar{z}_i] \in F^{i-2}$
in Theorem \ref{161215_1} may seem redundant. 
However, this assumption is necessary to carry out the induction argument. 
\end{rem} 

The rest of this section is devoted to 
the proof of Theorem \ref{161214_2} assuming Theorem \ref{161215_1}. 
We first need the following elementary lemma. 

\begin{lem}\label{161221_1} 
Let $V_*$, $W_*$ be chain complexes and $e: V_* \to W_*$ be a surjective quasi-isomorphism.
For any $x \in V_*$ and $y \in W_{*+1}$ such that 
\[ 
\partial x = 0, \qquad e(x)= \partial y 
\] 
there exists $\bar{y} \in V_{*+1}$ such that $e(\bar{y}) = y$ and $\partial \bar{y}=x$. 
\end{lem} 
\begin{proof} 
Since $e$ is surjective, there exists $z \in V_{*+1}$ such that $e(z)=y$. 
Then $e(x - \partial z) = 0$ and $\partial (x- \partial z) = 0$. 
Thus $x - \partial z$ is a cycle in $\Ker e$. 
Since $\Ker e$ is acyclic, there exists $w \in \Ker e$ such that 
$\partial w = x - \partial z$. 
Then, $\bar{y}:= z+w$ satisfies the condition of the lemma. 
\end{proof}

Key arguments are summarized in Lemma \ref{161216_1} below. 

\begin{lem}\label{161216_1} 
Let $I$, $U$ be as in Theorem \ref{161215_1}. 
There exists a sequence 
\[ 
(x_{i,j}, y_{i,j}, z_{i,j}, \bar{x}_{i,j}, \bar{y}_{i,j}, \bar{z}_{i,j})_{i \ge I, \, j \ge 0}
\] 
satisfying the following conditions for every $i \ge I$ and $j \ge 0$: 
\begin{equation}\label{170627_1} 
x_{i,0} = x_i, \quad
y_{i,0} = y_i, \quad
z_{i,0} = z_i, \quad
\bar{x}_{i,0} = \bar{x}_i, \quad
\bar{y}_{i,0} = \bar{y}_i, \quad 
\bar{z}_{i,0} = \bar{z}_i. 
\end{equation}
\begin{eqnarray}
&x_{i,j} \in F^1 C_{-1}, \quad \bar{x}_{i,j} \in F^1 \bar{C}_{-1}, \quad y_{i,j} \in F^{-U} C_2,  \\
&\bar{y}_{i,j} \in F^{-U} \bar{C}_2, \quad z_{i,j} \in F^{-1} C_1, \quad  \bar{z}_{i,j} \in F^{-1} \bar{C}_1. \nonumber
\end{eqnarray} 
\begin{equation}\label{170627_2} 
x_{i,j} = e_-(\bar{x}_{i,j}), \quad
y_{i,j} = e_-(\bar{y}_{i,j}), \quad
z_{i,j} = e_-(\bar{z}_{i,j}). 
\end{equation}
\begin{eqnarray}\label{170627_3} 
&\partial \bar{x}_{i,j} - \frac{1}{2} [ \bar{x}_{i,j}, \bar{x}_{i,j} ] \in F^{i+j}\bar{C}_{-2}, \qquad \partial \bar{y}_{i,j} - [ \bar{x}_{i,j} , \bar{y}_{i,j} ] - \bar{z}_{i,j} \in F^{i+j-U-1}\bar{C}_1. \\
&\partial \bar{z}_{i,j} - [ \bar{x}_{i,j} , \bar{z}_{i,j}] \in F^{i+j-2}\bar{C}_0.   \nonumber
\end{eqnarray} 
\begin{eqnarray}
&x_{i+1, j} - e_+(\bar{x}_{i,j}) \in F^{i+j}C_{-1}, \qquad y_{i+1, j} - e_+(\bar{y}_{i,j}) \in F^{i+j-U-1}C_2,  \\
&z_{i+1, j} - e_+(\bar{z}_{i,j}) \in F^{i+j-2}C_1. \nonumber
\end{eqnarray} 
\begin{equation} 
\bar{x}_{i, j+1} - \bar{x}_{i,j} \in F^{i+j}\bar{C}_{-1}, \quad
\bar{y}_{i, j+1} - \bar{y}_{i,j} \in F^{i+j-U-1}\bar{C}_2, \quad 
\bar{z}_{i, j+1} - \bar{z}_{i,j} \in F^{i+j-2}\bar{C}_1. 
\end{equation}
Moreover we require the following conditions: 
\begin{itemize} 
\item 
$\bar{x}_{i,j}(a,k) \ne 0$ only if $\omega(\bar{a}) \ge 2\ep$ or $a=0$, $k \ge 2$. 
Moreover, $\bar{x}_{i,j}(0,2)$ is a cycle in $C^\dR_n(\mca{L}_3(0))$ 
such that $[\bar{x}_{i,j}(0,2)]$ corresponds to $(-1)^{n+1}  [L]$. 
\item 
$\bar{z}_{i,j}(a,k) \ne 0$ only if $\omega(\bar{a}) \ge 2\ep$ or $a=0$. 
Moreover, $\bar{z}_{i,j}(0,0)$ is a cycle in $C^\dR_n(\mca{L}_1(0))$ 
such that $[\bar{z}_{i,j}(0,0)]$ corresponds to $(-1)^{n+1} [L]$. 
\end{itemize} 
\end{lem} 
\begin{proof}
The proof is by induction on $j$. 
Let us define $x_{i,0}, \ldots, \bar{z}_{i,0}$ by (\ref{170627_1}). 

Assuming that we have defined a sequence $(x_{i,j}, \ldots, \bar{z}_{i,j})_{i \ge I}$
which satisfies the required conditions, 
we are going to define a sequence $(x_{i,j+1}, \ldots, \bar{z}_{i, j+1})_{i \ge I}$. 

Let us set 
\begin{align*} 
\Delta^i_x &:= x_{i+1, j} - e_+(\bar{x}_{i,j})  \in F^{i+j} C_{-1},  \\ 
\Delta^i_y &:= y_{i+1, j} - e_+(\bar{y}_{i,j}) \in F^{i+j-U-1} C_2,  \\
\Delta^i_z &:= z_{i+1, j}  - e_+(\bar{z}_{i,j}) \in F^{i+j-2} C_1. 
\end{align*} 
Since $e_-$ preserves $\partial$, $[\, , \,]$ and filtrations, 
(\ref{170627_2}) and (\ref{170627_3}) show that 
\begin{align*} 
&\partial x_{i+1, j} - \frac{1}{2} [ x_{i+1, j}, x_{i+1, j} ] \in F^{i+j+1}C_{-2},  \\ 
&\partial y_{i+1, j} - [ x_{i+1, j}, y_{i+1, j}] - z_{i+1, j}  \in F^{i+j-U}C_1,  \\ 
&\partial z_{i+1, j} - [x_{i+1, j} , z_{i+1, j}] \in F^{i+j-1}C_{-2}. 
\end{align*} 
Then we obtain 
\begin{align*} 
&\partial \Delta^i_x + e_+(\partial \bar{x}_{i,j} - \frac{1}{2}[\bar{x}_{i,j} , \bar{x}_{i,j}] ) \in F^{i+j+1}C_{-2}, \\ 
&\partial \Delta^i_y + e_+(\partial \bar{y}_{i,j} - [\bar{x}_{i,j}, \bar{y}_{i,j}] - \bar{z}_{i,j}) \in F^{i+j-U}C_1,  \\
&\partial \Delta^i_z + e_+(\partial \bar{z}_{i,j} - [\bar{x}_{i,j} , \bar{z}_{i,j}]) \in F^{i+j-1}C_0. 
\end{align*} 
On the other hand, we obtain 
\begin{align*} 
\partial \bigg( \partial \bar{x}_{i,j} - \frac{1}{2}[\bar{x}_{i,j}, \bar{x}_{i,j}] \bigg)  &= 
-\frac{1}{2} \bigg( \bigg[\partial \bar{x}_{i,j} - \frac{1}{2}[\bar{x}_{i,j}, \bar{x}_{i,j} ] , \bar{x}_{i,j} \bigg] - \bigg[ \bar{x}_{i,j}, \partial \bar{x}_{i,j} - \frac{1}{2}[\bar{x}_{i,j}, \bar{x}_{i,j} ] \bigg] \bigg) \in F^{i+j+1}\bar{C}_{-3}, \\ 
\partial (\partial \bar{y}_{i,j} - [ \bar{x}_{i,j} , \bar{y}_{i,j}] - \bar{z}_{i,j})
&= - \bigg[ \partial \bar{x}_{i,j} - \frac{1}{2}[\bar{x}_{i,j}, \bar{x}_{i,j}] , \bar{y}_{i,j} \bigg] + [\bar{x}_{i,j} , \partial \bar{y}_{i,j}  - [\bar{x}_{i,j}, \bar{y}_{i,j}] - \bar{z}_{i,j} ] \\ 
&\qquad - (\partial \bar{z}_{i,j}  - [\bar{x}_{i,j} , \bar{z}_{i,j}]) \in F^{i+j-U} \bar{C}_0, \\
\partial (\partial \bar{z}_{i,j} - [ \bar{x}_{i,j} , \bar{z}_{i,j}]) 
&= - \bigg[ \partial \bar{x}_{i,j} - \frac{1}{2}[\bar{x}_{i,j}, \bar{x}_{i,j}], \bar{z}_{i,j} \bigg] 
+ [\bar{x}_{i,j}, \partial \bar{z}_{i,j} - [\bar{x}_{i,j}, \bar{z}_{i,j}]]  \in F^{i+j-1}\bar{C}_{-1}. 
\end{align*} 
Applying Lemma \ref{161221_1} for 
\[ 
e_+ :  F^D \bar{C}/ F^{D+1} \bar{C} \to F^D C / F^{D+1} C
\] 
where $D = i+j, i+j-U-1, i+j-2$
(recall that $e_+$ is a surjective quasi-isomorphism, as we have seen in Section 4.5), 
we show that there exist
\[ 
\bar{\Delta}^i_x \in F^{i+j} \bar{C}_{-1}, \quad
\bar{\Delta}^i_y \in F^{i+j-U-1} \bar{C}_2, \quad 
\bar{\Delta}^i_z \in F^{i+j-2} \bar{C}_1 
\] 
such that
\[ 
e_+(\bar{\Delta}^i_x) - \Delta^i_x \in F^{i+j+1}C_{-1}, \quad
e_+(\bar{\Delta}^i_y) - \Delta^i_y \in  F^{i+j-U}C_2, \quad
e_+(\bar{\Delta}^i_z) - \Delta^i_z \in F^{i+j-1}C_1 
\] 
and 
\begin{align*} 
&\partial \bar{\Delta}^i_x + ( \partial \bar{x}_{i,j} - \frac{1}{2} [\bar{x}_{i,j}, \bar{x}_{i,j}] ) \in F^{i+j+1}\bar{C}_{-2},  \\
&\partial \bar{\Delta}^i_y + (\partial \bar{y}_{i,j} - [\bar{x}_{i,j}, \bar{y}_{i,j}] - \bar{z}_{i,j}) \in F^{i+j-U}\bar{C}_1, \\
&\partial \bar{\Delta}^i_z + ( \partial \bar{z}_{i,j} - [\bar{x}_{i,j}, \bar{z}_{i,j}]) \in F^{i+j-1}\bar{C}_0.  
\end{align*} 

\begin{rem}\label{170304_1}
\begin{itemize}
\item 
The $(a,k)$-components of $\Delta^i_x$ and $\partial \bar{x}_{i,j} + \frac{1}{2}[\bar{x}_{i,j}, \bar{x}_{i,j}]$ are nonzero only if $\omega(\bar{a}) \ge 2\ep$ or $a=0$. 
Then we may take $\bar{\Delta}^i_x$ so that its $(a,k)$-component is nonzero only if $\omega(\bar{a}) \ge 2 \ep$ or $a=0$. 
\item 
Similarly, we can take $\bar{\Delta}^i_z$ so that 
$\bar{\Delta}^i_z(a,k) \ne 0$ only if $\omega(\bar{a}) \ge 2\ep$ or $a=0$. 
\item
Since $\bar{\Delta}^i_x \in F^2 \bar{C}_{-1}$ and $\bar{\Delta}^i_z \in F^0 \bar{C}_1$, 
it follows that 
$\bar{\Delta}^i_x(0,k)=0$ if $k=0,1,2$ and 
$\bar{\Delta}^i_z(0,k)=0$ if $k=0$. 
\end{itemize} 
\end{rem} 

Finally, let us set 
\[ 
\bar{x}_{i,j+1} := \bar{x}_{i,j} + \bar{\Delta}^i_x, \quad
\bar{y}_{i,j+1} := \bar{y}_{i,j} + \bar{\Delta}^i_y, \quad
\bar{z}_{i,j+1} := \bar{z}_{i,j} + \bar{\Delta}^i_z 
\] 
and 
\[ 
x_{i,j+1}:= e_-(\bar{x}_{i,j+1}), \quad
y_{i,j+1}:= e_-(\bar{y}_{i,j+1}), \quad
z_{i,j+1}:= e_-(\bar{z}_{i,j+1}). 
\] 
Now we have to check that, for every $i \ge I$ there holds 
\begin{align*} 
&\partial \bar{x}_{i, j+1} - \frac{1}{2} [\bar{x}_{i, j+1} , \bar{x}_{i,j+1}] \in F^{i+j+1} \bar{C}_{-2}, \\ 
&\partial \bar{y}_{i, j+1} - [\bar{x}_{i, j+1}, \bar{y}_{i, j+1}] - \bar{z}_{i, j+1} \in F^{i+j-U} \bar{C}_1, \\ 
&\partial \bar{z}_{i, j+1} - [\bar{x}_{i, j+1}, \bar{z}_{i, j+1}]  \in F^{i+j-1} \bar{C}_0, \\ 
&x_{i+1, j+1} - e_+(\bar{x}_{i, j+1})\in F^{i+j+1} C_{-1}, \\ 
&y_{i+1, j+1} - e_+(\bar{y}_{i, j+1}) \in F^{i+j-U}C_{-2}, \\ 
&z_{i+1, j+1} - e_+(\bar{z}_{i, j+1}) \in F^{i+j-1}C_1.
\end{align*} 
The first formula holds since 
\[ 
\partial \bar{x}_{i, j+1} - \frac{1}{2} [\bar{x}_{i, j+1} , \bar{x}_{i,j+1}]
= \biggl( \partial \bar{x}_{i,j} + \partial \bar{\Delta}^i_x - \frac{1}{2}[\bar{x}_{i,j}, \bar{x}_{i,j}] \biggr) - \frac{1}{2} [\bar{\Delta}^i_x, \bar{\Delta}^i_x] - [\bar{x}_{i,j} , \bar{\Delta}^i_x]
\] 
and all three terms in the RHS are in $F^{i+j+1} \bar{C}_{-2}$. 
Similarly, the second and the third formulas follow from 
\begin{align*} 
&\partial \bar{y}_{i, j+1} - [\bar{x}_{i, j+1}, \bar{y}_{i, j+1}] - \bar{z}_{i, j+1}
= \\
&\quad  (\partial \bar{\Delta}^i_y + \partial \bar{y}_{i,j} - [\bar{x}_{i,j}, \bar{y}_{i,j}] - \bar{z}_{i,j}) - [\bar{\Delta}^i_x, \bar{y}_{i,j}] - [\bar{x}_{i,j}, \bar{\Delta}^i_y] - [\bar{\Delta}^i_x, \bar{\Delta}^i_y] - \bar{\Delta}^i_z, \\
&\partial \bar{z}_{i, j+1} - [\bar{x}_{i, j+1} , \bar{z}_{i, j+1}] = \\ 
&\quad (\partial \bar{z}_{i, j} + \partial \bar{\Delta}^i_z - [\bar{x}_{i, j} , \bar{z}_{i, j}]) 
- [\bar{\Delta}^i_x, \bar{z}_{i, j}] - [\bar{x}_{i, j} , \bar{\Delta}^i_z] - [\bar{\Delta}^i_x, \bar{\Delta}^i_z]. 
\end{align*} 
The fourth formula holds since 
\begin{align*}
x_{i+1, j+1} - e_+(\bar{x}_{i,j+1}) &= (x_{i+1,j+1} - x_{i+1,j}) + (x_{i+1,j} - e_+(\bar{x}_{i,j})) + e_+(\bar{x}_{i,j} - \bar{x}_{i, j+1}) \\
&= e_-(\bar{\Delta}^{i+1}_x) + (\Delta^i_x - e_+(\bar{\Delta}^i_x)).
\end{align*} 
Similarly, the fifth and the sixth formulas follow from 
\begin{align*} 
y_{i+1, j+1} - e_+(\bar{y}_{i, j+1})
&=(y_{i+1, j+1} - y_{i+1, j}) + (y_{i+1, j} - e_+(\bar{y}_{i,j})) + e_+(\bar{y}_{i,j} - \bar{y}_{i, j+1}) \\
&=e_-(\bar{\Delta}^{i+1}_y) + \Delta^i_y - e_+(\bar{\Delta}^i_y), \\
z_{i+1, j+1} - e_+(\bar{z}_{i, j+1}) 
&=(z_{i+1, j+1} - z_{i+1, j}) + (z_{i+1, j} - e_+(\bar{z}_{i,j})) + e_+(\bar{z}_{i,j} - \bar{z}_{i, j+1}) \\
&=e_-(\bar{\Delta}^{i+1}_z) + \Delta^i_z - e_+(\bar{\Delta}^i_z). 
\end{align*} 
Finally, by the induction hypothesis, $\bar{x}_{i,j}$ and $\bar{z}_{i,j}$ satisfy the last two conditions in the statement.
Then Remark \ref{170304_1} shows that $\bar{x}_{i, j+1}$ and $\bar{z}_{i, j+1}$ also satisfy these conditions. 
\end{proof}

\begin{proof}
[\textbf{Proof of Theorem \ref{161214_2} assuming Theorem \ref{161215_1}}]

Let us fix an integer $i \ge I$. 
Then for every $j \ge 0$, there holds 
\[ 
x_{i, j+1} - x_{i,j} \in F^{i+j} C_{-1}, \quad
y_{i, j+1} - y_{i,j} \in F^{i+j-U-1} C_2, \quad
z_{i, j+1} - z_{i,j} \in F^{i+j-2} C_1.
\] 
Then, the limits 
\[ 
x:= \lim_{j \to \infty} x_{i,j} \in \wh{C}_{-1}, \quad
y:= \lim_{j \to \infty} y_{i,j} \in \wh{C}_2, \quad
z:= \lim_{j \to \infty} z_{i,j} \in \wh{C}_1
\] 
exist, and satisfy 
\[ 
\partial x - \frac{1}{2} [x,x] = 0, \qquad \partial y - [x,y] = z. 
\]
Conditions (iii), (iv) in Theorem \ref{161214_2} follow from the last two conditions in Lemma \ref{161216_1}. 
\end{proof} 

\section{Proof of Theorem \ref{161215_1} modulo technical results}

Our arguments in the rest of this paper are based on the theory of Kuranishi structures (abbreviated as K-structures), 
in particular we heavily rely on \cite{FOOO_Kuranishi} by Fukaya-Oh-Ohta-Ono. 
Section 10 very briefly explains some notions in the theory of Kuranishi structures. 

The goal of this section is to prove Theorem \ref{161215_1} modulo technical results, 
which are proved in Sections 8 and 9. 
Our proof of Theorem \ref{161215_1} is based on the following principle: 
given a compact K-space $(X, \wh{\mca{U}})$
with a CF-perturbation $\wh{\mf{S}} = (\wh{\mf{S}}^\ep)_{0 < \ep \le 1}$ and a strongly smooth map $\wh{f}: (X, \wh{\mca{U}}) \to \mca{L}_{k+1}$, 
one can define a de Rham chain $\wh{f}_*(X, \wh{\mca{U}}, \wh{\mf{S}}^\ep) \in C^\dR_*(\mca{L}_{k+1})$ for sufficiently small $\ep$. 
In Section 7.1, we state this principle and its variants in rigorous forms. 
Proofs of these results will be carried out in Section 8. 
In Sections 7.2--7.6, we apply this principle to prove Theorem \ref{161215_1}. 

In Section 7.2, we introduce compactified moduli spaces of (perturbed) pseudo-holomorphic disks, 
and equip K-structures (with boundaries and corners) on these spaces
so that the (normalized) boundary of each moduli space is naturally isomorphic to a disjoint union 
of fiber products of lower (in terms of symplectic area) moduli spaces. 
We call this relation ``boundary $\cong$ fiber product'' relation. 

To apply the principle in Section 7.1 to moduli spaces defined in Section 7.2, 
we have to define \textit{strongly smooth} maps from these moduli spaces to spaces of \textit{smooth} loops with marked points, 
so that these maps are compatible with the ``boundary $\cong$ fiber product'' relation. 
The idea is to assign the boundary loop to each pseudo-holomorphic disk, however we have to be careful to achieve smoothness, 
which can be very subtle on boundary of moduli spaces. 
In this paper, we achieve smoothness in the following two steps. 
The first step is to define \textit{strongly continuous} maps from these moduli spaces to spaces of \textit{continuous} loops with marked points. 
This is much easier and we explain details in Sections 7.3 and 7.4. 
The second step is to approximate (with respect to the $C^0$-topology) these continuous maps by smooth maps. 
In this step we use an abstract approximation result ($C^0$-approximation lemma) which we state in Section 7.5 and prove in Section 9. 
In Section 7.5, we also introduce CF-perturbations of K-structures on these moduli spaces. 
Finally in Section 7.6 we complete the proof assuming technical results presented in Sections 7.1 and 7.5, 
which are proved in Sections 8 and 9, respectively. 

\subsection{Strongly smooth map from a K-space with a CF-perturbation gives a de Rham chain} 

Let $(X, \wh{\mca{U}})$ be a K-space equipped with a differential form $\wh{\omega}$ and a CF-perturbation $\wh{\mf{S}} = (\wh{\mf{S}}^\ep)_{0 < \ep \le 1}$. 
(The notion of CF-perturbation is explained in Section 7 in \cite{FOOO_Kuranishi}). 
Given a strongly smooth map (see Definition \ref{170902_2} below) 
$\wh{f} : (X, \wh{\mca{U}}) \to \mca{L}_{k+1}$, we define 
\[
\wh{f}_*(X, \wh{\mca{U}}, \wh{\omega}, \wh{\mf{S}}^\ep ) \in C^\dR_* (\mca{L}_{k+1})
\]
for sufficiently small $\ep>0$, 
and state Stokes' formula and fiber product formula. 
We also consider the version that $X$ is an admissible K-space with boundaries and corners. 
The goal of this subsection is to state these results in a formal manner so that we can use them to complete the proof of Theorem \ref{161215_1}. 
Proofs of these results are explained in Section 8. 

\subsubsection{K-space without boundary} 

Let us start from the following definition: 

\begin{defn}\label{170902_2} 
Let $(X, \wh{\mca{U}})$ be a K-space without boundary. 
A \textit{strongly smooth map} from $(X, \wh{\mca{U}})$ to $\mca{L}_{k+1}$ is a family 
$\wh{f} = (f_p)_{p \in X}$ such that the following conditions hold: 
\begin{itemize}
\item $f_p$ is a smooth map (in the sense of Definition \ref{171205_1}) from $U_p$ to $\mca{L}_{k+1}$ for every $p \in X$.
\item For every $p \in X$ and $q \in \Image \psi_p$, 
there holds $f_p \circ \ph_{pq} = f_q|_{U_{pq}}$. 
\end{itemize} 
The underlying set-theoretic map $X \to \mca{L}_{k+1}$ is denoted by $f$. 
\end{defn}

\begin{thm}\label{170628_1} 
Let $(X, \wh{\mca{U}})$ be a compact, oriented K-space of dimension $d$ 
with a 
strongly smooth map $\wh{f} : (X, \wh{\mca{U}}) \to \mca{L}_{k+1}$, 
a differential form $\wh{\omega}$, 
and a CF-perturbation $\wh{\mf{S}} = (\wh{\mf{S}}^\ep)_{0<\ep \le 1}$. 
We assume that $\wh{\mf{S}}$ is transversal to $0$, 
and $\ev_0 \circ \wh{f}: (X, \wh{\mca{U}}) \to L$ is strongly submersive with respect to $\wh{\mf{S}}$ (see Definition 9.2 in \cite{FOOO_Kuranishi}). 
Under these assumptions, one can define a de Rham chain 
\begin{equation}\label{170918_2} 
\wh{f}_* (X, \wh{\mca{U}}, \wh{\omega}, \wh{\mf{S}}^\ep) \in C^\dR_{d- |\hat{\omega}|} (\mca{L}_{k+1})
\end{equation}
for sufficiently small $\ep$, 
so that Stokes' formula (Theorem \ref{170628_2}) and 
the fiber product formula (Theorem \ref{170628_3}) hold. 
\end{thm} 

\begin{rem}\label{170918_3} 
When $\wh{\omega} \equiv 1$, we abbreviate
the LHS of  (\ref{170918_2}) 
as $\wh{f}_*(X, \wh{\mca{U}}, \wh{\mf{S}}^\ep)$. 
\end{rem} 

\begin{rem}
In the statement of Theorem \ref{170628_1}, 
``for sufficiently small $\ep$'' is used slightly loosely. 
Strictly speaking, it means the following: 
the definition of $\wh{f}_*(X, \wh{\mca{U}}, \wh{\omega}, \wh{\mf{S}}^\ep)$ involves 
some auxiliary choices (good coordinate system and partition of unity), 
however it is well-defined in the sense of $\spadesuit$ (see Definition \ref{180301_1}). 
Namely, for any choices  $c_1$ and $c_2$, there exists $\ep(c_1, c_2)>0$ such that, 
the definition with $c_1$ coincides with the definition with $c_2$ when $\ep \in (0, \ep(c_1, c_2))$. 
A similar remark applies to all places in Section 7.1 where we use 
``for sufficiently small $\ep$''. 
\end{rem} 

Stokes' formula is easy to state: 

\begin{thm}\label{170628_2} 
For sufficiently small $\ep >0$, there holds 
\[ 
\partial 
(\wh{f}_* (X, \wh{\mca{U}}, \wh{\omega}, \wh{\mf{S}}^\ep)) = 
(-1)^{|\wh{\omega}|+1} \wh{f}_*(X, \wh{\mca{U}}, d \wh{\omega}, \wh{\mf{S}}^\ep). 
\]
\end{thm} 

To state the fiber product formula, 
we need some notations. 
Suppose, for $i=1,2$, we have the following data: 
\begin{itemize}
\item $(X_i, \wh{\mca{U}}_i)$: a compact oriented K-space of dimension $d_i$, 
\item a strongly smooth map $\wh{f}_i:  (X_i, \wh{\mca{U}_i}) \to \mca{L}_{k_i+1}$, 
\item a differential form $\wh{\omega}_i$ on $(X_i, \wh{\mca{U}}_i )$, 
\item a CF-perturbation $\wh{\mf{S}}_i $ on $(X_i, \wh{\mca{U}}_i )$ such that $\ev_0 \circ \wh{f}_i: (X_i, \wh{\mca{U}}_i) \to L$ is strongly submersive with respect to $\wh{\mf{S}}_i$.  
\end{itemize} 
Due to Theorem \ref{170628_1}, one can define 
$(\wh{f}_i)_*(X_i, \wh{\mca{U}}_i, \wh{\omega}_i, \wh{\mf{S}}^\ep_i) \in C^\dR_{d_i - |\wh{\omega}_i|} (\mca{L}_{k_i+1})$ 
for sufficiently small $\ep$. 
On the other hand, for every $j \in \{1, \ldots, k_1\}$, 
one can take a fiber product of K-spaces and define 
\[ 
(X_{12}, \wh{\mca{U}}_{12}) := (X_1, \wh{\mca{U}}_1) \fbp{\ev_j \circ \wh{f}_1}{\ev_0 \circ \wh{f}_2} (X_2, \wh{\mca{U}}_2).
\]

For the definition of fiber product of K-spaces, see Section 4.1 in \cite{FOOO_Kuranishi}. 
Our sign convention for the fiber product is explained in Section 4.2. 
One can also define fiber product of CF-perturbations  $\wh{\mf{S}}_1 \times \wh{\mf{S}}_2$
on $(X_{12}, \wh{\mca{U}}_{12})$ (see Definition 10.13 in \cite{FOOO_Kuranishi}). 
Finally we define a differential form $\wh{\omega}_{12}$ on $(X_{12}, \wh{\mca{U}}_{12})$ by 
\[ 
\wh{\omega}_{12} := (-1)^{(d_1- |\hat{\omega}_1| - n)|\hat{\omega}_2|}  \cdot  \wh{\omega}_1 \times \wh{\omega}_2 
\] 
and a strongly smooth map $\wh{f}_{12}: (X_{12}, \wh{\mca{U}}_{12}) \to \mca{L}_{k_1+k_2}$ by 
\begin{eqnarray}\label{171018_1} 
&(f_{12})_{(p_1, p_2)}(x_1, x_2) := \con_j ( (f_1)_{p_1}(x_1) , (f_2)_{p_2}(x_2)), \\
&(x_1 \in U_{p_1}, \, x_2 \in U_{p_2}, \, \ev_j \circ f_{p_1}(x_1) = \ev_0 \circ f_{p_2}(x_2)).  \nonumber 
\end{eqnarray} 
Then one can state the fiber product formula
(the proof will be sketched in Section 8): 

\begin{thm}\label{170628_3}
In the situation described above, there holds 
\[
(\wh{f}_{12})_*(X_{12}, \wh{\mca{U}}_{12}, \wh{\omega}_{12}, \wh{\mf{S}}_{12}^\ep) = 
(\wh{f}_1)_*(X_1, \wh{\mca{U}}_1, \wh{\omega}_1, \wh{\mf{S}}_1^\ep) \circ_j
(\wh{f}_2)_*(X_2, \wh{\mca{U}}_2, \wh{\omega}_2, \wh{\mf{S}}_2^\ep)
\]
for sufficiently small $\ep>0$, where the RHS is the fiber product of de Rham chains (see Section 4.3). 
\end{thm}

\subsubsection{Admissible K-space} 

Next we consider the case that $X$ is an admissible K-space, 
which roughly means that $X$ is a K-space with boundaries and corners, 
and all coordinate change data decay exponentially near boundaries. 
For rigorous definitions of 
admissible manifolds, vector bundles, and K-structures etc., 
see Section 25 in \cite{FOOO_Kuranishi}. 

\begin{defn}\label{170701_1} 
\begin{enumerate}
\item[(i):] Let $U$ be an admissible manifold.
A $C^\infty$-map  
\[ 
f: U \to \mca{L}_{k+1}; \quad u \mapsto (T(u), \gamma(u), t_1(u), \ldots, t_k(u))
\]
is called \textit{admissible} if the following conditions hold: 
\begin{itemize}
\item $T, t_1, \ldots, t_k: U \to \R$ are admissible (see Definition 25.3 in \cite{FOOO_Kuranishi}). 
\item $U \times S^1 \to L; \, (u, \theta) \mapsto \gamma(u)(T(u) \theta)$ is admissible (see Remark \ref{171205_2} below). 
\item $\ev_0 \circ f: U \to L$ is strata-wise submersive
(i.e. the restriction of $\ev_0 \circ f$ to each open strata of $U$ is a submersion). 
\end{itemize} 
\item[(ii):] Let $(X, \wh{\mca{U}})$ be an admissible K-space. 
An admissible map from $(X, \wh{\mca{U}})$ to $\mca{L}_{k+1}$ is a family $\wh{f} = (f_p)_{p \in X}$ such that 
\begin{itemize}
\item $f_p$ is an admissible map from $U_p$ to $\mca{L}_{k+1}$ for every $p \in X$. 
\item For every $p \in X$ and $q \in \Image \psi_p$, there holds $f_p \circ \ph_{pq} = f_q|_{U_{pq}}$. 
\end{itemize} 
The underlying set-theoretic map $X \to \mca{L}_{k+1}$
is denoted by $f$. 
\end{enumerate} 
\end{defn} 
\begin{rem}\label{171205_2} 
In the second bullet of the definition (i) above, 
the admissible structure on $U \times S^1$ is defined as follows: 
let $(x_1, \ldots, x_l,  y_1, \ldots, y_k) \,(x_i \in \R, y_i \in \R_{\ge 0})$ 
be an admissible chart on $U$ near a point on the codimension $k$ corner, 
and $z$ be any chart on $S^1$. 
Then we define the admissible structure on $U \times S^1$ so that 
$(x_1, \ldots, x_l, y_1, \ldots, y_k, z)$ is an admissible chart on $U \times S^1$. 
\end{rem}

The next result is a version of Theorem \ref{170628_1}
for admissible K-spaces. 

\begin{thm}\label{170628_4}
Let $(X, \wh{\mca{U}})$ be a compact, oriented, admissible K-space of dimension $d$, 
and 
$\wh{f}: (X, \wh{\mca{U}}) \to \mca{L}_{k+1}$ be an admissible map, 
$\wh{\omega}$ be an admissible differential form on $(X, \wh{\mca{U}})$, and 
$\wh{\mf{S}}$ be an admissible CF-perturbation of $(X, \wh{\mca{U}})$. 
We assume that $\wh{\mf{S}}$ is transversal to $0$, and 
$\ev_0 \circ \wh{f}:  (X, \wh{\mca{U}}) \to L$ is strata-wise strongly submersive with respect to $\wh{\mf{S}}$. 
Then one can define
\[ 
\wh{f}_*( X, \wh{\mca{U}}, \wh{\omega}, \wh{\mf{S}}^\ep) \in C^\dR_{d- |\hat{\omega}|} (\mca{L}_{k+1})
\]
for sufficiently small $\ep>0$, so that Stokes' formula (Theorem \ref{170628_5}) and 
the fiber product formula (Theorem \ref{170628_6}) hold. 
\end{thm}
\begin{rem}
To define the de Rham chain $\wh{f}_*( X, \wh{\mca{U}}, \wh{\omega}, \wh{\mf{S}}^\ep)$, 
we put a collar on $X$ and take an auxiliary cut-off function on the collar. 
It seems that the chain $\wh{f}_*( X, \wh{\mca{U}}, \wh{\omega}, \wh{\mf{S}}^\ep)$
depends on choice of the cut-off function. 
See Section 8.1.2 for more details. 
\end{rem}

Now we state Stokes' formula: 

\begin{thm}\label{170628_5} 
In the situation of Theorem \ref{170628_4}, for sufficiently small $\ep >0$ there holds
\[ 
\partial (\wh{f}_*(X, \wh{\mca{U}}, \wh{\omega}, \wh{\mf{S}}^\ep)) = 
(-1)^{|\wh{\omega}|}
\wh{f}_*(\partial X , \wh{\mca{U}}|_{\partial X}, \wh{\omega}|_{\partial X}, \wh{\mf{S}}^\ep|_{\partial X}) + (-1)^{|\wh{\omega}|+1}\wh{f}_*(X, \wh{\mca{U}}, d \wh{\omega}, \wh{\mf{S}}^\ep),
\]
where $\partial X$ denotes the normalized boundary of $X$,
which is again an admissible K-space. 
\end{thm}

Next we state the fiber product formula. 
Suppose, for $i=1, 2$, we have the following data: 
\begin{itemize}
\item An admissible K-space $(X_i, \wh{\mca{U}}_i)$. 
\item An admissible map $\wh{f}_i: (X_i, \wh{\mca{U}}_i) \to \mca{L}_{k_i+1}$. 
\item An admissible differential form $\wh{\omega}_i$ on $(X_i, \wh{\mca{U}}_i)$. 
\item An admissible CF-perturbation $\wh{\mf{S}}_i$ on $(X_i, \wh{\mca{U}}_i)$ such that 
$\ev_0 \circ \wh{f}_i: (X_i, \wh{\mca{U}}_i) \to L$ is strata-wise strongly submersive with respect to $\wh{\mf{S}}_i$.
\end{itemize} 
Under these assumptions, for every $j \in \{1, \ldots, k_1 \}$, 
the fiber product
\[
(X_{12}, \wh{\mca{U}}_{12}) := (X_1, \wh{\mca{U}}_1) \fbp{\ev_j \circ f_1}{\ev_0 \circ f_2} (X_2, \wh{\mca{U}}_2)
\]
has an admissible K-structure. 
Moreover, both $\wh{\omega}_1 \times \wh{\omega}_2$ and $\wh{\mf{S}}_1 \times \wh{\mf{S}}_2$ are admissible. 
Finally, a strongly smooth map $\wh{f}_{12}: (X_{12}, \wh{\mca{U}}_{12}) \to \mca{L}_{k_1+k_2}$ 
which is defined by the same formula as (\ref{171018_1}) is also admissible. 

\begin{thm}\label{170628_6} 
In the situation described above, there holds
\[
(\wh{f}_{12})_*(X_{12}, \wh{\mca{U}}_{12}, \wh{\omega}_{12}, \wh{\mf{S}}_{12}^\ep) = 
(\wh{f}_1)_*(X_1, \wh{\mca{U}}_1, \wh{\omega}_1, \wh{\mf{S}}_1^\ep) \circ_j
(\wh{f}_2)_*(X_2, \wh{\mca{U}}_2, \wh{\omega}_2, \wh{\mf{S}}_2^\ep)
\]
for sufficiently small $\ep>0$.
\end{thm}

\subsubsection{Admissible K-space over an interval} 

Finally we consider the case that $X$ is an admissible K-space and 
the target of the map $f$ is $[a,b] \times  \mca{L}_{k+1}$. 

\begin{defn}\label{170701_2} 
\begin{enumerate}
\item[(i):] Let $U$ be an admissible manifold. 
A $C^\infty$-map $f: U \to [a, b]  \times \mca{L}_{k+1}$
is called \textit{admissible} if the following conditions hold: 
\begin{itemize} 
\item $\pr_{\mca{L}_{k+1}} \circ f: U \to \mca{L}_{k+1}$ is admissible in the sense of Definition \ref{170701_1}. 
\item Let us denote $\tau:= \pr_{[a,b ]} \circ f$, and suppose that $\tau(p)=a$ (resp. $\tau(p)=b$) for $p \in U$. 
Then $p$ is on codimension $k \ge 1$ 
corner of $U$, 
and there exists an admissible chart 
$(y, t_1, \ldots, t_k) \,(t_i \in \R_{\ge 0})$ defined near $p$, 
such that 
$p$ corresponds to $(y_0, 0, \ldots, 0)$ and 
$\tau(y, t_1, \ldots, t_k) = t_1 + a$ (resp. $b-t_1$). 
\item $(\tau , \ev_0 \circ \pr_{\mca{L}_{k+1}} \circ f): U \to [a,b] \times L$ is a corner stratified submersion (see Definition 26.3 in \cite{FOOO_Kuranishi}). 
\end{itemize} 

\item[(ii):] 
Let $(X, \wh{\mca{U}})$ be an admissible K-space. 
An admissible map from $(X, \wh{\mca{U}})$ to $[a, b] \times \mca{L}_{k+1}$
is a family $\wh{f} = (f_p)_{p \in X}$ such that 
\begin{itemize}
\item $f_p$ is an admissible map from $U_p$ to $[a, b] \times \mca{L}_{k+1}$ for every $p \in X$. 
\item For any $p \in X$ and $q \in \Image \psi_p$, there holds $f_p \circ \ph_{pq} = f_q|_{U_{pq}}$.
\end{itemize} 
The underlying set-theoretic map
$X \to [a,b] \times \mca{L}_{k+1}$ 
is denoted by $f$. 
\end{enumerate}
\end{defn}

In the situation of Definition \ref{170701_2} (ii), the normalized boundary $\partial X$ is decomposed as
$\partial X = \partial_v X \sqcup \partial_h X$, where 
$\partial_v X$ (resp. $\partial_h X$) denotes the \textit{vertical} (resp. \textit{horizontal}) boundary
(see Definition 26.10 in \cite{FOOO_Kuranishi}). 
Moreover, $\partial_v X$ is decomposed into 
$\part_- X:= f^{-1}( \{a\} \times \mca{L}_{k+1})$ and
$\part_+ X:= f^{-1}(\{b\} \times \mca{L}_{k+1})$. 

\begin{thm}\label{170701_3} 
Let $(X, \wh{\mca{U}})$ be a compact, oriented, admissible K-space of dimension $d$, 
$\wh{f}$ be an admissible map (in the sense of Definition \ref{170701_2}) 
from $(X, \wh{\mca{U}})$ to $[a, b] \times \mca{L}_{k+1}$, 
$\wh{\omega}$ be an admissible differential form  on $(X, \wh{\mca{U}})$, 
$\wh{\mf{S}}$ be an admissible CF-perturbation of $(X, \wh{\mca{U}})$.
Assume that $\wh{\mf{S}}$ is transversal to $0$, and 
\[ 
(\pr_{[a,b]} \circ \wh{f} , \ev_0 \circ \pr_{\mca{L}_{k+1}} \circ \wh{f}): (X, \wh{\mca{U}}) \to [a,b]  \times L
\]
is a stratified submersion with respect to $\wh{\mf{S}}$. 
Then one can define 
\[ 
\wh{f}_*(X, \wh{\mca{U}}, \wh{\omega}, \wh{\mf{S}}^\ep) \in \bar{C}^\dR_{d- |\wh{\omega}| - 1} (\mca{L}_{k+1})
\]
for sufficiently small $\ep>0$, so that there holds 
\begin{equation}
e_+(\wh{f}_*(X, \wh{\mca{U}}, \wh{\omega}, \wh{\mf{S}}^\ep)) =  (\wh{f}|_{\partial_+ X})_* (\partial_+ X, \wh{\mca{U}}|_{\partial_+ X}, \wh{\omega}|_{\partial_+ X}, \wh{\mf{S}}^\ep|_{\partial_+ X}), 
\end{equation} 
\begin{equation}\label{171206_1}
e_-(\wh{f}_*(X, \wh{\mca{U}}, \wh{\omega}, \wh{\mf{S}}^\ep)) =  (\wh{f}|_{\partial_- X})_* (\partial_- X, \wh{\mca{U}}|_{\partial_- X}, \wh{\omega}|_{\partial_- X}, \wh{\mf{S}}^\ep|_{\partial_- X}). 
\end{equation} 
Moreover Stokes' formula and the fiber product formula hold. 
\end{thm} 

Let us state Stokes' formula: 

\begin{prop}
For sufficiently small $\ep>0$, there holds
\[ 
\partial \wh{f}_*(X, \wh{\mca{U}}, \wh{\omega}, \wh{\mf{S}}^\ep) = 
(-1)^{|\hat{\omega}|}
(\wh{f}|_{\part_h X})_*(\part_h X, \wh{\mca{U}}|_{\part_h X}, \wh{\omega}|_{\part_h X}, \wh{\mf{S}}^\ep|_{\part_h X})  
+ (-1)^{|\hat{\omega}| + 1}
\wh{f}_*(X, \wh{\mca{U}}, d \wh{\omega}, \wh{\mf{S}}^\ep). 
\]
\end{prop}

Finally we state the fiber product formula.
Suppose, for $i \in \{1, 2\}$, we have 
$(X_i, \wh{\mca{U}}_i)$, 
$\wh{\omega}_i$, 
$\wh{\mf{S}}_i$, 
$k_i \in \Z_{\ge 0}$ and 
$\wh{f}_i: (X_i, \wh{\mca{U}}_i) \to [a,b] \times \mca{L}_{k_i+1}$ satisfying the assumptions 
in Theorem \ref{170701_3}.  
For every $j \in \{1, \ldots, k_1 \}$, 
the fiber product
\[
(X_{12}, \wh{\mca{U}}_{12}) := (X_1, \wh{\mca{U}}_1) \fbp{\ev_j \circ \wh{f}_1}{\ev_0 \circ \wh{f}_2} (X_2, \wh{\mca{U}}_2)
\]
has an admissible K-structure. 
Moreover, both $\wh{\omega}_1 \times \wh{\omega}_2$ and $\wh{\mf{S}}_1 \times \wh{\mf{S}}_2$ are admissible. 
Finally, a smooth map $\wh{f}_{12}: (X_{12}, \wh{\mca{U}}_{12}) \to [a, b] \times \mca{L}_{k_1+k_2}$ 
which is defined by the same formula as (\ref{171018_1}) is also admissible. 

\begin{thm}\label{171018_2} 
In the situation described above, there holds
\[
(\wh{f}_{12})_*(X_{12}, \wh{\mca{U}}_{12}, \wh{\omega}_{12}, \wh{\mf{S}}_{12}^\ep)  = 
(\wh{f}_1)_*(X_1, \wh{\mca{U}}_1, \wh{\omega}_1, \wh{\mf{S}}_1^\ep) \circ_j
(\wh{f}_2)_*(X_2, \wh{\mca{U}}_2, \wh{\omega}_2, \wh{\mf{S}}_2^\ep)
\]
for sufficiently small $\ep>0$.
\end{thm}

\subsection{Moduli spaces of (perturbed) pseudo-holomorphic disks with boundary marked points} 

\subsubsection{Uncompactified moduli spaces} 

Let $D$ denote the unit holomorphic disk, namely $D:= \{ z \in \C \mid |z| \le 1\}$. 
For every $k \in \Z_{\ge 0}$ and $\beta \in H_2(\C^n, L)$, we define a set $\mtrg{\mca{M}}_{k+1}(\beta)$ 
in the following way. 
When $\beta=0$ and $k=0$ or $k=1$, we define 
both 
$\mtrg{\mca{M}}_1(0)$ and $\mtrg{\mca{M}}_2(0)$ to be empty set. 
In the other cases, namely $\beta \ne 0$ or $k \ne 0, 1$, we define 
\[ 
\mtrg{\mca{M}}_{k+1}(\beta):= \{(u,z_0, \ldots, z_k)\}/\Aut (D)
\]
where $u: (D, \partial D) \to (\C^n, L)$ satisfies $\bar{\partial} u = 0$, $[u] = \beta$, and $z_0, \ldots, z_k$ are distinct points on $\partial D$ aligned in anti-clockwise order. 
The right action of $\Aut (D)$ is defined so that 
\[ 
(u, z_0, \ldots, z_k)^\rho:= (u \circ \rho, \rho^{-1}(z_0), \ldots, \rho^{-1}(z_k)). 
\] 
For each $j \in \{0, \ldots, k\}$, we define an evaluation map 
$\ev_j: \mtrg{\mca{M}}_{k+1}(\beta) \to L$ by 
\begin{equation}\label{170610_1} 
\ev_j [(u, z_0, \ldots, z_k)] := u(z_j). 
\end{equation}

Next we take a Hamiltonian $H \in C^\infty_c([0,1] \times \C^n)$ which displaces $L$. 
For every $t \in [0,1]$, let us set $H_t(z):= H(t,z)$, 
define a Hamiltonian vector field $X_{H_t}$ by 
$\omega_n( \, \cdot \, , X_{H_t}) = dH_t( \, \cdot \,)$, 
and define an isotopy $(\ph^t_H)_{t \in [0,1]}$ on $\C^n$ by 
\[
\ph^0_H = \id_{\C^n}, \qquad 
\partial_t \ph^t_H = X_{H_t}(\ph^t_H) \quad (\forall t \in [0,1]). 
\]
Then ``$H$ displaces $L$'' means that $\ph^1_H(L) \cap L = \emptyset$. 
We also assume that $H_t \equiv 0$ when $t \in [0,1/3] \cup [2/3,1]$. 

Let us define a family of perturbed Cauchy-Riemann operators $(\bar{\partial}_r)_{r \ge 0}$ in the following way. 
We first take a $C^\infty$-function $\chi: \R \to [0,1]$ such that 
$\chi \equiv 0$ on $\R_{\le 0}$ 
and $\chi \equiv 1$ on $\R_{\ge 1}$. 
For every $r \in \R_{\ge 0}$, we define 
$\chi_r(s):= \chi(r+s)\chi(-r-s)$. 
In particular $\chi_0 \equiv 0$. 

Let us take a complex structure $J$ on $\R \times [0,1]$ so that
$J(\partial_s) = \partial_t$, where $s$ denotes the coordinate on $\R$ and $t$ denotes the coordinate on $[0,1]$. 
Next we take a biholomorphic map from $D \setminus \{-1, 1\}$ to $\R \times [0,1]$. 
This defines $C^\infty$-functions 
$s: D \setminus \{-1, 1\} \to \R$
and $t: D \setminus \{-1, 1\} \to [0,1]$. 
For every $r \in \R_{\ge 0}$
and $u: (D, \partial D) \to (\C^n, L)$, we define
\[ 
\bar{\partial}_r u:= (du - X_{\chi_r(s) H_t}(u) \otimes dt)^{0,1}. 
\] 
Obviously $\bar{\partial}_0 = \bar{\partial}$. 
Now let us define 
\[
\mtrg{\mca{N}}^r_{k+1}(\beta):= \{(u, z_0=1, z_1, \ldots, z_k)\}
\] 
where $u: (D, \partial D) \to (\C^n, L)$ satisfies $\bar{\partial}_r u = 0$, $[u] = \beta$, and $1, z_1, \ldots, z_k$ are distinct points on $\partial D$ aligned in anti-clockwise order. 
For each $j \in \{0,\ldots, k\}$, 
the evaluation map $\ev_j: \mtrg{\mca{N}}^r_{k+1}(\beta) \to L$ is defined in the same formula as (\ref{170610_1}). 
Finally we define 
\[
\mtrg{\mca{N}}^{\ge 0}_{k+1}(\beta):= \bigcup_{r \ge 0} \mtrg{\mca{N}}^r_{k+1}(\beta).
\]
Let us summarize basic properties of these moduli spaces: 

\begin{lem}\label{170625_1} 
\begin{enumerate} 
\item[(i):] If $\omega(\beta)<0$ or $\omega(\beta)=0$ and $\beta \ne 0$, then $\mtrg{\mca{M}}_{k+1}(\beta) = \mtrg{\mca{N}}^0_{k+1}(\beta)=\emptyset$.
\item[(ii):] Consider the case $\beta=0$: 
\begin{itemize}
\item $\mtrg{\mca{M}}_{k+1}(0)$ consists of constant maps for every $k \ge 2$.
\item $\mtrg{\mca{N}}^0_{k+1}(0)$ consists of constant maps for every $k \ge 0$. 
\end{itemize} 
\item[(iii):] Let $\| \, \cdot \, \|$ denote the Hofer's norm, namely 
\[ 
\| H \| := \int_0^1 \, (\max H_t - \min H_t) \, dt. 
\] 
If $\omega(\beta) + 2 \| H\| < 0$, then $\mtrg{\mca{N}}^{\ge 0} _{k+1}(\beta)  = \emptyset$. 
\item[(iv):] First notice that, for every $k \in \Z_{\ge 0}$ and $r \in \R_{\ge 0}$, 
\[
\mtrg{\mca{M}}_{k+1}(\beta) = \emptyset \iff \mtrg{\mca{M}}_1(\beta) = \emptyset, \qquad
\mtrg{\mca{N}}^r_{k+1}(\beta) = \emptyset \iff \mtrg{\mca{N}}^r_1(\beta)=\emptyset.
\]
Now for every $c \in \R$, sets 
\begin{align*} 
&\{ \beta \in H_2(\C^n, L) \mid \mtrg{\mca{M}}_1(\beta) \ne \emptyset, \, \omega(\beta)  < c\},  \\ 
&\{ \beta \in H_2(\C^n, L) \mid \mtrg{\mca{N}}^r_1(\beta) \ne \emptyset, \, \omega(\beta) < c\}
\end{align*} 
are both finite. 
\item[(v):] 
For every $\beta \in H_2(\C^n, L)$, 
there exists $r(\beta)>0$ such that
$\bigcup_{r \ge r(\beta)} \mtrg{\mca{N}}^r_1(\beta) = \emptyset$. 
\end{enumerate} 
\end{lem}
\begin{proof}
(i), (ii) are straightforward from definitions.
(iii) follows from standard computations. 
(iv) follows from Gromov compactness theorem.  
If (v) is not the case, there exists 
$v: (\R \times [0,1], \R \times \{0,1\}) \to (\C^n, L)$ satisfying 
\[ 
\int_{\R \times [0,1]} | \partial_s v|^2 \, dsdt < \infty, \qquad 
\partial_s v - J( \partial_t v - X_{H_t}(v)) = 0, 
\] 
where $J$ denotes the standard complex structure on $\C^n$.
Then $\gamma(t):= \lim_{s \to \infty} v(s,t)$ satisfies 
$\gamma(0), \gamma(1) \in L$ and $\partial_t \gamma(t)= X_{H_t}(\gamma(t))$, 
contradicting the assumption that $H$ displaces $L$. 
\end{proof}

\subsubsection{Compactified moduli spaces} 

In this subsubsection, we define compactified moduli spaces 
$\mca{M}_{k+1}(\beta)$, $\mca{N}^0_{k+1}(\beta)$ and $\mca{N}^{\ge 0}_{k+1}(\beta)$,
taking fiber products of uncompactified moduli spaces along decorated rooted ribbon trees
(see Definition \ref{170902_1} below). 
K-structures on these spaces are defined in the next subsubsection. 

The next definition is a modified version 
of Definition 21.2 in \cite{FOOO_Kuranishi}. 

\begin{defn}\label{170902_1} 
A \textit{decorated rooted ribbon tree} is a pair $(T, B)$ such that: 
\begin{itemize} 
\item $T$ is a connected tree. Let $C_0(T)$ and $C_1(T)$ be the set of all vertices and edges of $T$, respectively. 
\item For each $v \in C_0(T)$ we fix a cyclic order of the set of edges containing $v$ (ribbon structure). 
\item $C_0(T)$ is divided into the set of \textit{exterior} vertices $C_{0,\exterior}(T)$ and the set of \textit{interior} vertices $C_{0, \interior}(T)$. 
For every $v \in C_{0, \interior}(T)$, we define $k_v$ to be the valency of $v$ minus $1$. 
\item We fix one element of $C_{0, \exterior}(T)$, which we call the \textit{root}. 
\item The valency of every exterior vertex is $1$. 
\item $B$ is a map from $C_{0, \interior}(T)$ to $H_2(\C^n, L)$. 
For every vertex $v$, either $\omega_n(B(v))>0$ or $B(v)=0$ holds. 
\item Every vertex $v$ with $B(v)=0$ has valency at least $3$.
\end{itemize} 

For every $k \in \Z_{\ge 0}$ and $\beta \in H_2(\C^n, L)$, 
we denote by $\mca{G}(k+1, \beta)$
the set of decorated rooted ribbon trees $(T, B)$ such that:
\begin{itemize}
\item[(I):] $\# C_{0, \exterior}(T) = k+1$.
\item[(II):] $\sum_{v \in C_{0, \interior}(T)} B(v)= \beta$.
\end{itemize} 
An edge is called \textit{exterior} if it contains an exterior vertex. 
Otherwise it is called \textit{interior}. 
$C_{1, \exterior}(T)$ (resp. $C_{1,\interior}(T)$)
denotes the set of exterior (resp. interior) edges. 
\end{defn} 

We also define a natural notion of \textit{reduction} on $\mca{G}(k+1: \beta)$. 

\begin{defn}\label{171014_2} 
Let $(T, B) \in \mca{G}(k+1: \beta)$ and $e \in C_{1, \interior}(T)$, and $v_0$, $v_1$ be vertices of $e$. 
By collapsing $e$ to a new vertex $v_{01}$, we get $(T', B') \in \mca{G}(k+1: \beta)$ such that
\begin{align*} 
C_0(T') &:= (C_0(T) \setminus \{v_0, v_1\}) \cup \{v_{01}\}, \\ 
C_1(T') &:= C_1(T) \setminus \{e\}, \\ 
B'(v)&:= \begin{cases} B(v) &(v \ne v_{01}) \\ B(v_0) + B(v_1) &(v = v_{01}). \end{cases}
\end{align*} 
An element of $\mca{G}(k+1: \beta)$ which can be obtained from $(T, B)$ by repeating 
this process is called a \textit{reduction} of $(T, B)$. 
\end{defn} 

Now let us define moduli spaces $\mca{M}_{k+1}(\beta)$, $\mca{N}^0_{k+1}(\beta)$ and $\mca{N}^{\ge 0}_{k+1}(\beta)$. 
For every $(T, B) \in \mca{G}(k+1, \beta)$, 
one can define 
\begin{equation}\label{170611_1} 
\ev_{\interior}: \prod_{v \in C_{0, \interior}(T)} \mtrg{\mca{M}}_{k_v+1}(B(v)) \to \prod_{e \in C_{1,\interior}(T)} L^2
\end{equation}
in the same manner as the definition of the
map (21.2) in \cite{FOOO_Kuranishi}. 
We also consider 
\begin{equation}\label{170611_3}
\ev_{\exterior}: \prod_{v \in C_{0, \interior}(T)} \mtrg{\mca{M}}_{k_v+1}(B(v)) \to \prod_{e \in C_{1,\exterior}(T)} L \cong L^{k+1}
\end{equation}
where the isomorphism on the right is defined
by labeling exterior vertices by $\{0, \ldots, k\}$ in positive cyclic order so that 
the root is labeled by $0$. 

We also consider the diagonal map: 
\begin{equation}\label{170611_2}
\Delta: \prod_{e \in C_{1,\interior}(T)} L  \to \prod_{e \in C_{1,\interior}(T)} L^2; \quad (x_e)_e \mapsto (x_e, x_e)_e. 
\end{equation} 
Then we define $\mca{M}_{k+1}(\beta)$
by taking fiber products of (\ref{170611_1}) and (\ref{170611_2}). 
Namely 
\[
\mca{M}_{k+1}(\beta) := \bigsqcup_{(T, B) \in \mca{G}(k+1, \beta)} 
\bigg(\prod_{e \in C_{1,\interior}(T)} L \bigg)
 \fbp{\Delta}{\ev_{\interior}}
\bigg(\prod_{v \in C_{0, \interior}(T)} \mtrg{\mca{M}}_{k_v+1}(B(v)) \bigg). 
\]
We define $\ev^{\mca{M}} = (\ev^{\mca{M}}_0, \ldots, \ev^{\mca{M}}_k): \mca{M}_{k+1}(\beta )\to L^{k+1}$
by restricting $\ev_\exterior$. 

The definition of  $\mca{N}^0_{k+1}(\beta)$ is similar. 
For any $(T,B)  \in \mca{G}(k+1, \beta)$ 
and $v_0 \in C_{0, \interior}(T)$, 
we define 
\begin{equation}\label{170702_1} 
\ev_{\interior}: 
\prod_{v \in C_{0, \interior}(T) \setminus \{v_0\}} \mtrg{\mca{M}}_{k_v+1}(B(v)) \times \mtrg{\mca{N}}^0_{k_{v_0}+1} (B(v_0))  \to \prod_{e \in C_{1, \interior}(T)} L^2 
\end{equation}
and 
\begin{equation}\label{170702_2}
\ev_{\exterior}: 
\prod_{v \in C_{0, \interior}(T) \setminus \{v_0\}} \mtrg{\mca{M}}_{k_v+1}(B(v)) \times \mtrg{\mca{N}}^0_{k_{v_0}+1} (B(v_0))  \to L^{k+1}
\end{equation}
in manners similar to (\ref{170611_1}) and (\ref{170611_3}). 
Then we define: 
\begin{align*} 
\mca{N}^0_{k+1}(\beta)&:= \bigsqcup_{\substack{(T, B) \in \mca{G}(k+1, \beta) \\ v_0 \in C_{0, \interior}(T)}} \\
&\bigg(\prod_{e \in C_{1,\interior}(T)} L \bigg)
\fbp{\Delta}{\ev_\interior}
\bigg(
\prod_{v \in C_{0, \interior}(T) \setminus \{v_0\}} \mtrg{\mca{M}}_{k_v+1}(B(v)) \times \mtrg{\mca{N}}^0_{k_{v_0}+1} (B(v_0))
\bigg). 
\end{align*} 
We define $\ev^{\mca{N}^0} = (\ev^{\mca{N}^0}_0, \ldots, \ev^{\mca{N}^0}_k): \mca{N}^0_{k+1}(\beta) \to L^{k+1}$ 
by restring $\ev_\exterior$ defined in  (\ref{170702_2}). 
$\mca{N}^{\ge 0}_{k+1}(\beta)$ 
and
$\ev^{\mca{N}^{\ge 0}}: \mca{N}^{\ge 0}_{k+1}(\beta) \to L^{k+1}$
are defined in a similar way. 

Now we have defined \textit{sets}
$\mca{M}_{k+1}(\beta)$, 
$\mca{N}^0_{k+1}(\beta)$, and 
$\mca{N}^{\ge 0}_{k+1}(\beta)$. 
These sets have natural topologies: 
the topology on $\mca{M}_{k+1}(\beta)$ is described in Section 7.1.4 \cite{FOOO_09}
(actually our situation is much simpler, since we have neither interior marked points nor sphere bubbles). 
The topologies on 
$\mca{N}^0_{k+1}(\beta)$ and 
$\mca{N}^{\ge 0}_{k+1}(\beta)$ are defined in similar ways, 
and details are omitted. 

\subsubsection{K-structures on compactified moduli spaces}

In this subsubsection, we consider a system of K-structures on compactified moduli spaces defined in the previous subsubsection. 
More accurately, we consider admissible K-structures (see Section 25 in \cite{FOOO_Kuranishi}), 
and define a certain inductive system of such admissible K-spaces (this is a variant of ``inductive system of $A_\infty$ correspondence'' in Definition 21.17 in \cite{FOOO_Kuranishi}). 
Our goal in this subsubsection is to spell out an accurate statement in Theorem \ref{170611_5}. 

Recall that we took $\ep>0$ so that 
every nonconstant holomorphic disk with boundary on $L$ has area at least $2\ep$. 
We take $U \in \Z_{>0}$ so that 
$\ep (U-1) \ge 2 \| H \|$. 

\begin{thm}\label{170611_5} 
For every $k \in \Z_{\ge 0}$, $m \in \Z_{\ge 0}$ and $P \in \{ \{m\}, [m, m+1]\}$, 
there exist the following data: 
\begin{enumerate}
\item[(i):] \textbf{(Moduli spaces)}
Compact, oriented, admissible K-spaces 
\begin{align*} 
&\mca{M}_{k+1}(\beta: P)  \qquad (\beta \in H_2(\C^n, L), \, \omega_n(\beta) < \ep(m+1-k)), \\
&\mca{N}^0_{k+1}(\beta: P) \qquad (\beta \in H_2(\C^n, L), \, \omega_n(\beta) <\ep(m-1-k)),  \\
&\mca{N}^{\ge 0}_{k+1}(\beta: P)  \qquad (\beta \in H_2(\C^n, L), \, \omega_n(\beta) < \ep(m-k-U)), 
\end{align*}
whose underlying topological spaces are 
\[ 
P \times \mca{M}_{k+1}(\beta), \quad 
P \times \mca{N}^0_{k+1}(\beta), \quad
P \times \mca{N}^{\ge 0}_{k+1}(\beta), 
\] 
respectively. 
Dimensions of these K-spaces are given by 
\begin{align*} 
&\dim \mca{M}_{k+1}(\beta: P) = \mu(\beta) + n+k-2 + \dim P,  \\
&\dim \mca{N}^0_{k+1}(\beta: P) = \mu(\beta) + n+k + \dim P, \\
&\dim \mca{N}^{\ge 0}_{k+1}(\beta: P) = \mu(\beta) + n+k+1 + \dim P. 
\end{align*} 
\item[(ii):]
\textbf{(Evaluation maps)} 
Corner stratified strongly smooth maps 
(from K-spaces to manifolds with corners, see Definition 26.6 (1) in \cite{FOOO_Kuranishi})
\begin{align*}
\ev^{\mca{M}, P}&: \mca{M}_{k+1}(\beta: P) \to P \times L^{k+1},                 \\
\ev^{\mca{N}^0, P}&: \mca{N}^0_{k+1}(\beta: P) \to P \times L^{k+1},           \\
\ev^{\mca{N}^{\ge 0}, P}&:\mca{N}^{\ge 0}_{k+1}(\beta: P) \to P \times L^{k+1}, 
\end{align*} 
such that their underlying set-theoretic maps are: 
\[ 
\id_P \times \ev^{\mca{M}}, \quad
\id_P \times \ev^{\mca{N}^0}, \quad 
\id_P \times \ev^{\mca{N}^{\ge 0}}. 
\]
We require that the following maps to $P \times L$ 
\[ 
(\id_P \times \pr_0) \circ \ev^{\mca{M}, P}, \quad
(\id_P \times \pr_0) \circ \ev^{\mca{N}^0, P}, \quad
(\id_P \times \pr_0) \circ \ev^{\mca{N}^{\ge 0}, P}
\]
are corner stratified weak submersions (see Definition 26.6 (2) in \cite{FOOO_Kuranishi}). 
Here $\pr_0: L^{k+1} \to L$ is defined by $(p_0, \ldots, p_k) \mapsto p_0$. 

\item[(iii):] \textbf{(Energy zero part)}
An isomorphism of admissible K-structures 
\[
\mca{M}_{k+1}(0: P) \cong P \times L \times D^{k-2}
\]
for every $k \ge 2$, 
so that $\ev^{\mca{M}, P}: \mca{M}_{k+1}(0: P) \to P \times L^{k+1}$ 
coincides with $\pr_P \times (\pr_L)^{k+1}$. 
Here $D^{k-2}$ in the RHS is identified with the Stasheff cell; see \cite{Fukaya_Oh} Section 10. 
The coordinate near boundary is 
$1/\log T$, 
where $T$ denotes the length of neck region; 
see Remark 25.45 in \cite{FOOO_Kuranishi}. 

\item[(iv):] \textbf{(Compatibility at boundaries)}
The following isomorphisms of admissible K-spaces preserving orientations
($\partial$ in the LHS denotes normalized boundaries, 
and the fiber product 
$\fbp{\ev_i}{\ev_0}$ is abbreviated as $\fbp{i}{0}$): 
\begin{equation}\label{170903_1} 
\partial \mca{M}_{k+1}(\beta: m) \cong \bigsqcup_{\substack{k_1+k_2=k+1 \\ 1 \le i \le k_1 \\ \beta_1 + \beta_2 = \beta}} (-1)^{\ep_0} 
\mca{M}_{k_1+1}(\beta_1: m) \fbp{i}{0}  \mca{M}_{k_2+1}(\beta_2: m), 
\end{equation}
\begin{eqnarray}\label{170903_2} 
&\partial \mca{N}^0_{k+1}(\beta: m) \cong  \\  
&\bigsqcup_{\substack{k_1+k_2=k+1 \\ 1 \le i \le k_1 \\ \beta_1 + \beta_2 = \beta}}
\biggl( (-1)^{\ep_1}  \mca{N}^0_{k_1+1}(\beta_1: m) \fbp{i}{0} \mca{M}_{k_2+1}(\beta_2: m)  \nonumber \\ 
&\qquad\qquad \sqcup  (-1)^{\ep_2} \mca{M}_{k_1+1}(\beta_1: m) \fbp{i}{0} \mca{N}^0_{k_2+1}(\beta_2: m)  \biggr), \nonumber 
\end{eqnarray} 
\begin{eqnarray}\label{170903_3} 
&\partial \mca{N}^{\ge 0} _{k+1}(\beta: m ) \cong   \mca{N}^0_{k+1}(\beta:m) \sqcup \\ 
&\bigsqcup_{\substack{k_1+k_2=k+1 \\ 1 \le i \le k_1 \\ \beta_1 + \beta_2 = \beta}} 
\biggl( (-1)^{\ep_3} \mca{N}^{\ge 0}_{k_1+1}(\beta_1: m) \fbp{i}{0}  \mca{M}_{k_2+1}(\beta_2: m) \nonumber \\ 
&\qquad\qquad \sqcup  (-1)^{\ep_4} \mca{M}_{k_1+1}(\beta_1: m) \fbp{i}{0}  \mca{N}^{\ge 0}_{k_2+1}(\beta_2: m)  \biggr), \nonumber 
\end{eqnarray} 
\begin{eqnarray}\label{170903_4} 
&\partial \mca{M}_{k+1}(\beta: [m,m+1]) \cong (-1)^{\ep_5} \mca{M}_{k+1}(\beta:m)  \sqcup (-1)^{\ep_6} \mca{M}_{k+1}(\beta: m+1) \, \sqcup \\ 
&\bigsqcup_{\substack{k_1+k_2=k+1 \\ 1 \le i \le k_1 \\ \beta_1 + \beta_2 = \beta}} (-1)^{\ep_7} 
\mca{M}_{k_1+1}(\beta_1: [m, m+1]) \fbp{i}{0}  \mca{M}_{k_2+1}(\beta_2: [m, m+1]),  \nonumber 
\end{eqnarray}
\begin{eqnarray}\label{170903_5} 
&\partial \mca{N}^0_{k+1}(\beta:  [m,m+1] ) \cong (-1)^{\ep_8}  \mca{N}^0_{k+1}(\beta: m) \sqcup (-1)^{\ep_9} \mca{N}^0_{k+1}(\beta: m+1) \, \sqcup \\  
&\bigsqcup_{\substack{k_1+k_2=k+1 \\ 1 \le i \le k_1 \\ \beta_1 + \beta_2 = \beta}}
\biggl( (-1)^{\ep_{10}} \mca{N}^0_{k_1+1}(\beta_1: [m, m+1]) \fbp{i}{0} \mca{M}_{k_2+1}(\beta_2: [m,m+1])  \nonumber \\
&\qquad\qquad \sqcup  (-1)^{\ep_{11}} \mca{M}_{k_1+1}(\beta_1: [m, m+1]) \fbp{i}{0} \mca{N}^0_{k_2+1}(\beta_2: [m, m+1])  \biggr), \nonumber 
\end{eqnarray} 
\begin{eqnarray}\label{170903_6} 
&\partial \mca{N}^{\ge 0} _{k+1}(\beta: [m,m+1]) \cong 
(-1)^{\ep_{12}} \mca{N}^{\ge 0}_{k+1}(\beta: m) \sqcup (-1)^{\ep_{13}} \mca{N}^{\ge 0}_{k+1}(\beta: m+1)  \\ 
&\sqcup (-1)^{\ep_{14}} \mca{N}^0_{k+1}(\beta: [m, m+1]) \sqcup  \nonumber \\   
&\bigsqcup_{\substack{k_1+k_2=k+1 \\ 1 \le i \le k_1 \\ \beta_1 + \beta_2 = \beta}}
\biggl( (-1)^{\ep_{15}} \mca{N}^{\ge 0}_{k_1+1}(\beta_1: [m,m+1] ) \fbp{i}{0} \mca{M}_{k_2+1}(\beta_2: [m,m+1])  \nonumber \\
&\qquad\qquad   \sqcup  (-1)^{\ep_{16}} \mca{M}_{k_1+1}(\beta_1: [m, m+1] ) \fbp{i}{0}  \mca{N}^{\ge 0}_{k_2+1}(\beta_2: [m, m+1])  \biggr). \nonumber 
\end{eqnarray} 
The signs $\ep_0, \ldots, \ep_{16}$ are given as follows: 
\begin{align*} 
\ep_0 &= (k_1-i)(k_2-1) + n+ k_1,  \\
\ep_1 &= \ep_0 +k+ k_1, \, \ep_2 = \ep_0 + k + k_2,  \\
\ep_3 &= \ep_1 + 1, \, \ep_4 = \ep_2 +k_1 + 1, \\ 
\ep_5 &= 1,\, \ep_6 = 0, \, \ep_7 = \ep_0 + 1, \\ 
\ep_8 &= 1, \, \ep_9 = 0, \, \ep_{10} = \ep_1 + 1, \, \ep_{11} = \ep_2 + 1, \\ 
\ep_{12}&= 1, \, \ep_{13} = 0, \, \ep_{14}=1, \, \ep_{15} = \ep_3 + 1, \, \ep_{16} = \ep_4 + 1. 
\end{align*} 
\item[(v):]\textbf{(Compatibility at corners I)}
First we introduce the following notations: 
\begin{itemize} 
\item 
For any admissible K-space $X$ and $l \in \Z_{\ge 1}$, 
let $\wh{S}_l X$ denote the normalized codimension $l$ corner of $X$; 
see \cite{FOOO_Kuranishi} Definition 24.17. 
\item 
For every nonnegative integers $d$ and $m$, we denote 
\[ 
\wh{S}_d \{m\} := \begin{cases} \{m \} &(d=0) \\ \emptyset &(d \ge 1), \end{cases} \qquad 
\wh{S}_d [m, m+1] := \begin{cases} [m, m+1]  &(d=0) \\ \{m, m+1\} &(d=1) \\ \emptyset &(d \ge 2). \end{cases}
\]
\end{itemize} 
Then, there are the following isomorphisms of admissible K-spaces
(\ref{170614_1}), (\ref{170614_2}), and (\ref{170614_3}): 
here we do not consider orientations of moduli spaces. 
\begin{align}\label{170614_1} 
&\wh{S}_l \mca{M}_{k+1}(\beta: P) \cong  \\
&\qquad \bigsqcup_{\substack{(T, B) \in \mca{G}(k+1,\beta) \\ \# C_{1, \interior}(T)  + d = l}}   
\bigg( \prod_{e \in C_{1, \interior}(T)} \wh{S}_d P \times L \bigg) \fbp{\Delta}{\ev_\interior}
\bigg( \prod_{v \in C_{0, \interior}(T)} \mca{M}_{k_v+1}(B(v): \wh{S}_d P) \bigg) \nonumber
\end{align} 
where the fiber product in the RHS is taken over $\prod_{e \in C_{1, \interior}(T)} (\wh{S}_d P \times L)^2$. 
Notice that the fiber product makes sense, since 
\[
(\id_P \times \pr_0) \circ \ev^{\mca{M}, P}: \mca{M}_{k+1}(\beta: P) \to P \times L 
\]
is a corner stratified weak submersion, as we assumed in (ii). 
\begin{align}\label{170614_2} 
&\wh{S}_l \mca{N}^0_{k+1}(\beta: P) \cong 
\bigsqcup_{\substack{(T, B) \in \mca{G}(k+1,\beta) \\ \# C_{1, \interior}(T) + d = l \\ v_0 \in C_{0, \interior}(T)}}
\bigg( \prod_{e \in C_{1, \interior}(T)} \wh{S}_d P \times L \bigg) \fbp{\Delta}{\ev_\interior} \\ 
&\qquad \bigg( \prod_{v \in C_{0, \interior}(T)\setminus \{v_0\}} \mca{M}_{k_v+1}(B(v): \wh{S}_d P)  \times \mca{N}_{k_{v_0}+1}^0(B(v_0): \wh{S}_d P) \bigg).  \nonumber
\end{align} 
\begin{align}\label{170614_3}
&\wh{S}_l \mca{N}^{\ge 0}_{k+1}(\beta: P) \cong 
\bigsqcup_{\substack{(T, B) \in \mca{G}(k+1,\beta) \\ \# C_{1, \interior}(T) + d  =  l \\ v_0 \in C_{0, \interior}(T)}}
\bigg( \prod_{e \in C_{1, \interior}(T)} \wh{S}_d P \times L \bigg) \fbp{\Delta}{\ev_\interior} \\ 
&\qquad \bigg( \prod_{v \in C_{0, \interior}(T)\setminus \{v_0\}} \mca{M}_{k_v+1}(B(v): \wh{S}_d P)  \times \mca{N}_{k_{v_0}+1}^{\ge 0}(B(v_0): \wh{S}_d P) \bigg)   \nonumber  \\
&\qquad \sqcup \, \bigsqcup_{\substack{(T, B) \in \mca{G}(k+1,\beta) \\ \# C_{1, \interior}(T)+d = l-1  \\ v_0 \in C_{0, \interior}(T)}}
\bigg( \prod_{e \in C_{1, \interior}(T)} \wh{S}_d P \times L \bigg) \fbp{\Delta}{\ev_\interior} \nonumber \\ 
&\qquad \bigg( \prod_{v \in C_{0, \interior}(T)\setminus \{v_0\}} \mca{M}_{k_v+1}(B(v): \wh{S}_d P)  \times \mca{N}_{k_{v_0}+1}^0(B(v_0): \wh{S}_d P) \bigg).  \nonumber
\end{align} 
\item[(vi):]\textbf{(Compatibility at corners II)}
Let $X$ be either $\mca{M}_{k+1}(\beta: P)$, $\mca{N}^0_{k+1}(\beta: P)$ or $\mca{N}^{\ge 0}_{k+1}(\beta: P)$. 
Then, for every $l, l' \in \Z_{\ge 1}$, 
the canonical covering map 
$\pi_{l', l}: \wh{S}_{l'}(\wh{S}_l X) \to \wh{S}_{l+l'} X$
(see Proposition 24.16 in \cite{FOOO_Kuranishi})
coincides with the map defined from the fiber product presentation in (v). 
\end{enumerate} 
\end{thm} 

Construction of these K-structures are mostly the same as the 
construction of K-structure on $\mca{M}_{k+1}(\beta)$, 
which is explained in \cite{FOOO_09} Section 7.1. 
Evaluation maps in (ii) 
and isomorphisms in (iv), (v) naturally follow from the construction of these K-structures. 
Admissibility of this K-structure is a consequence of the exponential decay estimate in \cite{FOOO_Gluing_2016}; see Section 25.5 in \cite{FOOO_Kuranishi}. 
We do not spell out a detailed proof of Theorem \ref{170611_5}. 
Nevertheless, in the rest of this subsubsection, we explain: 
\begin{itemize} 
\item Explicit description of Kuranishi charts of $\mca{M}_{k+1}(\beta: P)$, which we use in Section 7.4. 
\item Total order on the set of moduli spaces, which we need to carry out inductive argument. 
\item Orientations: how signs $\ep_0, \ldots, \ep_{16}$ are computed. 
\end{itemize} 

\textbf{Explicit description of Kuranishi charts of $\mca{M}_{k+1}(\beta: P)$.}

Let $k \in \Z_{\ge 0}$ and $\beta \in H_2(\C^n, L)$. 
Let $\mca{MM}_{k+1}(\beta)$ denote the set consists of tuples $(u, z_0, z_1, \ldots, z_k)$ such that 
$u: (D, \partial D) \to (\C^n, L)$ is a $C^\infty$-map such that $\bar{\partial} u = 0$ on a \textit{neighborhood of $\partial D$}
and $[u] = \beta$, and $z_0, z_1, \ldots, z_k$ are distinct points on $\partial D$ aligned in the anti-clockwise order. 
\begin{rem} 
To define K-structure on $\mca{M}_{k+1}(\beta)$ we consider an obstruction bundle $E$ and a perturbed Cauchy-Riemann equation $\bar{\partial}u \in E$. 
One can take $E$ so that every section of $E$ is supported on a compact set of $\text{int} D$, 
thus these perturbed holomorphic maps (with boundary marked points) are elements in $\mca{MM}_{k+1}(\beta)$. 
\end{rem} 
Now we can state the explicit description of Kuranishi chart in Lemma \ref{171014_1} below. 

\begin{lem}\label{171014_1} 
Let $p \in \mca{M}_{k+1}(\beta: P)$ and $\mca{U}_p = (U_p, \mca{E}_p, s_p, \psi_p)$ be a K-chart at $p$. 
Let $(T, B)$ be an element in $\mca{G}(k+1: \beta)$ such that 
\[ 
p \in  P \times \bigg(\prod_{e \in C_{1,\interior}(T)} L \bigg)
 \fbp{\Delta}{\ev_{\interior}}
\bigg(\prod_{v \in C_{0, \interior}(T)} \mtrg{\mca{M}}_{k_v+1}(B(v)) \bigg). 
\] 
Then, $U_p$ can be (set theoretically) embedded into 
\[ 
\bigsqcup_{(T', B')}
P \times \bigg(\prod_{e \in C_{1,\interior}(T')} L \bigg)
 \fbp{\Delta}{\ev_{\interior}}
\bigg(\prod_{v \in C_{0, \interior}(T')} \mca{MM}_{k_v+1}(B'(v)) \bigg) 
\] 
where $(T', B')$ runs over all reductions (see Definition \ref{171014_2})  of $(T, B)$. 
\end{lem} 

\textbf{Total order on the set of moduli spaces.}
To achieve compatibility conditions (iv), (v) and (vi) in Theorem \ref{170611_5}, 
we need a total order on the set of moduli spaces, 
so that the following property is satisfied: 
if a moduli space $X$ is a part of the normalized boundary $\partial Y$ of 
another moduli space $Y$, then $X<Y$. 
To achieve this property we take a total order 
which is described as follows: 
\begin{itemize} 
\item If the $P$-part of a moduli space $X$ is $\{m\}$ and the $P$-part of a moduli space $Y$ is $[m', m'+1]$, then $X<Y$ 
($m$ and $m'$ are arbitrary elements in $\Z_{\ge 0}$). 
\item If the $P$-part of $X$ is $\{m\}$ and the $P$-part of $Y$ is $\{m+1\}$, then $X<Y$. 
\item If the $P$-part of $X$ is $[m, m+1]$ and the $P$-part of $Y$ is $[m+1, m+2]$, then $X<Y$. 
\item For each $m \in \Z_{\ge 0}$, we take a total order on moduli spaces with $P$-parts equal to $\{m\}$, in the following manner: 
\begin{itemize} 
\item For any $k, k', k'' \in \Z_{\ge 0}$ and $\beta, \beta', \beta'' \in H_2(\C^n, L)$, 
\[ 
\mca{M}_{k+1}(\beta: m) < \mca{N}^0_{k'+1}(\beta': m) < \mca{N}^{\ge 0}_{k''+1}(\beta'': m). 
\]
\item If $k, k' \in \Z_{\ge 0}$ and $\beta, \beta' \in H_2(\C^n, L)$ satisfy 
$\omega_n(\beta) + \ep(k-1) < \omega_n(\beta') + \ep(k'-1)$, then 
$\mca{M}_{k+1}(\beta: m) < \mca{M}_{k'+1}(\beta': m)$, 
$\mca{N}^0_{k+1}(\beta: m) < \mca{N}^0_{k'+1}(\beta': m)$, and 
$\mca{N}^{\ge 0}_{k+1}(\beta: m) < \mca{N}^{\ge 0}_{k'+1}(\beta': m)$. 
\item For each $c \in \R$, 
we choose arbitrary total orders on sets 
\begin{align*} 
&\{ \mca{M}_{k+1}(\beta: m) \mid \omega_n(\beta) + \ep(k-1) = c\},  \\ 
&\{ \mca{N}^0_{k+1}(\beta: m) \mid \omega_n(\beta) + \ep(k-1) = c\}, \\ 
&\{ \mca{N}^{\ge 0}_{k+1}(\beta: m) \mid \omega_n(\beta) + \ep(k-1) = c\}. \\ 
\end{align*} 
\end{itemize} 
\item For each $m \in \Z_{\ge 0}$, we take a total order on moduli spaces with $P$-parts equal to $[m, m+1] $, 
in a manner similar to above. 
\end{itemize} 

\begin{rem}\label{171206_2} 
For each moduli space $X$, 
there are only finitely many moduli spaces which are smaller than $X$. 
Therefore, we can assign $\tau(X) \in (1/2, 1)$ for each moduli space $X$, 
so that $Y<X \implies \tau(Y) > \tau(X)$. 
We use these numbers in Section 7.6. 
\end{rem} 

\textbf{Orientations.}
We explain how signs $\ep_0, \ldots, \ep_{16}$ are computed. 

$\ep_0$: First we orient $\mca{M}_{k+1}(\beta)$ following Section 8.3 in \cite{FOOO_09}. 
It is sufficient to orient its interior $\mtrg{\mca{M}}_{k+1}(\beta) = \{(u, z_0, \ldots, z_k)\}/\Aut(D)$. 
We consider the set $\{(u, z_0, \ldots, z_k)\}$ as an open subset of 
$\mtrg{\mca{M}}(\beta) \times (\partial D)^{k+1}$, where 
\[
\mtrg{\mca{M}}(\beta):= \{ u: (D, \partial D) \to (\C^n, L) \mid \bar{\partial} u=0, \, [u]=\beta\}
\]
is canonically oriented by Theorem 8.1.1 in \cite{FOOO_09} and 
$\partial D = \{ e^{i \theta}  \mid \theta \in \R/2\pi \Z\}$ is oriented so that $\partial_\theta$ is of positive direction
(anti-clockwise orientation). 
On the other hand, 
following Convention 8.3.1 in \cite{FOOO_09}, 
$\Aut (D)$ is oriented so that the diffeomorphism 
\[ 
\Aut (D) \to (\partial D)^3; \, \rho \mapsto (\rho^{-1}(1), \rho^{-1}(e^{2\pi i/3}) , \rho^{-1}(e^{4 \pi i/3}))
\] 
is orientation preserving. 
Finally, the quotient $\mtrg{\mca{M}}_{k+1}(\beta)$ is oriented so that a natural isomorphism (determined uniquely up to homotopy) 
\[ 
T(\mca{M}_{k+1}(\beta)) \oplus T(\Aut(D)) \cong T(\mca{M}(\beta)) \oplus  T(\partial D)^{\oplus  k+1}
\]
preserves orientations (see page 692 in \cite{FOOO_09}). 

Now let us compute $\ep_0$. 
When $i=1$, Proposition 8.3.3 in \cite{FOOO_09} shows $\ep_0 = (k_1-1)(k_2-1)+n+k_1$. 
For arbitrary $i \in \{1, \ldots, k_1\}$, 
we can compute $\ep_0$ as 
\[
\ep_0 = (k_1-1)(k_2-1) + n+ k_1 + (i-1) + (i-1) k_2 = (k_1-i)(k_2-1) + n+ k_1
\] 
where the term $(i-1)$ comes from exchanging input boundary marked points $1, \ldots, i-1$ and $i$, 
and the term $(i-1)k_2$ comes from exchanging output boundary marked points $1, \ldots, i-1$ and $i, \ldots, i+k_2-1$. 

$\ep_1$ and $\ep_2$: 
We orient $\mtrg{\mca{N}}^0_{k+1}(\beta) = \{ (u, z_0=1, z_1, \ldots, z_k)\}$ 
as an open subset of $\mtrg{\mca{M}}(\beta) \times (\partial D)^k$. 
Then $\ep_1 = \ep_0 + k+k_1$ and $\ep_2 = \ep_0 +k+ k_2$ follow from the next lemma. 
Following \cite{FOOO_09} page 694, 
we say that an isomorphism of K-spaces is $(-1)$-oriented (resp. $(+1)$-oriented) 
if the isomorphism reverses (resp. preserves) orientations. 

\begin{lem} 
Let us orient $\Aut(D, 1)$ so that the embedding 
\[ 
\Aut(D, 1) \to (\partial D)^2 ; \, \rho \mapsto (\rho^{-1}(e^{2\pi i/3}), \rho^{-1}(e^{4\pi i/3})) 
\]
preserves orientations. 
Then, the isomorphism of K-spaces 
\[
\mtrg{\mca{M}}_{k+1}(\beta) \cong \mtrg{\mca{N}}^0_{k+1}(\beta)/\Aut (D,1)
\]
is $(-1)^k$-oriented. 
\end{lem}
\begin{proof}
Recall that $\mtrg{\mca{M}}_{k+1}(\beta) =\{(u, z_0, \ldots, z_k)\}/\Aut (D)$. 
The natural diffeomorphism $\Aut(D) \cong \partial D \times \Aut(D, 1)$ preserves orientations, 
and the first factor $\partial D$ acts on the $0$-th marked point $z_0$, 
thus the factor $(-1)^k$ comes from exchanging $z_0$ and $z_1, \ldots, z_k$. 
\end{proof} 

$\ep_3$ and $\ep_4$: 
$\mca{N}^{\ge 0}_{k+1}(\beta)$ can be identified (on a neighborhood of $\mca{N}^0_{k+1}(\beta)$) with 
$[0, 1) \times \mca{N}^0_{k+1}(\beta)$.
Following Convention 8.2.1 in \cite{FOOO_09}, 
we orient $[0,1)$ so that $\partial/\partial t$ is of \textit{negative} direction, 
where $t$ denotes the canonical coordinate on $\R$. 
Then 
\[ 
\partial ([0, 1) \times \mca{N}^0) = \{0\} \times \mca{N}^0 \sqcup (-1) \cdot  [0,1) \times \partial \mca{N}^0
\] 
and $\ep_3 = \ep_1 + 1$ immediately follows. 
$\ep_4 = \ep_2 + k_1 +1$ follows from 
\[ 
[0,1) \times (\mca{M}_{k_1+1}(\beta_1) \fbp{i}{0} \mca{N}^0_{k_2+1}(\beta_2))
= (-1)^\delta \mca{M}_{k_1+1}(\beta_1) \fbp{i}{0} ( [0,1) \times \mca{N}^0_{k_2+1}(\beta_2) )
\]
where $\delta = \dim \mca{M}_{k_1+1}(\beta_1) - n \equiv k_1 \,(\text{mod}\, 2)$. 

$\ep_5, \ldots, \ep_{16}$: 
$\mca{M}_{k+1}(\beta:[m, m+1])$ and $[m, m+1] \times \mca{M}_{k+1}(\beta)$ 
are different K-spaces with a common underlying topological space. 
One can take a common open substructure
(see Definition 3.20 in \cite{FOOO_Kuranishi}) 
of these two K-spaces. 
Thus, to fix orientations, one can identify these two K-spaces. 
The same remark applies to $\mca{N}^0_{k+1}(\beta: [m, m+1])$ and $\mca{N}^{\ge 0}_{k+1}(\beta: [m, m+1])$. 
Now let us orient $[m, m+1]$ so that $\partial/\partial t$ is of \textit{positive} direction, 
where $t$ denotes the canonical coordinate on $\R$. 
Then 
\[
\partial ([m, m+1]) = (+1) \cdot \{ m+1\}  \sqcup (-1) \cdot \{m\}. 
\] 
Then $\ep_5, \ldots, \ep_{16}$ can be determined from 
$\ep_0, \ldots, \ep_4$ using the following formulas ($X$ and $Y$ denote arbitrary K-spaces): 
\begin{align*} 
 \partial ([m, m+1] \times X)&= \{m+1\} \times X \sqcup  (-1) \cdot \{m\} \times X \sqcup (-1) \cdot  [m, m+1] \times \partial X, \\ 
[m, m+1] \times ( X \times_L Y) &= ([m, m+1] \times X) \times_{[m, m+1] \times L} ([m, m+1] \times Y). 
\end{align*} 

\subsection{Spaces of continuous paths and loops} 

Let $\Pi^\con$ denote the set of continuous Moore paths on $L$, namely: 
\[
\Pi^{\con}:= \{ (T, \gamma) \mid T \in \R_{\ge 0}, \, \gamma \in C^0([0,T], L) \}.
\]
For each $q \in L$, let $c_q$ denote the constant path at $q$ with length $0$. 
Namely, $c_q = (0, \gamma_q)$
where $\gamma_q$ denotes the unique map from $\{0\}$ to $q$. 

To define a metric $d_\Pi$ on $\Pi^\con$, 
we fix (throughout this paper) an auxiliary Riemannian metric on $L$, 
and $d_L$ denotes the associated metric on $L$. 
For later use we fix a constant $\rho_L \in \R_{>0}$ 
which is smaller than the injectivity radius of $L$
such that for any $r \in (0, \rho_L]$ and $x \in L$ 
the ball with center $x$ and radius $r$ is geodesically convex. 

Now let us define a metric $d_\Pi$ on $\Pi^\con$ by 
\[ 
d_\Pi((T, \gamma), (T', \gamma'))) := \max \big\{ |T-T'|, \,  \max_{0 \le s \le 1} d_L(\gamma(sT), \gamma'(sT')) \big\}. 
\]
We define $\ev_0, \ev_1: \Pi^\con \to L$ by 
\[
\ev_0 (T, \gamma):= \gamma(0), \qquad \ev_1(T, \gamma):= \gamma(T). 
\] 
Then we can define a concatenation map
\[ 
\Pi^\con \fbp{\ev_1}{\ev_0} \Pi^\con \to \Pi^\con; \qquad (\Gamma_0, \Gamma_1) \mapsto \Gamma_0*\Gamma_1
\] 
by $(T, \gamma) * (T', \gamma'):= (T+T', \gamma*\gamma')$, 
where 
$\gamma*\gamma': [0,T+T'] \to L$ is defined by 
\[
(\gamma*\gamma')(t):= \begin{cases} \gamma(t) &(0 \le t \le T), \\ \gamma'(t-T) &(T \le t \le T+T').  \end{cases}
\]

\begin{lem}\label{170611_4} 
The concatenation map defined as above is continuous with respect to $d_\Pi$. 
\end{lem}
\begin{proof}
We have to show that, if sequences 
$(T_j, \gamma_j)_{j \ge 1}$ and 
$(T'_j, \gamma'_j)_{j \ge 1}$ satisfy 
\[
\lim_{j \to \infty} d_\Pi( (T_j, \gamma_j), (T, \gamma)) = \lim_{j \to \infty} d_\Pi((T'_j, \gamma'_j), (T', \gamma')) = 0
\]
then 
\begin{equation}\label{170414_1} 
\lim_{j \to \infty}  d_\Pi( (T_j + T'_j, \gamma_j * \gamma'_j), (T+T', \gamma*\gamma')) = 0.
\end{equation} 
If (\ref{170414_1}) is not the case, there exists a sequence $(s_j)_{j \ge 1}$ in $[0,1]$ and $\ep>0$ such that 
\begin{equation}\label{170414_2}
d_L((\gamma_j*\gamma'_j)(s_j(T_j+T'_j)),  (\gamma*\gamma')(s_j(T+T')))  \ge \ep
\end{equation}
for every $j$. 
Then $T_j, T'_j >0$ for sufficiently large $j$, thus we may assume that 
$T_j, T'_j>0$ holds for every $j \ge 1$. 

By taking a subsequence, 
we may also assume that either $s_j \le T_j/(T_j+T'_j)$ or $s_j \ge T_j/(T_j+T'_j)$ holds for all $j$. 
In the following we only consider the case $s_j \le T_j/(T_j+T'_j)$, since the proof in the other case is parallel. 
Now we obtain 
\[
(\gamma_j*\gamma'_j) ( s_j (T_j+T'_j)) = \gamma_j \biggl( s_j \cdot \frac{T_j+T'_j}{T_j} \cdot T_j \biggr).
\]
We set $\tau_j:= s_j (T_j+T'_j)/T_j \in [0,1]$. Then the LHS of (\ref{170414_2}) is bounded from above by 
\begin{equation}\label{170422_1}
d_L(\gamma_j(\tau_j T_j), \gamma(\tau_j T)) + d_L(\gamma(\tau_j T), (\gamma*\gamma')(s_j (T+T'))). 
\end{equation} 
Then the first term of (\ref{170422_1}) goes to $0$ as $j \to \infty$, since 
\[
d_L(\gamma_j(\tau_j T), \gamma(\tau_j T)) \le d_\Pi ((T_j, \gamma_j), (T, \gamma)). 
\]
Since $\gamma*\gamma'$ is (uniformly) continuous, 
to show that 
the second term of (\ref{170422_1}) goes to $0$ as $j \to \infty$, 
it is sufficient to check: 
\begin{equation}\label{170422_2} 
\lim_{j \to \infty} \tau_j T - s_j (T+T') = 0.
\end{equation}
Using the obvious identity 
\begin{equation}\label{170613_1} 
\tau_j T - s_j(T+T') = s_j (T_j+T'_j) (T/T_j)  - s_j (T+T'), 
\end{equation} 
we can check (\ref{170422_2}) by considering the following three cases: 
\begin{itemize} 
\item $T>0$: this case $\lim_{j \to \infty} T/T_j = 1$, thus the RHS of (\ref{170613_1}) goes to $0$. 
\item $T=0$ and $T'>0$: this case the first term is $0$, and the second term goes to $0$ because 
 $0 \le s_j \le T_j/(T_j+T'_j)$ implies $\lim_{j \to \infty} s_j = 0$. 
\item $T,T'=0$: this case both the first and second terms are $0$.  
\end{itemize} 
This completes the proof. 
\end{proof}

Next we consider the space of continuous Moore loops with marked points. 
For every $k \in \Z_{\ge 0}$, 
let $\mca{L}_{k+1}^\con$ denote the set consists of 
$(T, \gamma, t_0, \ldots, t_k)$ such that 
\begin{itemize}
\item $T \in \R_{\ge 0}$ and $\gamma \in C^0([0, T], L)$ such that $\gamma(0)=\gamma(T)$. 
\item $0 = t_0 \le t_1 \le \ldots \le t_k \le T$.
\end{itemize} 
For each $a \in H_1(L: \Z)$, 
let $\mca{L}_{k+1}^\con(a)$ denote the subspace of 
$\mca{L}_{k+1}^\con$ which consists of 
$(T, \gamma, t_0, \ldots, t_k)$ such that $[\gamma]=a$.
The set $\mca{L}_{k+1}^\con$ can be identified with the set 
\[ 
\{ (\Gamma_0, \ldots, \Gamma_k)  \in (\Pi^\con)^{k+1} \mid \ev_1(\Gamma_i) = \ev_0(\Gamma_{i+1}) \,(0 \le i \le k-1), \, \ev_1(\Gamma_k) = \ev_0(\Gamma_0)\}. 
\] 
Using this identification, a metric $d_{\mca{L}_{k+1}}$ on $\mca{L}_{k+1}^\con$ is defined by 
\[ 
d_{\mca{L}_{k+1}}((\Gamma_0, \ldots, \Gamma_k), (\Gamma'_0, \ldots, \Gamma'_k)) := \max_{0 \le i \le k} d_{\Pi} (\Gamma_i, \Gamma'_i). 
\] 
We consider the topology on $\mca{L}_{k+1}^\con$ induced from $d_{\mca{L}_{k+1}}$. 

The evaluation map 
\[ 
\evl_j: \mca{L}_{k+1}^\con \to  L \quad (j \in \{0, \ldots, k\})
\] 
and the concatenation map 
\[ 
\con_j: 
\mca{L}_{k+1}^\con \fbp{\evl_j}{\evl_0} \mca{L}_{k'+1}^\con \to \mca{L}_{k+k'}^\con \quad (j \in \{1, \ldots, k\})
\] 
are defined in the same way as in the case of smooth loops (see Section 4.1). 
The evaluation map is obviously continuous, 
and the concatenation map is continuous by Lemma \ref{170611_4}. 
The evaluation map $\evl_j$ will be abbreviated as $\ev_j$ when there is no risk of confusion.

\subsection{Strongly continuous maps to $\mca{L}_{k+1}^\con$}

The goal of this subsection is to define
strongly continuous maps 
from moduli spaces of (perturbed) pseudo-holomorphic disks (with boundary marked points) 
to spaces of continuous loops (with marked points), 
so that natural compatibility conditions (spelled out in Proposition \ref{170613_2}) are satisfied. 
Throughout this subsection, the space $\mca{L}_{k+1}^\con$ is equipped with the topology defined from the metric $d_{\mca{L}_{k+1}}$, 
which is defined in Section 7.3. 

To state Proposition \ref{170613_2} we have to introduce the following notations. 
For any $\beta \in H_2(\C^n, L)$ and $(T, B) \in \mca{G}(k+1: \beta)$, 
one can define a continuous map 
\begin{equation}
\ev_{\interior}: \prod_{v \in C_{0, \interior}(T)} P \times \mca{L}_{k_v+1}^\con (\partial B(v)) \to \prod_{e \in C_{1, \interior}(T)} (P \times L)^2 
\end{equation} 
in a way similar to (\ref{170611_1}). 
Also let us recall the diagonal map (\ref{170611_2}). 
Then one can define a continuous map 
\begin{equation}\label{170615_1} 
\biggl(\prod_{e \in C_{1, \interior}(T)} P \times L \biggr) \fbp{\Delta}{\ev_{\interior}} \biggl( \prod_{v \in C_{0, \interior}(T)} P \times \mca{L}_{k_v+1}^\con (\partial B(v)) \biggr)
\to P \times \mca{L}_{k+1}^\con(\partial \beta) 
\end{equation}
by taking concatenations of loops, 
where the fiber product in the LHS is taken over $\prod_{e \in C_{1,\interior}(T)} (P \times L)^2$. 

\begin{prop}\label{170613_2}
For every $k \in \Z_{\ge 0}$, $m \in \Z_{\ge 0}$, and $P \in \{ \{m\}, [m, m+1]\}$, 
one can define strongly continuous maps
\begin{equation}\label{170705_1} 
\ev^{\mca{M}}:  \mca{M}_{k+1}(\beta : P) \to P \times \mca{L}_{k+1}^\con(\partial \beta)  \qquad (\omega_n(\beta) < \ep (m+1-k)), 
\end{equation}
\begin{equation}\label{170705_2} 
\ev^{\mca{N}^0}: \mca{N}^0_{k+1}(\beta : P) \to P \times \mca{L}_{k+1}^\con(\partial \beta) \qquad(\omega_n(\beta) < \ep(m-1-k)), 
\end{equation}
\begin{equation}\label{170705_3} 
\ev^{\mca{N}^{\ge 0}}: \mca{N}^{\ge 0}_{k+1}(\beta: P) \to P \times \mca{L}_{k+1}^\con(\partial \beta) \qquad(\omega_n(\beta) < \ep(m-k-U)), 
\end{equation}
so that the following diagrams commute 
for every $(T, B) \in \mca{G}(k+1, \beta)$: 
\begin{itemize}
\item 
\[
\xymatrix{
\biggl(\prod_e P \times L \biggr) \fbp{\Delta}{\ev_\interior} \biggl( \prod_v \mca{M}_{k_v+1}(B(v): P) \biggr)
\ar[r]\ar[d] & \mca{M}_{k+1}(\beta: P) 
\ar[d] \\
\biggl(\prod_e P \times L \biggr) \fbp{\Delta}{\ev_\interior}\biggl( \prod_v P \times \mca{L}_{k_v+1}^\con(\partial B(v))  \biggr)
 \ar[r]_-{(\ref{170615_1})}& P \times \mca{L}_{k+1}^\con(\partial \beta) 
\\
}
\]
where the first horizontal map is defined from (\ref{170614_1}) by setting $d=0$, and 
vertical maps are defined by (\ref{170705_1}). 
\item 
\[
\xymatrix{
\biggl(\prod_e P \times L \biggr) \fbp{\Delta}{\ev_\interior}  (\star)
\ar[r] \ar[d] & \mca{N}^0_{k+1}(\beta: P) 
\ar[d] \\
\biggl(\prod_e P \times L \biggr) \fbp{\Delta}{\ev_\interior}\biggl( \prod_v P \times \mca{L}_{k_v+1}^\con(\partial B(v))  \biggr)
 \ar[r]_-{ (\ref{170615_1})} & P \times \mca{L}_{k+1}^\con(\partial \beta) 
\\
}
\]
where $(\star):= \prod_{v \ne v_0} \mca{M}_{k_v+1} (B(v): P)  \times \mca{N}^0_{k_{v_0}+1}(B(v_0): P)$ and 
the first horizontal map is defined from (\ref{170614_2}) by setting $d=0$, and 
vertical maps are defined by (\ref{170705_2}). 
\item 
\[
\xymatrix{
\biggl(\prod_e P \times L \biggr) \fbp{\Delta}{\ev_\interior} (\star)
\ar[r] \ar[d] & \mca{N}^{\ge 0}_{k+1}(\beta: P) 
\ar[d] \\
\biggl(\prod_e P \times L \biggr) \fbp{\Delta}{\ev_\interior}\biggl( \prod_v P \times \mca{L}_{k_v+1}^\con(\partial B(v))  \biggr)
 \ar[r]_-{(\ref{170615_1})} & P \times \mca{L}_{k+1}^\con(\partial \beta)
\\
}
\]
where $(\star) =\prod_{v \ne v_0} \mca{M}_{k_v+1} (B(v): P)  \times \mca{N}^{\ge 0}_{k_{v_0}+1}(B(v_0): P) $ and 
the first horizontal map is defined from (\ref{170614_3}) by setting $d=0$, and 
vertical maps are defined by  (\ref{170705_3}). 
\end{itemize} 
\end{prop}

In the rest of this subsection, 
we explain the definition of 
the strongly continous map (\ref{170705_1}). 
Definitions of 
(\ref{170705_2}) and (\ref{170705_3}) are similar. 
For each $p \in \mca{M}_{k+1}(\beta: P)$, 
let $\mca{U}_p = (U_p, \mca{E}_p, s_p, \psi_p)$
be a K-chart at $p$. 
To define a strongly continuous map (\ref{170705_1}), 
it is sufficient to define a continuous map 
$\ev^{\mca{M}}_p: U_p \to P \times \mca{L}_{k+1}^\con$
for each $p$, 
so that compatibility conditions with coordinate changes are satisfied. 

By Lemma \ref{171014_1}, every $x \in U_p$ is represented by an element of 
\[ 
P \times \bigg(\prod_{e \in C_{1,\interior}(T')} L \bigg)
 \fbp{\Delta}{\ev_{\interior}}
\bigg(\prod_{v \in C_{0, \interior}(T')} \mca{MM}_{k_v+1}(B'(v)) \bigg) 
\] 
where $(T', B')$ is a \textit{reduction} of $(T, B)$. 
We denote the representative by 
\[
(\pi , (u^v, z^v_0, \ldots, z^v_{k_v})_v). 
\] 
Now we define 
\[ 
\ev^{\mca{M}}_p(x):= (\pi, \ev( (u^v, z^v_0, \ldots, z^v_{k_v})_v)) \in P \times \mca{L}_{k+1}^\con(\partial \beta), 
\] 
where $\ev ((u^v, z^v_0, \ldots, z^v_{k_v})_v)$ is the concatenation (defined by (\ref{170615_1})) 
of $\ev(u^v, z^v_0, \ldots, z^v_{k_v}) \in \mca{L}_{k_v+1}^\con(\partial B'(v))$ (which is defined below) for all $v \in C_{0, \interior}(T')$. 

To define $\ev(u^v, z^v_0, \ldots, z^v_{k_v})$, 
we distinguish the case $u^v$ is constant and the case $u^v$ is nonconstant. 
\begin{itemize}
\item 
The case $u^v$ is a constant map to $q \in L$: we define 
\[ 
\ev(u^v, z^v_0, \ldots, z^v_{k_v}) := (c_q, \ldots, c_q)
\] 
where $c_q$ denotes the constant path at $q$ with length $0$. 
\item 
The case $u^v$ is nonconstant: we use the length parametrization. 
For every $j \in \{0,\ldots, k_v\}$ we define $\theta_j \in [0, 2\pi)$ by $z^v_j / z^v_0= e^{\sqrt{-1}  \theta_j}$, thus $0=\theta_0 < \theta_1 < \cdots < \theta_{k_v} < 2\pi$. 
Then we define
\begin{align*} 
T^{u^v} _j &:= \int_{\theta_j}^{\theta_{j+1}} \bigg\lvert \frac{d}{d\theta} u^v(z^v_0 \cdot e^{\sqrt{-1} \theta}) \bigg\rvert \, d \theta, \\
\lambda^{u^v} _j &:  [\theta_j, \theta_{j+1}] \to [0, T^{u^v} _j] ;\quad \Theta \mapsto \int_{\theta_j}^\Theta \bigg\lvert \frac{d}{d\theta} u^v(z^v_0 \cdot e^{\sqrt{-1} \theta}) \bigg\rvert \, d \theta, \\ 
\gamma^{u^v} _j &: [0, T^{u^v} _j] \to L; \quad  t \mapsto u^v (e^{\sqrt{-1} (\lambda^{u^v} _j)^{-1}(t)}). 
\end{align*} 
$\frac{d}{d \theta} u^v (e^{\sqrt{-1}\theta})$ vanishes at only finitely many $\theta$
since $u^v$ is nonconstant, 
thus $T^{u^v} _j$ is positive and $\lambda^{u^v} _j$ is strictly increasing. 
Hence $(\lambda^{u^v} _j)^{-1}: [0, T^{u^v} _j] \to [\theta_j, \theta_{j+1}]$ is a well-defined continuous map. 
Then we define 
\[
\ev(u^v, z^v_0, \ldots, z^v_{k_v} ):= (T^{u^v}_j, \gamma^{u^v}_j)_{0 \le j \le k_v}. 
\]
This completes the definition of $\ev^{\mca{M}}_p(x)$. 
\end{itemize} 
\begin{rem}\label{171014_3} 
For later purpose (proof of Lemma \ref{170615_2}) we define 
$\Lambda^{u^v}_j: [0,1] \to [0, T^{u^v}_j]$ by 
$\Lambda^{u^v}_j(s) := \lambda^{u^v}_j((1-s)\theta_j + s \theta_{j+1})$. 
\end{rem} 

The above definition of $\ev^{\mca{M}}_p(x)$ does not depend on choices of representatives of $x$. 
In particular, the family of maps $(\ev^{\mca{M}}_p)_p$ is compatible with coordinate changes of K-charts. 
To show that this family defines a strongly continuous map, 
we have to check that $\ev^{\mca{M}}_p$ is continuous for each $p$. 

\begin{lem}\label{170615_2} 
$\ev^{\mca{M}}_p: U_p \to P \times \mca{L}_{k+1}^\con(\partial \beta)$ is continuous for every $p \in \mca{M}_{k+1}(\beta: P)$. 
\end{lem} 
\begin{proof}
\textbf{Step 1.} 
Let $(x_l)_{l \ge 1}$ be a sequence in $U_p$ which converges to $x_\infty  \in U_p$. 
For simplicity, we work under the following assumptions: 
\begin{itemize}
\item $k=0$. 
\item Each $x_l$ is represented by a single holomorphic map; we denote the representative by $(\pi_l, u_l, z_{l,0})$. 
\item $x_\infty$ is represented by an element of 
\[ 
P \times \bigg(\prod_{e \in C_{1,\interior}(T_\infty)} L \bigg)
 \fbp{\Delta}{\ev_{\interior}}
\bigg(\prod_{v \in C_{0, \interior}(T_\infty)} \mca{MM}_{k_v+1}(B_\infty(v)) \bigg) 
\] 
for some $(T_\infty, B_\infty) \in \mca{G}(k+1, \beta)$; see Lemma \ref{171014_1}. 
We denote the representative by
$(\pi_\infty, (u^v_\infty , z^v_{\infty, 0}, \ldots, z^v_{\infty,  k_v})_{v \in C_{0, \interior}(T_\infty)})$. 
\end{itemize} 
Since $\lim_{l \to \infty} \pi_l = \pi_\infty$ is obvious, it is sufficient to show that 
\begin{equation}\label{171013_1} 
\lim_{l \to \infty} d_{\mca{L}_1}(\ev(u_l, z_{l,0}), \ev( (u^v_\infty , z^v_{\infty, 0}, \ldots, z^v_{\infty,  k_v})_v)) = 0. 
\end{equation} 

\textbf{Step 2.} 
We first consider the case $\sharp C_{0, \interior}(T_\infty)=1$, 
and abbreviate $u^v_\infty$ as $u_\infty$. 
We may assume that 
$\lim_{l \to \infty} u_l = u_\infty$ in the $C^\infty$-topology, and 
$z_{l,0} = 1$ for every $l$ (including $l= \infty$). 
When $u_\infty$ is constant, then $u_l|_{\partial D}$ converges to a constant map in $C^\infty$ (in particular $C^0$) topology, 
then (\ref{171013_1}) follows. 
When $u_\infty$ is nonconstant, 
then $u_l$ is nonconstant for sufficiently large $l$. 
Moreover 
$\lim_{l \to \infty} T^{u_l}_0 = T^{u_\infty}_0$
and  
$\lim_{l \to \infty} \Lambda^{u_l}_0 = \Lambda^{u_\infty}_0$ in the $C^0$-topology. 
Then (\ref{171013_1}) follows from the next lemma, 
the proof of which is completely elementary and omitted. 

\begin{lem} 
Let $(\tau_l)_l$ be a sequence in $C^0([0,1], [0,1])$ such that 
\begin{itemize}
\item $(\tau_l)_l$ has a $C^0$-limit $\tau_\infty: [0,1] \to [0,1]$.
\item For every $l$ (including $\infty$), $\tau_l$ is strictly increasing and $\tau_l (0)=0$, $\tau_l(1)=1$. 
\end{itemize}
Then $(\tau_l)^{-1}$ converges to $(\tau_\infty)^{-1}$ in the $C^0$-topology. 
\end{lem} 

\textbf{Step 3.} 
Now we consider the case $\sharp C_{0, \interior}(T_\infty) >1$. 
We fix $\ep>0$, which can be arbitrarily small. 
For sufficiently large every $l$, 
there exists a decomposition (depends on $\ep$ but we drop it from the following notations) 
\[ 
 D= V_l \sqcup \bigsqcup_{v \in C_{0, \interior}(T_\infty) } U^v_l 
\] 
such that the following conditions are satisfied (see \cite{FOOO_09} Section 7.1.4): 
\begin{itemize} 
\item $V_l$ is a compact set
with the number of connected components is $\sharp C_{0, \interior}(T_\infty)$, 
and $U^v_l$ is a connected open set for every $v \in C_{0, \interior}(T_\infty)$. 
\item $V_l \cap  \partial D$ is a disjoint union of $2\sharp C_{0, \interior}(T_\infty) - 1$ closed intervals.  
\item $\int_{V_l \cap \partial D} \bigg\lvert  \frac{d}{d\theta} u_l(e^{\sqrt{-1} \theta}) \bigg\rvert < \ep$. 
\item For every  $v \in C_{0, \interior}(T_\infty)$, 
$\overline{U^v_l} \cap \partial D$ is a disjoint union of $k_v+1$ closed intervals.
We denote them as $I^v_l(0), \ldots, I^v_l(k_v)$ aligned in anti-clockwise order. 
\item There exists a sequence $(\rho^v_l)_l$ in $\Aut(D)$ such that $(\rho^v_l)^{-1}(U^v_l)$ is independent on $l$ (denoted by $U^v$), 
and $u_l \circ \rho^v_l$ converges to $u^v_\infty$ (with respect to the $C^\infty$-topology) on $\overline{U^v}$. 
\item $\overline{U^v} \cap \partial D$ is a disjoint union of $k_v+1$ closed intervals, 
which we denote by $I^v_{\infty}(0), \ldots, I^v_{\infty}(k_v)$ in anti-clockwise order. 
\end{itemize}

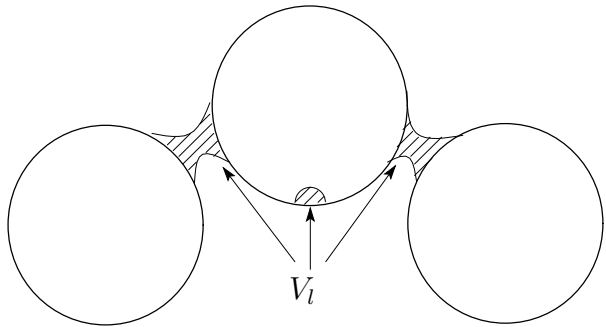
\begin{center}
\input{glueddisks.tpc}
\end{center} 

For each $v \in C_{0, \interior}(T_\infty)$ and $j \in \{0, \ldots, k_v\}$, 
our assumption $u_l \circ \rho^v_l \to u^v_\infty$ 
implies that 
\[ 
\lim_{l \to \infty} d_\Pi (u_l|_{I^v_l(j)} ,  u^v_\infty|_{I^v_\infty(j)}) = 0. 
\] 
The proof is same as our argument in Step 2. 
On the other hand, for every $l$, the total lengths of $u_l|_{V_l \cap \partial D}$ is at most $\ep$. 
Since $\ev(u_l, z_{l,0}) $ is obtained as the concatenation of these paths, 
it is not difficult to see that 
\[ 
\limsup_{l \to \infty} \, d_{\mca{L}_1}(\ev(u_l, z_{l,0}), \ev( (u^v_\infty , z^v_{\infty, 0}, \ldots, z^v_{\infty,  k_v})_v)) \le 10 \ep. 
\] 
Since $\ep$ can be arbitrarily small, this implies 
the convergence (\ref{171013_1}). 
\end{proof}

\subsection{$C^0$-approximation lemma and CF-perturbation} 

We first introduce the notion of ``$\ep$-closeness'' 
for strongly continuous maps from a K-space to a metric space: 

\begin{defn}
Let $(X, \wh{\mca{U}})$ be a K-space, 
$(Y, d)$ be a metric space, 
and $\wh{f}, \wh{g}:  (X,\wh{\mca{U}}) \to Y$ be strongly continuous maps.  
For any $\ep>0$, we say that 
$\wh{f}$ and $\wh{g}$ are $\ep$-close, if 
$d(f_p(x), g_p(x)) < \ep$ for every $p \in X$ and $x \in U_p$. 
\end{defn} 

Let us state a key technical result in this subsection, which we call $C^0$-approximation lemma. 
The notion of ``open substrcture'' of a given K-structure is defined in Definition 3.20 in \cite{FOOO_Kuranishi}. 
The constant $\rho_L$ in the statement was introduced in the second paragraph of Section 7.3. 

\begin{thm}\label{170430_2} 
Let $(X, \wh{\mca{U}})$ be a compact K-space and 
$\wh{f}: (X, \wh{\mca{U}}) \to \mca{L}_{k+1}^\con$ be a strongly continuous map such that 
$\ev^{\mca{L}}_j \circ \wh{f}: (X, \wh{\mca{U}}) \to L$ is strongly smooth for every $j \in \{0, \ldots, k\}$. 
Let $Z$ be a closed subset of $X$ and $\wh{g}: (Z, \wh{\mca{U}}|_Z) \to \mca{L}_{k+1}$ 
be a strongly smooth map such that:
\begin{itemize}
\item $\ev^{\mca{L}}_j \circ \wh{g} = \ev^{\mca{L}}_j \circ \wh{f}|_Z$ for every $j \in \{0, \ldots, k\}$.
\item $\wh{g}$ is $\ep$-close to $\wh{f}|_Z$ with respect to $d_{\mca{L}_{k+1}}$. 
\end{itemize} 
If $\ep < \rho_L$, 
there exists an open substructure $\wh{\mca{U}_0}$ of $\wh{\mca{U}}$ 
and a strongly smooth map $\wh{g'}: (X, \wh{\mca{U}_0}) \to \mca{L}_{k+1}$
such that the following conditions hold: 
\begin{itemize} 
\item 
$\wh{g'}$ is $\ep$-close to $\wh{f}|_{\wh{\mca{U}_0}}$. 
\item 
$\evl_j \circ \wh{g'} = \evl_j \circ \wh{f}|_{\wh{\mca{U}_0}}$ for every $j \in \{0, \ldots, k\}$. 
\item 
$\wh{g'} = \wh{g}$ on $\wh{\mca{U}_0}|_Z$. 
\end{itemize} 
\end{thm} 

The proof of Theorem \ref{170430_2} is carried out in Section 9. 
Combining Theorem \ref{170430_2} with results from \cite{FOOO_Kuranishi}, 
we obtain Theorem \ref{170430_1} below. 
In the statement of Theorem \ref{170430_1} we use the following notions from \cite{FOOO_Kuranishi}
without repeating their definitions: 
\begin{itemize} 
\item Thickening of K-structures: see Section 5.2 in \cite{FOOO_Kuranishi}. 
\item Collared versions of K-structures, strongly smooth maps and CF-perturbations: see Section 17.5 in \cite{FOOO_Kuranishi}. 
\end{itemize} 
The assumptions in Theorem \ref{170430_1} 
on K-structures and CF-perturbations are 
similar to Situations 17.43 and 17.57 in \cite{FOOO_Kuranishi}, respectively. 

\begin{thm}\label{170430_1} 
Suppose that we are given the following data: 
\begin{itemize}
\item $k \in \Z_{\ge 0}$, $\tau \in (0, 1)$ and $\ep \in (0, \rho_L)$. 
\item A $\tau$-collared K-space $(X, \wh{\mca{U}})$. 
\item A $\tau$-collared strongly continuous map $\wh{f}: (X, \wh{\mca{U}}) \to \mca{L}_{k+1}^\con$ 
such that $\ev^{\mca{L}}_j \circ \wh{f} : (X, \wh{\mca{U}}) \to L$ is admissible for every $j \in \{0, \ldots, k\}$, 
and $\ev^{\mca{L}}_0 \circ \wh{f}$ is strata-wise weakly submersive. 
\item For every $l \in \Z_{\ge 1}$, 
a $\tau$-collared K-structure $\wh{\mca{U}^+_l}$ on $\wh{S}_l(X)$
which is a thickening of $\wh{\mca{U}}|_{\wh{S}_l(X)}$. 
\item For every $l_1, l_2 \in \Z_{\ge 1}$, 
 $(l_1+l_2)!/(l_1)!(l_2)!$-fold covering of $\tau$-collared $K$-spaces
\[ 
\wh{S}_{l_1}(\wh{S}_{l_2}(X), \wh{\mca{U}^+_{l_2}}) \to 
(\wh{S}_{l_1+l_2}(X), \wh{\mca{U}^+_{l_1+l_2}})
\]
such that the following diagrams commute for every $l_1, l_2, l_3 \in \Z_{\ge 1}$: 
\[
\xymatrix{
\wh{S}_{l_1}(\wh{S}_{l_2}(\wh{S}_{l_3}(X), \wh{\mca{U}^+_{l_3}})) \ar[r] \ar[d] & \wh{S}_{l_1+l_2}(\wh{S}_{l_3}(X), \wh{\mca{U}^+_{l_3}}) \ar[d] \\
\wh{S}_{l_1}(\wh{S}_{l_2+l_3}(X), \wh{\mca{U}^+_{l_2+l_3}})  \ar[r] & (\wh{S}_{l_1+l_2+l_3}(X), \wh{\mca{U}^+_{l_1+l_2+l_3}}), \\
}
\]
\[
\xymatrix{
\wh{S}_{l_1}(\wh{S}_{l_2}(X, \wh{\mca{U}}))  \ar[r] \ar[d] &  \wh{S}_{l_1}(\wh{S}_{l_2}(X), \wh{\mca{U}^+_{l_2}}) \ar[d] \\
\wh{S}_{l_1+l_2}(X, \wh{\mca{U}})  \ar[r] & (\wh{S}_{l_1+l_2}(X), \wh{\mca{U}^+_{l_1+l_2}}).  \\
}
\]
\item A $\tau$-collared CF-perturbation $\wh{\mf{S}^+_l}$ of $(\wh{S}_l(X), \wh{\mca{U}^+_l})$ for every $l \in \Z_{\ge 1}$,
such that the pullback of $\wh{\mf{S}^+_{l_1+l_2}}$ by 
$\wh{S}_{l_1}(\wh{S}_{l_2}(X), \wh{\mca{U}^+_{l_2}}) \to (\wh{S}_{l_1+l_2}(X), \wh{\mca{U}^+_{l_1+l_2}})$ 
coincides with the restriction of $\wh{\mf{S}^+_{l_2}}$
for every $l_1, l_2 \in \Z_{\ge 1}$. 
\item A $\tau$-collared admissible map $\wh{f^+_l}:  (\wh{S}_l(X) , \wh{\mca{U}^+_l}) \to \mca{L}_{k+1}$
for every $l \in \Z_{\ge 1}$ such that: 
\begin{itemize}
\item The pullback of $\wh{f^+_{l_1+l_2}}$ by $\wh{S}_{l_1}(\wh{S}_{l_2}(X), \wh{\mca{U}^+_{l_2}}) \to (\wh{S}_{l_1+l_2}(X), \wh{\mca{U}^+_{l_1+l_2}})$ 
coincides with the restriction of $\wh{f^+_{l_2}}$ 
for every $l_1, l_2 \in \Z_{\ge 1}$. 
\item For every $j \in \{0, \ldots, k\}$, 
$\ev^{\mca{L}}_j \circ \wh{f_l^+}: (\wh{S}_l(X), \wh{\mca{U}^+_l}) \to L$ coincides with 
$\ev^{\mca{L}}_j \circ \wh{f} |_{\wh{S}_l(X)}: (\wh{S}_l(X), \wh{\mca{U}}|_{\wh{S}_l(X)}) \to L$
via the KK-embedding $\wh{\mca{U}}|_{\wh{S}_l(X)} \to \wh{\mca{U}^+_l}$. 
\item $\ev^{\mca{L}}_0 \circ \wh{f_l^+}:   (\wh{S}_l(X), \wh{\mca{U}^+_l}) \to L$ 
is strata-wise strongly submersive with respect to $\wh{\mf{S}^+_l}$. 
\item $\wh{f^+_l}$ is $\ep$-close to $\wh{f}|_{\wh{S}_l(X)}$. 
\end{itemize} 
\end{itemize} 
Then, for any $\tau' \in (0, \tau)$, there exist the following data: 
\begin{itemize}
\item A $\tau'$-collared K-structure $\wh{\mca{U}^+}$ on $X$, which is a thickening of $\wh{\mca{U}}$. 
\item An isomorphism of $\tau'$-collared K-structures $\wh{\mca{U}^+}|_{\wh{S}_l(X)} \cong \wh{\mca{U}^+_l}$ for every $l \in \Z_{\ge 1}$. 
\item A $\tau'$-collared CF-perturbation $\wh{\mf{S}^+}$ of $(X, \wh{\mca{U}^+})$ such that 
$\wh{\mf{S}^+}|_{\wh{S}_l(X)}$ coincides with $\wh{\mf{S}^+_l}$
via the isomorphism of K-spaces $\wh{\mca{U}^+}|_{\wh{S}_l(X)} \cong \wh{\mca{U}^+_l}$. 
\item A $\tau'$-collared admissible map $\wh{f^+}:  (X, \wh{\mca{U}^+}) \to \mca{L}_{k+1}$ such that: 
\begin{itemize}
\item $\wh{f^+}$ is $\ep$-close to $\wh{f}$. 
\item For every $j \in \{0, \ldots, k\}$, $\ev^{\mca{L}}_j \circ \wh{f^+}$ coincides with $\ev^{\mca{L}}_j \circ \wh{f}$ with respect to the KK-embedding $\wh{\mca{U}} \to \wh{\mca{U}^+}$. 
\item $\ev^{\mca{L}}_0 \circ \wh{f^+} : (X, \wh{\mca{U}^+}) \to L$ is strata-wise strongly submersive with respect to $\wh{\mf{S}^+}$. 
\end{itemize} 
\end{itemize}
\end{thm}
\begin{proof}
The $K$-structure $\wh{\mca{U}^+}$ and the CF-perturbation $\wh{\mf{S}^+}$ are defined by 
Propositions 17.62 and 17.65 in \cite{FOOO_Kuranishi}, respectively; 
here we apply Proposition 17.65 (2) to $\ev^{\mca{L}}_0 \circ \wh{f}: (X, \wh{\mca{U}}) \to L$. 
Moreover $\wh{f^+_1}: \wh{S}_1(X) \to \mca{L}_{k+1}$
extends to the $\tau'$-neighborhood of $\partial X$ (denoted by $N(\tau')$). 
Then we can apply Theorem \ref{170430_2} to conclude the proof, 
taking $\mca{N}(\tau')$ as $Z$ and the extension of $\wh{f^+_1}$ as $\wh{g}$. 
\end{proof} 

\begin{rem}\label{180308_1} 
Nextly, we should state and prove a version of Theorem \ref{170430_1} 
such that $(X, \wh{\mca{U}})$ is a $\tau$-collared K-space, and 
\[ 
\wh{f}: (X, \wh{\mca{U}}) \to [a, b]^{\boxplus \tau} \times \mca{L}_{k+1}^\con 
\] 
is a $\tau$-collared strongly continuous map, 
where $a<b$ are real numbers and $[a,b]^{\boxplus \tau}: = [a-\tau, b+\tau]$. 
However, the statement is similar to Theorem \ref{170430_1} 
and involves further notations, 
thus here we choose not to write it down in detail. 
\end{rem} 

\subsection{Wrap-up of the proof}

Now we can complete the proof of Theorem \ref{161215_1} assuming results
presented in Sections 7.1 and 7.5. 

Let $X$ be one of moduli spaces considered in Theorem \ref{170611_5} (i). 
In Section 7.4, we defined a strongly continuous map from $X$ to the space of continuous loops (with marked points). 
This map naturally extends to a $1$-collared strongly continuous map from $X^{\boxplus 1}$, 
see Lemma-Definition 17.35 and Lemma 17.37 (3)  in \cite{FOOO_Kuranishi}. 
Taking $\tau(X) \in (1/2, 1)$ as in Remark \ref{171206_2}
and successively applying Theorem \ref{170430_1} and Remark \ref{180308_1}, 
$\bar{X}:= X^{\boxplus 1/2}$ is equipped with an admissible CF-perturbation and 
an admissible strongly smooth map to the space of smooth loops (with marked points). 
In conclusion, we obtain the following data
for every $k \in \Z_{\ge 0}$, $m \in \Z_{\ge 0}$ and $P \in \{ \{m\}, [m, m+1]\}$. 

\begin{enumerate}
\item[(i):] Compact admissible K-spaces 
\begin{align*} 
&\bar{\mca{M}}_{k+1}(\beta: P)  \qquad ( \omega_n(\beta) < \ep(m+1-k)), \\
&\bar{\mca{N}}^0_{k+1}(\beta: P) \qquad (\omega_n(\beta) <\ep(m-1-k)),  \\
&\bar{\mca{N}}^{\ge 0}_{k+1}(\beta: P) \qquad (\omega_n(\beta) < \ep(m-k-U)), 
\end{align*}
and admissible CF-perturbations on these K-spaces. 
\item[(ii):] Admissible maps 
\begin{align} 
\ev^{\bar{\mca{M}}_{k+1}}:  \bar{\mca{M}}_{k+1}(\beta: P) & \to  P  \times \mca{L}_{k+1}(\partial \beta)  \\
\ev^{\bar{\mca{N}}^0_{k+1}}: \bar{\mca{N}}^0_{k+1}(\beta: P) & \to  P \times \mca{L}_{k+1}(\partial \beta)  \\
\ev^{\bar{\mca{N}}^{\ge 0}_{k+1}}: \bar{\mca{N}}^{\ge 0}_{k+1}(\beta: P) &\to P \times \mca{L}_{k+1}(\partial \beta)  
\end{align} 
such that compositions with $\id_P \times \ev_0$
(which are admissible maps to $P \times L$) 
are strata-wise strongly submersive with respect to the CF-perturbations in (i). 
\item[(iii):] Isomorphisms of admissible K-spaces, 
which are obtained from isomorphisms (\ref{170903_1})--(\ref{170903_6}) 
by replacing each moduli space $X$ with $\bar{X}$. 
For example, we obtain
\begin{equation}\label{171019_1} 
\partial \bar{\mca{M}} _{k+1}(\beta: m) \cong \bigsqcup_{\substack{k_1+k_2=k+1 \\ 1 \le i \le k_1 \\ \beta_1 + \beta_2 = \beta}} (-1)^{\ep_0} 
\bar{\mca{M}}_{k_1+1}(\beta_1: m) \fbp{i} {0} \bar{\mca{M}}_{k_2+1}(\beta_2: m) 
\end{equation} 
from (\ref{170903_1}). 
We require that these isomorphisms are compatible with 
CF-perturbations in (i) and evaluation maps (ii). 
\end{enumerate} 
Applying results in Section 7.1, one can define 
\begin{align} 
x_m(a,k)&:= \sum_{\omega_n(\bar{a}) < \ep(m+1-k)} (-1)^{n+1}  \ev_*(\bar{\mca{M}}_{k+1}(\bar{a}, \{m\})) \\ 
\bar{x}_m(a,k)&:= \sum_{\omega_n(\bar{a}) < \ep(m+1-k)} (-1)^{n+1}  \ev_*(\bar{\mca{M}}_{k+1}(\bar{a}, [m, m+1])), \\ 
y_m(a,k)&:=  \sum_{\omega_n(\bar{a}) < \ep(m-U-k)}  (-1)^{n+k+1} \ev_*(\bar{\mca{N}}_{k+1}^{\ge 0}(\bar{a}, \{m\})),  \\ 
\bar{y}_m(a,k)&:= \sum_{\omega_n(\bar{a}) < \ep(m-U-k)} (-1)^{n+k+1}   \ev_*(\bar{\mca{N}}_{k+1}^{\ge 0} (\bar{a}, [m, m+1])),  \\
z_m(a,k)&:=  \sum_{\omega_n(\bar{a}) < \ep(m-1-k)} (-1)^{n+k+1} \ev_*(\bar{\mca{N}}_{k+1}^0 (\bar{a}, \{m\})), \\ 
\bar{z}_m(a,k)&:= \sum_{\omega_n(\bar{a}) < \ep(m-1-k)}  (-1)^{n+k+1}  \ev_*(\bar{\mca{N}}_{k+1}^0 (\bar{a}, [m, m+1])). 
\end{align} 
\begin{rem} 
In the above formulas we abbreviate differential form $1$ (see Remark \ref{170918_3}) and 
CF-perturbations. We also abbreviate superscripts of $\ev$ for simplicity. 
\end{rem} 

\begin{rem} 
Here we explain one issue which can be overlooked by abbreviating CF-perturbations in the above formulas. 
As is clear from Section 7.1, to define a de Rham chain 
one has to fix a parameter of CF-perturbation $e \in (0, 1]$ 
(here we use a letter $e$, not to be confused with $\ep$ used in the above formulas). 
Hence we take a strictly decreasing sequence of positive real numbers $(e_m)_{m \ge 1}$ satisfying 
$\lim_{m \to \infty} e_m = 0$, 
and define $x_m(a,k)$ by 
\[ 
x_m(a,k) := \sum_{\omega_n(\bar{a}) < \ep(m+1-k)} (-1)^{n+1}  \ev_*(\bar{\mca{M}}_{k+1}(\bar{a}, \{m\}), \wh{\mf{S}}^{e_m}). 
\] 
$y_m(a,k)$ and $z_m(a,k)$ are defined by similar formulas, using $e_m$. 
Then 
$\bar{x}_m(a,k)$ is defined by interpolating $e_m$ and $e_{m+1}$ on the moduli space $\bar{\mca{M}}_{k+1}(\bar{a}, [m, m+1])$. 
$\bar{y}_m(a,k)$ and $\bar{z}_m(a,k)$ are defined in a similar way. 
\end{rem}

Let us check that the requirements in Theorem \ref{161215_1} are satisfied. 
$x_m= e_-(\bar{x}_m)$, $y_m = e_-(\bar{y}_m)$, $z_m = e_-(\bar{z}_m)$ follow from the above definition
and (\ref{171206_1}). 
Moreover 
\[ 
(x_{m+1} - e_+(\bar{x}_m)) (a,k) \ne 0 \implies  \omega_n(\bar{a}) \ge \ep(m+1-k)
\] 
show $x_{m+1} - e_+(\bar{x}_m) \in F^m$. 
By similar arguments one can show 
$y_{m+1} - e_+(\bar{y}_m) \in F^{m-U-1}$ and 
$z_{m+1} - e_+(\bar{z}_m) \in F^{m-2}$. 

The isomorphism (\ref{171019_1}) shows
\[ 
\partial \bar{x}_m(a,k) = \sum_{\substack{k_1+k_2=k+1 \\ a_1+a_2 = a \\ 1 \le i \le k_1}} (-1)^{(k_1-m)(k_2-1) + (k_1-1)} \bar{x}_m (a_1, k_1) \circ_i \bar{x}_m (a_2, k_2)
\] 
for every $(a,k)$ such that $\omega_n(\bar{a}) < \ep(m+1-k)$, 
thus $\part \bar{x}_m - \frac{1}{2} [ \bar{x}_m,  \bar{x}_m] \in F^m$. 
By similar arguments one can show 
$\part \bar{y}_m - [\bar{x}_m, \bar{y}_m] - \bar{z}_m \in F^{m-U-1}$ and 
$\part \bar{z}_m - [\bar{x}_m, \bar{z}_m] \in F^{m-2}$. 

Finally $x_m(\bar{a}, k) \ne 0 \implies \mca{M}_{k+1}(\bar{a}) \ne \emptyset$, 
thus $\omega_n(\bar{a}) \ge 2\ep$ or $a=0$, $k \ge 2$. 
Moreover 
$[x_m(0,2)] = (-1)^{n+1} [\mca{M}_3(0)] = (-1)^{n+1}[L]$. 
Similarly, 
$z_m(\bar{a}, k) \ne 0 \implies \mca{N}^0_{k+1}(\bar{a}) \ne \emptyset$, 
thus $\omega_n(\bar{a}) \ge 2\ep$ or $a=0$. 
Moreover
$[z_m(0,0)] = (-1)^{n+1}[\mca{N}^0_1(0)] = (-1)^{n+1}[L]$. 
\qed

\section{Strongly smooth map from a K-space with a CF-perturbation gives a de Rham chain} 

The goal of this section is to explain proofs of results presented in Section 7.1. 
Namely, given a strongly smooth map from a K-space (equipped with a differential form and a CF-perturbation) to $\mca{L}_{k+1}$ (the space of loops with $k+1$ marked points), 
we define a de Rham chain on $\mca{L}_{k+1}$ and prove Stokes' formula and fiber product formula. 
Here we imitate arguments in \cite{FOOO_Kuranishi} Sections 7 and 9 in our setting. 
In Section 8.1 we consider smooth maps from single K-charts.
In Section 8.2 we consider strongly smooth maps from spaces with good coordinate system (GCS), and in Section 8.3 strongly smooth maps from K-spaces. 
In Sections 8.2 and 8.3, we only consider spaces without boundaries (and corners), 
since generalizations to spaces with boundaries are straightforward. 

Throughout this section, $X$ denotes a separable, metrizable topolgoical space.

\subsection{Single K-chart} 

In this subsection, 
given a smooth map from a K-chart (equipped with a CF-perturbation and a differential form) 
to $\mca{L}_{k+1}$, 
we define a de Rham chain on $\mca{L}_{k+1}$. 
We also prove Stokes's formula and fiber product formula. 
We first consider K-charts without boundary, and then proceed to the case of K-charts with boundaries. 

\subsubsection{K-chart without boundary} 
Suppose we are given the following data: 
\begin{itemize}
\item $\mca{U}=(U, \mca{E}, s, \psi)$ is a K-chart of $X$. 
\item $f: U \to \mca{L}_{k+1}$ is a smooth map in the sense of Definition \ref{171205_1}. 
\item $\omega \in \mca{A}^*_c(U)$. 
\item $\mf{S} = (\mf{S}^\ep)_{0 < \ep \le 1}$ is a CF-perturbation of $\supp \omega$ such that $\ev_0 \circ f: U \to L$ is strongly submersive with respect to $\mf{S}$
(see Definition 7.24 and Definition-Lemma 7.25 in \cite{FOOO_Kuranishi}). 
\end{itemize}

We are going to define $f_*(\mca{U}, \omega, \mf{S}^\ep) \in C^\dR_*(\mca{L}_{k+1})$ for every $\ep \in (0,1]$. 
Let $( \mf{V}_{\mf{r}}, \mca{S}_{\mf{r}})_{\mf{r} \in \mf{R}}$ be a representative of $\mf{S}$
(see Definitions 7.15, 7.16 and 7.19 in \cite{FOOO_Kuranishi}), such that 
\begin{itemize} 
\item $\mf{V}_{\mf{r}} = (V_{\mf{r}}, E_{\mf{r}}, \phi_{\mf{r}}, \wh{\phi}_{\mf{r}})$ is a manifold chart 
(we do not consider group actions: see Remark \ref{170912_1}) of $(U, \mca{E})$ such that
$(\phi_{\mf{r}}(V_{\mf{r}}))_{\mf{r} \in \mf{R}}$ covers $U$. 
Let $s_{\mf{r}}: V_{\mf{r}} \to E_{\mf{r}}$ denote the pull back of $s$ by $\phi_{\mf{r}}$. 
\item $\mca{S}_{\mf{r}} = (W_{\mf{r}}, \eta_{\mf{r}}, \{ \mf{s}^\ep_{\mf{r}}\}_{\ep})$ is a CF-perturbation of 
$\mca{U}$ on $\mf{V}_{\mf{r}}$. Namely 
\begin{itemize}
\item $W_{\mf{r}}$ is an open neighborhood of $0$ of a finite-dimensional oriented vector space $\wh{W_{\mf{r}}}$.  
\item $\mf{s}^\ep_{\mf{r}}: V_{\mf{r}} \times W_{\mf{r}} \to E_{\mf{r}}$ is a $C^\infty$ map transversal to $0$ for every $\ep \in (0,1]$. 
\item $\lim_{\ep \to 0} \mf{s}^\ep_{\mf{r}} (y, \xi) = s_{\mf{r}}(y)$ in compact $C^1$ topology on $V_{\mf{r}} \times W_{\mf{r}}$.
\item $\eta_{\mf{r}} \in \mca{A}^{\dim W_{\mf{r}}} _c(W_{\mf{r}})$ such that $\int_{W_{\mf{r}}} \eta_{\mf{r}} = 1$. 
\end{itemize} 
\item $\ev_0 \circ f \circ \phi_{\mf{r}} \circ \pr_{V_\mf{r}}: (\mf{s}^\ep_{\mf{r}})^{-1}(0) \to L $ is a submersion for every $\mf{r} \in \mf{R}$ and $\ep \in (0, 1]$. 
\end{itemize} 

$V_{\mf{r}}$ and $E_{\mf{r}}$ are oriented so that $\phi_{\mf{r}}$ and $\wh{\phi}_{\mf{r}}$ preserve orientations. 
$(\mf{s}^\ep_{\mf{r}})^{-1}(0)$ is oriented so that the isomorphism 
\[ 
E_{\mf{r}} \oplus T  (\mf{s}^\ep_{\mf{r}})^{-1}(0) \cong T V_{\mf{r}} \oplus TW_{\mf{r}}
\] 
preserves orientations, following Convention 8.2.1 in \cite{FOOO_09}. 

We take a partition of unity $\{ \chi_{\mf{r}}\}_{\mf{r} \in \mf{R}}$ subordinate to $(\phi_{\mf{r}}(V_{\mf{r}}))_{\mf{r} \in \mf{R}}$, 
i.e., 
$\chi_\mf{r} \in C^\infty_c(U, [0,1])$ and $\supp \chi_\mf{r}  \subset \phi_\tau(V_\mf{r})$
for every $\mf{r} \in \mf{R}$, and 
$\sum_{\mf{r} \in \mf{R}} \chi_{\mf{r}} \equiv 1$ on $\supp \omega$.
Then, for each $\ep \in (0, 1]$, we define 
\begin{equation}
f_*(\mca{U}, \omega, \mf{S}^\ep):= \sum_{\mf{r} \in \mf{R}} f_*(\mca{U}, \chi_{\mf{r}} \omega, \mf{V}_{\mf{r}}, \mca{S}^\ep_{\mf{r}}), 
\end{equation}
where the RHS is defined as 
\begin{equation}\label{180112_1} 
f_*(\mca{U}, \chi_{\mf{r}} \omega, \mf{V}_{\mf{r}}, \mca{S}^\ep_{\mf{r}}):= (-1)^\dagger  ( (\mf{s}^\ep_{\mf{r}})^{-1}(0), f  \circ \phi_{\mf{r}} \circ \pr_{V_{\mf{r}}} , (\phi_{\mf{r}} \circ \pr_{V_{\mf{r}}})^* (\chi_{\mf{r}} \omega) \wedge (\pr_{W_{\mf{r}}})^*(\eta_{\mf{r}})), 
\end{equation}
\begin{equation}\label{171105_2} 
\dagger:= \dim W_{\mf{r}} \cdot ( \rk \mca{E} + |\omega|). 
\end{equation} 
$(\mf{s}^\ep_{\mf{r}})^{-1}(0)$ is an oriented $C^\infty$-manifold, 
and $\ev_0 \circ  f  \circ \phi_{\mf{r}} \circ \pr_{V_{\mf{r}}}$
is a submersion, thus
the RHS in (\ref{180112_1}) makes sense. 

\begin{rem} 
Strictly speaking, one has to fix an embedding of $(\mf{s}^\ep_{\mf{r}})^{-1}(0)$ to Euclidean space to define a de Rham chain. 
However it is easy to check that the de Rham chain does not depend on choice of embedding. 
\end{rem} 

A priori, the above definition may depend on choices of representative and partition of unity. 
In Lemma \ref{170323_1} below we prove well-definedness and Stokes' formula.

\begin{lem}\label{170323_1} 
\begin{enumerate} 
\item[(i):] For any $\omega_1, \omega_2 \in \mca{A}^*_c(U)$, $a_1, a_2 \in \R$ and $\ep \in (0, 1]$, 
\[
f_* (\mca{U}, a_1 \omega_1 + a_2 \omega_2 , \mf{S}^\ep) = a_1 f_*(\mca{U}, \omega_1, \mf{S}^\ep) + a_2 f_*(\mca{U}, \omega_2, \mf{S}^\ep).
\]
\item[(ii):] The above definition does not depend on choices of representative of the CF-perturbation $\mf{S}$ and partition of unity. 
\item[(iii):] $\partial f_*(\mca{U}, \omega, \mf{S}^\ep) = (-1)^{|\omega|+1} f_*(\mca{U}, d\omega, \mf{S}^\ep)$ for every $\ep \in (0, 1]$. 
\end{enumerate} 
\end{lem}
\begin{proof}
(i) is obvious as far as we fix a representative of $\mf{S}$ and a partition of unity.  
Thus to prove (ii), it is sufficient to prove the following claim: 
\begin{quote} 
Let $\mca{S}_i = (W_i, \eta_i, \{\mf{s}^\ep_i\}_\ep) \, (i=1,2)$ be two equivalent CF-perturbations 
of $\mca{U}$ on a manifold chart $(V, E, \phi, \wh{\phi})$ and 
$\omega \in \mca{A}^*_c(U)$ such that $\supp \omega \subset \phi(V)$. 
Then
\begin{align*} 
&( (\mf{s}^\ep_1)^{-1}(0), f \circ \phi \circ \pr_V,  (\phi \circ \pr_V)^*\omega \wedge (\pr_{W_1})^* \eta_1)  \\
&=(-1)^\dagger ( (\mf{s}^\ep_2)^{-1}(0), f \circ \phi \circ \pr_V,  (\phi \circ \pr_V)^*\omega \wedge (\pr_{W_2})^* \eta_2), \\
&\dagger:=  (\dim W_1 - \dim W_2)(\rk \mca{E} + |\omega|) 
\end{align*} 
\end{quote} 
By definition of equivalence (see Definition 7.5 in \cite{FOOO_Kuranishi}), 
we may assume that there exists a linear projection 
$\Pi: \wh{W}_2 \to \wh{W}_1$ such that $(\Pi)_! (\eta_2) = \eta_1$ and 
$\mf{s}^\ep_2 = (\id_V \times \Pi)^*(\mf{s}^\ep_1)$. 
Then, pushout of 
$\id_V \times \Pi: (\mf{s}^\ep_2)^{-1}(0) \to (\mf{s}^\ep_1)^{-1}(0)$ 
sends 
$(\phi \circ \pr_V)^*\omega \wedge (\pr_{W_2})^* \eta_2$
to  $(-1)^\dagger (\phi \circ \pr_V)^*\omega \wedge (\pr_{W_1})^* \eta_1$, 
which completes the proof. 

To prove (iii), by (i) and (ii) we may assume that $\supp \omega$ is sufficiently small and $\chi_{\mf{r}} \equiv 1$ on $\supp \omega$ for some $\mf{r} \in \mf{R}$. 
In this case (iii) is obvious except for signs, which can be checked by simple computations. 
\end{proof}

Now we can state the fiber product formula. 
Suppose, for each $i \in \{1, 2\}$, we have 
$X_i$, $\mca{U}_i = (U_i, \mca{E}_i, s_i, \psi_i)$, $f_i: U_i \to \mca{L}_{k_i+1}$, 
$\omega_i$, $\mf{S}_i$ as before. 
Then, for every $\ep \in (0, 1]$ one can define 
\[
(f_i)_*(\mca{U}_i, \omega_i, \mf{S}_i^\ep) \in C^\dR_* (\mca{L}_{k_i+1})
\]
for each $i \in \{1, 2\}$. 
On the other hand, for each $j \in \{1, \ldots, k_1\}$, 
one can take a fiber product of K-charts 
\[ 
\mca{U}_{12} = \mca{U}_1 \fbp{\ev_j \circ f_1}{\ev_0 \circ f_2} \mca{U}_2. 
\] 
One can also define a fiber product of CF-perturbations  $\mf{S}_1 \times \mf{S}_2$ on $\mca{U}_{12}$. 
Finally we define a differential form $\omega_{12}$ on $\mca{U}_{12}$ by 
\[ 
\omega_{12} := (-1)^{(\dim U_1 - \rk \mca{E}_1- |\omega_1| - n)|\omega_2|}  \cdot  \omega_1 \times \omega_2, 
\] 
and a smooth map 
\[ 
f_{12}: U_1 \fbp{\ev_j \circ f_1}{\ev_0 \circ f_2} U_2 \to \mca{L}_{k_1+k_2}; \quad 
(x_1, x_2) \mapsto \con_j (f_1(x_1), f_2(x_2)). 
\] 
Then one can state the fiber product formula as follows. 
The proof is obvious except for signs, 
which can be checked by direct computations.

\begin{lem}\label{171110_3} 
In the situation described above, there holds 
\[ 
(f_{12})_* (\mca{U}_{12}, \omega_{12}, \mf{S}^\ep_{12}) = 
(f_1)_* (\mca{U}_1, \omega_1, \mf{S}^\ep_1) \circ_j 
(f_2)_* (\mca{U}_2, \omega_2, \mf{S}^\ep_2)
\]
for every $\ep \in (0, 1]$. 
\end{lem} 

\subsubsection{K-chart with boundary} 

Suppose we are given the following data: 
\begin{itemize}
\item $\mca{U} = (U, \mca{E}, s, \psi)$ is an admissible K-chart on $X$, 
i.e. $U$ is an admissible manifold with boundaries (and corners), 
the vector bundle $\mca{E}$ and the section $s$ are also admissible. 
\item $f: U \to \mca{L}_{k+1}$ is an admissible map in the sense of Definition \ref{170701_1} (i). 
\item $\omega \in \mca{A}^*_c(U)$ is admissible. 
\item $\mf{S} = ( \mf{S}^\ep)_{0<\ep \le 1}$ is an admissible CF-perturbation of $\supp \omega$
such that $\ev_0 \circ f: U \to L$ is strata-wise strongly submersive with respect to $\mf{S}$. 
\end{itemize} 
Under these assumptions, 
our goal is to define $f_*(\mca{U}, \omega, \mf{S}^\ep) \in C^\dR_*(\mca{L}_{k+1})$ 
for every $\ep \in (0, 1]$.

Let $(\mf{V}_{\mf{r}}, \mca{S}_{\mf{r}})_{\mf{r} \in \mf{R}}$ be a representative of $\mf{S}$, 
and $(\chi_{\mf{r}})_{\mf{r} \in \mf{R}}$ be a partition of unity subordinate to 
$(\phi_{\mf{r}}(V_{\mf{r}}))_{\mf{r} \in \mf{R}}$ such that 
$\sum_{\mf{r}} \chi_{\mf{r}} \equiv 1$
on $\supp \omega$. 
Then we define 
\begin{equation}\label{171105_1} 
f_*( \mca{U}, \omega, \mf{S}^\ep) := \sum_{\mf{r} \in \mf{R}} f_*(\mca{U}, \chi_{\mf{r}} \omega, \mf{V}_{\mf{r}}, \mca{S}^\ep_{\mf{r}})
\end{equation} 
where each term in the RHS is defined below. 
The proof of well-definedness (the RHS does not depend on choices of representatives of $\mf{S}$ and partition of unity) 
is straightforward and omitted. 

Let $D:= \dim U$. 
We may assume that $V_{\mf{r}}$ is an open neighborhood of $(t_1,\ldots, t_D) \in  (\R_{\ge 0})^D$. 
We define 
\[
\mca{R}: \R^D \to  (\R_{\ge 0})^D ; \,(t_1,\ldots, t_D) \mapsto (t'_1, \ldots, t'_D)
\]
by $t'_i := \begin{cases} t_i &(t_i \ge 0) \\  0 &(t_i < 0)  \\ \end{cases}$ for every $1 \le i \le D$. 
We take $\kappa \in C^\infty(\R, [0,1])$ 
such that 
$\kappa \equiv 1$ on a neighborhood of $\R_{\ge 0}$, 
and $\kappa \equiv 0$ on a neighborhood of $\R_{\le -1}$. 
Now let us define the following data: 
\begin{itemize} 
\item $\bar{V}_{\mf{r}}:= \mca{R}^{-1}(V_{\mf{r}}) $, $\bar{E}_{\mf{r}} := \mca{R}^*E_{\mf{r}}$.
\item $\bar{\mf{s}}^\ep_{\mf{r}}: = ( \mca{R}|_{\bar{V}_{\mf{r}}} \times \id_{W_{\mf{r}}} )^*(\mf{s}^\ep_{\mf{r}})$. 
\item $\bar{f}_{\mf{r}} := f \circ \phi_{\mf{r}} \circ \mca{R}|_{\bar{V}_{\mf{r}}}$. 
\item $\overline{\chi_{\mf{r}} \omega}(t_1, \ldots, t_D):= \kappa(t_1) \cdots \kappa(t_D) \cdot (\phi_{\mf{r}}  \circ \mca{R}|_{\bar{V}_{\mf{r}}} )^*(\chi_{\mf{r}} \omega)$. 
\end{itemize} 
Finally, we define 
\begin{equation}\label{170914_1} 
 f_*(\mca{U}, \chi_{\mf{r}} \omega, \mf{V}_{\mf{r}}, \mca{S}^\ep_{\mf{r}})
:= (-1)^\dagger ( (\bar{\mf{s}}^\ep_{\mf{r}})^{-1}(0),  \bar{f}_{\mf{r}} \circ \pr_{\bar{V}_{\mf{r}}}, 
\pr^*_{\bar{V}_{\mf{r}}} ( \overline{\chi_{\mf{r}} \omega}) \wedge \pr^*_{W_{\mf{r}}} (\eta_{\mf{r}})), 
\end{equation} 
where the sign $\dagger$ is defined by (\ref{171105_2}). 
Note that $\ev_0 \circ \bar{f}_{\mf{r}} \circ \pr_{\bar{V}_{\mf{r}}}:  (\bar{\mf{s}}^\ep_{\mf{r}})^{-1}(0) \to L$ 
is a submersion, thus the map 
$\bar{f}_{\mf{r}} \circ \pr_{\bar{V}_{\mf{r}}}:  (\bar{\mf{s}}^\ep_{\mf{r}})^{-1}(0) \to L$ 
is smooth. 

\begin{rem} 
The RHS of (\ref{170914_1}) may depend on the choice of the cutoff function $\kappa$. 
Here we fix such $\kappa$ and drop it from our notation in the following arguments. 
\end{rem} 

The fiber product formula holds in the obvious manner and omitted. 
Stokes' formula is formulated as follows. 

\begin{prop}
\[ 
\partial (f_*(\mca{U}, \omega, \mf{S}^\ep)) =  
(-1)^{|\omega|} (f|_{\partial \mca{U}})_* (\partial \mca{U}, \omega|_{\partial \mca{U}}, \mf{S}^\ep|_{\partial \mca{U}}) +  
(-1)^{|\omega|+1} f_*(\mca{U}, d \omega, \mf{S}^\ep). 
\]
\end{prop}
\begin{proof}
This follows from 
\begin{align*} 
d(\overline{\chi_{\mf{r}}\omega}) (t_1,\ldots, t_D)  &=  \kappa(t_1) \cdots \kappa(t_D) (\phi_{\mf{r}} \circ \mca{R}|_{\bar{V}_{\mf{r}}})^* (d(\chi_{\mf{r}}\omega)) \\ 
&+ \sum_{i=1}^D \kappa(t_1) \cdots  d\kappa(t_i) \cdots \kappa(t_D) ( \phi_{\mf{r}} \circ \mca{R}|_{\bar{V}_{\mf{r}}})^*(\chi_{\mf{r}} \omega) 
\end{align*} 
and $\partial U$ is oriented so that $TU \cong \R_{\text{out}} \oplus T(\partial U)$ preserves orientations.  
\end{proof}

\subsubsection{K-chart with boundary over an interval} 

Suppose we are given the following data: 
\begin{itemize}
\item $\mca{U} = (U, \mca{E}, s, \psi)$ is an admissible K-chart on $X$. 
\item $f: U \to [a, b] \times \mca{L}_{k+1}$ is an admissible map in the sense of Definition \ref{170701_2} (i). 
\item $\omega \in \mca{A}^*_c(U)$ is admissible. 
\item $\mf{S} = ( \mf{S}^\ep)_{0<\ep \le 1}$ is an admissible CF-perturbation of $\supp \omega$
such that $\ev_0 \circ f: U \to [a,b] \times L$ is strata-wise strongly submersive with respect to $\mf{S}$. 
\end{itemize} 
Under these assumptions, 
our goal is to define 
$f_*(\mca{U}, \omega, \mf{S}^\ep) \in \bar{C}^\dR_*(\mca{L}_{k+1})$ for every $\ep \in  (0,1]$. 
For simplicity of notations, in the following we assume that $a=-1$, $b=1$. 

Let $(\mf{V}_{\mf{r}}, \mca{S}_{\mf{r}})_{\mf{r} \in \mf{R}}$ be a representative of $\mf{S}$, 
and $(\chi_{\mf{r}})_{\mf{r} \in \mf{R}}$ be a partition of unity subordinate to 
$(\phi_{\mf{r}}(V_{\mf{r}}))_{\mf{r} \in \mf{R}}$ such that 
$\sum_{\mf{r}} \chi_{\mf{r}} \equiv 1$
on $\supp \omega$. 
Then we define 
\begin{equation}\label{170918_1} 
f_*( \mca{U}, \omega, \mf{S}^\ep) := \sum_{\mf{r} \in \mf{R}} f_*(\mca{U}, \chi_{\mf{r}} \omega, \mf{V}_{\mf{r}}, \mca{S}^\ep_{\mf{r}})
\end{equation} 
where each term in the RHS is defined below. 
The proof of well-definedness is omitted. 

To define  
$f_*(\mca{U}, \chi_{\mf{r}} \omega, \mf{V}_{\mf{r}}, \mca{S}^\ep_{\mf{r}})$, 
it is sufficient to consider the following three cases: 
\begin{enumerate}
\item[(i):]  $f \circ \phi_{\mf{r}} (V_{\mf{r}})$ is contained in $(-1, 1) \times \mca{L}_{k+1}$. 
\item[(ii):] $f \circ \phi_{\mf{r}} (V_{\mf{r}})$ is contained in $[-1, 1) \times \mca{L}_{k+1}$ and intersects $\{-1\}  \times \mca{L}_{k+1}$. 
\item[(iii):] $f \circ \phi_{\mf{r}} (V_{\mf{r}})$ is contained in $(-1, 1] \times \mca{L}_{k+1}$ and intersects $\{1\} \times \mca{L}_{k+1}$. 
\end{enumerate} 
The case (i) is similar to the case in the previous subsubsection and omitted. 
In the following we only consider the case (ii), since the case (iii) is completely parallel. 

Let $D:= \dim U$. 
We may assume that $V_{\mf{r}}$ is an open neighborhood of 
$(0, t_2, \ldots, t_D)$ in $(\R_{\ge 0})^D$ and 
$f_\R \circ \ph_{\mf{r}}(t_1, \ldots, t_D) = t_1-1$ .
Here $f_\R$ denotes $\pr_{\R} \circ f$. Similarly, we set $f_{\mca{L}} := \pr_{\mca{L}_{k+1}} \circ f$. 

We define $\bar{V}_{\mf{r}}$, $\bar{E}_{\mf{r}}$ and $\bar{\mf{s}}^\ep_{\mf{r}}$ in the same way as in the previous subsubsection. 
We also define
$\bar{f}_{\mf{r}}$ and $\overline{\chi_{\mf{r}}\omega}$ as follows: 
\begin{itemize} 
\item $\bar{f}_{\mf{r}}: \bar{V}_{\mf{r}} \to \R \times \mca{L}_{k+1}$ is defined by 
\[
\pr_{\mca{L}_{k+1}} \circ \bar{f}_{\mf{r}} := f_{\mca{L}} \circ \phi_{\mf{r}} \circ \mca{R}|_{\bar{V}_{\mf{r}}}, \qquad 
\pr_{\R} \circ \bar{f}_{\mf{r}} (t_1, \ldots, t_D): = t_1 - 1. 
\] 
\item $\overline{\chi_{\mf{r}} \omega}(t_1, \ldots, t_D):= \kappa(t_2) \cdots \kappa(t_D) \cdot (\phi_{\mf{r}} \circ \mca{R}|_{\bar{V}_{\mf{r}}})^*(\chi_{\mf{r}} \omega)$. 
\end{itemize} 
Here $\kappa \in C^\infty(\R, [0,1])$ is taken in the previous subsubsection. 
Then we define 
\[
f_*(\mca{U}, \chi_{\mf{r}} \omega, \mf{V}_{\mf{r}}, \mca{S}^\ep_{\mf{r}}):= 
(-1)^\dagger 
((\bar{\mf{s}}^\ep_{\mf{r}})^{-1}(0), \bar{f}_{\mf{r}} \circ \pr_{\bar{V}_{\mf{r}}} , \tau_+, \tau_-,  
\pr^*_{\bar{V}_{\mf{r}}} ( \overline{\chi_{\mf{r}} \omega}) \wedge \pr^*_{W_{\mf{r}}} (\eta_{\mf{r}})), 
\] 
where the sign $\dagger$ is defined as before and 
$\tau_-$, $\tau_+$ are defined as follows. 
$\tau_+$ is defined in the obvious way, since $f \circ \phi_{\mf{r}}(V_{\mf{r}})$ does not intersect $\{1\} \times \mca{L}_{k+1}$.
On the other hand, $\tau_-$ is defined as a restriction of 
\[ 
(\bar{V}_{\mf{r}} \cap \{t_1 \le 0\}) \times W_{\mf{r}} \to \R_{\le -1} \times ( (V_{\mf{r}}  \cap \{t_1=0\}) \times W_{\mf{r}} ); 
\,
(t_1,\ldots, t_D, w) \mapsto (t_1-1, t_2, \ldots, t_D, w). 
\] 
This completes the definition of 
$f_*(\mca{U}, \chi_{\mf{r}} \omega, \mf{V}_{\mf{r}}, \mca{S}^\ep_{\mf{r}})$, 
thus the definition of 
$f_*( \mca{U}, \omega, \mf{S}^\ep)$. 

The fiber product formula is stated and proved in the obvious way. 
Stokes' formula 
\[ 
\partial f_*(\mca{U}, \omega, \mf{S}^\ep) = (-1)^{|\omega|} (f|_{\partial_h \mca{U}})*(\partial_h \mca{U}, \omega, \mf{S}^\ep) + (-1)^{|\omega|+1} f_*(\mca{U}, d \omega, \mf{S}^\ep), 
\] 
where $\partial_h \mca{U}$ is the restriction of $\mca{U}$ to $\partial_h U$, 
and 
\begin{align*} 
&e_+(f_*(\mca{U}, \omega, \mf{S}^\ep)) = (f_{\mca{L}}|_{U_1})_*(\mca{U}|_{U_1}, \omega|_{U_1}, \mf{S}^\ep|_{U_1}), \\ 
&e_-(f_*(\mca{U}, \omega, \mf{S}^\ep)) = (f_{\mca{L}}|_{U_{-1}})_*(\mca{U}|_{U_{-1}}, \omega|_{U_{-1}}, \mf{S}^\ep|_{U_{-1}}) 
\end{align*}
can be checked directly and proofs are omitted. 

\subsection{Space with GCS}

In this subsection, we assume that $X$ is compact. 
Let $\wt{\mca{U}} = ( \{ \mca{U}_{\mf{p}} \}_{\mf{p} \in \mf{P}} , \{ \Phi_{\mf{p}\mf{q}}\}_{\mf{p} \ge \mf{q}})$
be a GCS (good coordinate system) without boundary on $X$ (see Section 10). 
We start from the following definition.

\begin{defn}\label{171108_1} 
A strongly smooth map $\wt{f}$ from $(X, \wt{\mca{U}})$ to $\mca{L}_{k+1}$ is a family 
$(f_{\mf{p}})_{\mf{p} \in \mf{P}}$
which satisfies the following conditions: 
\begin{itemize}
\item For every $\mf{p} \in \mf{P}$, $f_{\mf{p}}$ is a smooth map  from $U_{\mf{p}}$ to $\mca{L}_{k+1}$. 
\item For every $\mf{q} \le \mf{p}$, there holds $f_\mf{p} \circ \ph_{\mf{p}\mf{q}} = f_{\mf{q}}|_{U_{\mf{p}\mf{q}}}$. 
\end{itemize} 
\end{defn} 

Let $\mca{K}$ be a support system of $\wt{\mca{U}}$
and $\wt{\mf{S}}$ be a CF-perturbation of $(\wt{\mca{U}}, \mca{K})$
(see Definitions 5.6 and 7.47 in \cite{FOOO_Kuranishi}). 
Here we recall the definition of support system
(in the case $Z=X$): 

\begin{defn}
A support system of $\wt{\mca{U}}$ is $\mca{K} = (\mca{K}_{\mf{p}}) _{\mf{p} \in \mf{P}}$ 
where $\mca{K}_{\mf{p}}$ is a compact set of $U_{\mf{p}}$ for each $\mf{p} \in \mf{P}$
which is a closure of an open subset $\mtrg{\mca{K}}_{\mf{p}}$, 
and $\bigcup_{\mf{p} \in \mf{P}} \psi_{\mf{p}}(\mtrg{\mca{K}}_{\mf{p}} \cap s^{-1}_{\mf{p}}(0)) =X$. 
\end{defn}

We assume that $\wt{\mf{S}}$ is transversal to $0$, 
and $\ev_0 \circ \wt{f}: (X, \wt{\mca{U}})  \to L$ is strongly submersive with respect to $\wt{\mf{S}}$. 
Also, let $\wt{\omega} =  (\omega_{\mf{p}})_{\mf{p} \in \mf{P}}$ be a differential form on $(X, \wt{\mca{U}})$. 
Given these data, we are going to define
\begin{equation}\label{171108_2} 
\wt{f} _*(X, \wt{\mca{U}}, \wt{\omega}, \wt{\mf{S}}^\ep) \in C^\dR_*(\mca{L}_{k+1})
\end{equation} 
for sufficiently small $\ep>0$. 
Note that the support system $\mca{K}$ is a part of the data to define (\ref{171108_2}), 
though it is implicit in the above formula. 
\begin{rem} 
In contrast to the case of single K-charts, where de Rham chain is defined for all $\ep \in (0, 1]$, 
the de Rham chain (\ref{171108_2}) is defined only when $\ep>0$ is sufficiently small. 
\end{rem} 

In Section 8.2.1 we state and prove some technical lemmas. 
In Section 8.2.2 we define (\ref{171108_2}) and check its well-definedness, invariance under GG-embedding, and Stokes' formula. 
We only consider spaces with GCS \textit{without} boundaries, since generalization to spaces with boundaries (and corners) is straightforward. 
The arguments of this subsection partly follow arguments in Sections 7.5--7.7 in \cite{FOOO_Kuranishi} with some modifications. 

\subsubsection{Technical lemmas} 

Recall that $|\mca{K}| := \bigg( \bigsqcup_{\mf{p} \in \mf{P}} \mca{K}_{\mf{p}} \bigg) / \sim$ equipped with the 
topology from $|\wt{\mca{U}}|$ is metrizable; see Definition 5.6 (3) in \cite{FOOO_Kuranishi}. 
We fix a metric $d$ on $|\mca{K}|$ which is compatible with this topology. 

For any $S \subset |\mca{K}|$ and $\delta>0$ we set 
\[
B_\delta(S) := \{ x \in |\mca{K}| \mid d(S, x) < \delta\}, \qquad
\bar{B}_\delta(S):= \{ x \in |\mca{K}| \mid d(S, x) \le \delta \}.
\]

Recall the notion of ``support set'' of CF-perturbations from Definition 7.72 in \cite{FOOO_Kuranishi}: 
\begin{quote} 
For each $\mf{p} \in \mf{P}$, let $\{ (\mf{V}_{\mf{r}}, \mca{S}_{\mf{r}}) \mid \mf{r} \in \mf{R}\}$ be a representative of the CF-perturbation $\mf{S}_{\mf{p}}$ of $\mca{U}_{\mf{p}}$. 
Then for each $\ep \in (0,1]$, we define $\Pi ( (\mf{S}^\ep_{\mf{p}})^{-1}(0)))$ to be the set consisting of all $x \in U_{\mf{p}}$ such that 
there exists $\mf{r} \in \mf{R}$, $y \in V_{\mf{r}}$, $\xi \in W_{\mf{r}}$ such that 
\[ 
\phi_{\mf{r}} (y) = x, \qquad s^\ep_{\mf{r}}(y, \xi)=0, \qquad \xi \in \supp \eta_{\mf{r}}. 
\]
This definition is independent of the choice of representative. 
Then we define 
\[ 
\Pi ((\wt{\mf{S}}^\ep)^{-1}(0)):= \bigcup_{\mf{p} \in \mf{P}} (\mca{K}_{\mf{p}} \cap \Pi ( (\mf{S}^\ep_{\mf{p}})^{-1}(0))) \subset |\mca{K}|
\]
and call it the \textit{support set} of $\wt{\mf{S}}^\ep$. 
\end{quote} 

\begin{lem}\label{170329_1} 
For any neighborhood $U$ of $\bigcup_{\mf{p} \in \mf{P}} (\mca{K}_{\mf{p}} \cap s_{\mf{p}}^{-1}(0))$ in $|\mca{K}|$, 
there exists $\ep_0>0$ such that 
$0 < \ep  < \ep_0 \implies \Pi ((\wt{\mf{S}}^\ep)^{-1}(0)) \subset U$. 
\end{lem} 
\begin{proof}
If this is not the case, there exists a sequence $(\ep_m)_{m \ge 1}$ of positive real numbers converging to $0$, 
and a sequence $(x_m)_{m \ge 1}$ such that $x_m \in  \Pi ( (\wt{\mf{S}}^{\ep_m})^{-1}(0)) \setminus U$ for every $m \ge 1$. 
Up to subsequence $(x_m)_m$ has a limit in $|\mca{K}|$ which we denote by $x$. 

There exists $\mf{q} \in \mf{P}$ such that $x_m \in \mca{K}_{\mf{q}}$ for infinitely many $m$, thus we may assume that $x_m \in \mca{K}_{\mf{q}}$ for all $m$, 
then $x \in \mca{K}_{\mf{q}}$. It is sufficient to show that $s_{\mf{q}}(x)=0$, 
which implies $x_m \in U$ for sufficiently large $m$, a contradiction. 

Take a manifold chart $(V, E, \phi, \wh{\phi})$ of $(U_{\mf{q}}, \mca{E}_{\mf{q}})$ such that $x \in \phi(V)$
and $\mf{S}_{\mf{q}}$ is locally represented by $(W, \eta, (\mf{s}^\ep)_\ep)$
where $\mf{s}^\ep: V \times W \to E$ and $\eta$ is a compactly supported form on $W$. 
Let $s: V \to E$ denote the pullback of $s_{\mf{q}}$ by $\phi$. 
Then $\mf{s}^\ep(y, \xi) \to s(y) \,(y \in V, \, \xi \in W)$
as $\ep \to 0$ in the compact $C^1$ topology on $V \times W$. 

Take $y \in V$ so that $\phi(y)=x$ and 
$y_m \in V$ for sufficiently large $m$ so that $\phi(y_m)=x_m$. 
Then $y_m \to y$ and 
there exists $\xi_m \in \supp \eta$ such that $\mf{s}^{\ep_m}(y_m, \xi_m)=0$.
Since $\supp \eta$ is compact, this shows that $s(y) = 0$, 
which implies $s_{\mf{q}}(x)=0$. 
\end{proof} 

\begin{lem}\label{170330_1}
For any support system $\mca{K'} < \mca{K}$
(see Definition 5.6 (2) in \cite{FOOO_Kuranishi}), 
there exist $\delta>0$ and $\ep_0>0$ such that 
\[ 
0 < \ep  < \ep_0 \implies 
B_\delta ( \mca{K}'_{\mf{p}}) \cap \Pi ( (\wt{\mf{S}}^\ep)^{-1}(0)) \subset \mca{K}_{\mf{p}} \, (\forall \mf{p} \in \mf{P}).
\]
\end{lem} 
\begin{proof}
We take a support system $\mca{K}''$ such that 
$\mca{K}' < \mca{K}'' < \mca{K}$. 

\textbf{Step 1.} We first show that there exists $\delta>0$ such that 
\[ 
\bar{B}_\delta ( \mca{K}'_{\mf{p}}) \cap \bigcup_{\mf{q} \in \mf{P}} (\mca{K}_{\mf{q}} \cap s_{\mf{q}}^{-1}(0)) \subset \mca{K}''_{\mf{p}} \,(\forall \mf{p} \in \mf{P}). 
\]
Suppose that this is not the case. 
Then there exist $\mf{p} \in \mf{P}$ and a sequence $(x_m)_{m \ge 1}$ on $\bigcup_{\mf{q} \in \mf{P}} (\mca{K}_{\mf{q}} \cap s_{\mf{q}}^{-1}(0)) \setminus \mca{K}''_{\mf{p}}$
such that $d(x_m, \mca{K}'_{\mf{p}}) \to 0$ as $m \to \infty$. 

Since $\mf{P}$ is a finite set, there exists $\mf{q} \in \mf{P}$ such that $x_m \in \mca{K}_{\mf{q}} \cap s_{\mf{q}}^{-1}(0)$ for infinitely many $m$. 
We may assume that $x_m \in \mca{K}_{\mf{q}} \cap s_{\mf{q}}^{-1}(0)$ for all $m$, and 
there exists $x \in \mca{K}_{\mf{q}} \cap \mca{K}'_{\mf{p}}$ such that $d(x_m, x) \to 0$. 
Then $s_{\mf{q}}(x) = 0$, since $x_m \to x$ in $\mca{K}_{\mf{q}}$ and $s_{\mf{q}}(x_m) = 0$ for all $m$. 

Since $\mca{K}_{\mf{q}} \cap \mca{K}_{\mf{p}} \ne \emptyset$ in $|\wt{\mca{U}}|$, 
at least one of the following three cases holds: 
\begin{itemize} 
\item $\mf{p} \ge \mf{q}$: This case $\mca{K}_{\mf{q}} \setminus U_{\mf{p}}$ is compact, thus $d(\mca{K}_{\mf{q}} \setminus U_{\mf{p}}, \mca{K}'_{\mf{p}})>0$. 
Since $\lim_{m \to \infty} x_m = x$ and $x \in \mca{K}'_{\mf{p}}$, we obtain $x_m \in U_{\mf{p}}$ for sufficiently large $m$. 
Since $\mca{K}'_{\mf{p}} \subset \mtrg{\mca{K}}''_{\mf{p}}$ (this follows from $\mca{K}' < \mca{K}''$; see Definition 5.6 (2) in \cite{FOOO_Kuranishi}), 
we obtain $x_m \in \mca{K}''_{\mf{p}}$ for sufficiently large $m$, contradicting our assumption. 
\item $\mf{p} \le \mf{q}$ and $\dim U_{\mf{p}} = \dim U_{\mf{q}}$: This case $\mca{K}_{\mf{q}} \setminus U_{\mf{p}}$ is compact 
since $\Phi_{\mf{q}\mf{p}}$ is an open embedding, and we obtain a contradiction as in the previous case. 
\item $\mf{p} < \mf{q}$ and $\dim U_{\mf{p}} < \dim U_{\mf{q}}$: 
Since $\lim_{m \to \infty} x_m = x \in \mca{K}'_{\mf{p}} \subset  \mtrg{\mca{K}}''_{\mf{p}}$ and $x_m \notin \mca{K}''_{\mf{p}}$ for every $m$, 
we obtain $x_m \notin \ph_{\mf{q}\mf{p}}(U_{\mf{q}\mf{p}})$ for sufficiently large $m$. 
On the other hand $x = \lim_{m \to \infty} x_m$ satisfies $s_{\mf{q}}(x)=0$ and 
$D_x s_{\mf{q}} : T_xU_{\mf{q}}/ T_x U_{\mf{p}} \to (\mca{E}_{\mf{q}})_x/(\mca{E}_{\mf{p}})_x$
is an isomorphism, 
which implies that $s_{\mf{q}}(x_m) \ne 0$
for sufficiently large every $m$, 
contradicting our assumption. 
\end{itemize} 

\textbf{Step 2.} 
Taking $\delta$ as in Step 1, we prove that there exists $\ep_0>0$ such that 
\[
0 < \ep < \ep_0 \implies B_\delta(\mca{K}'_{\mf{p}}) \cap \Pi ( (\wt{\mf{S}}^\ep)^{-1}(0)) \subset  \mca{K}_{\mf{p}} \quad(\forall \mf{p} \in \mf{P}). 
\]
If this is not the case, there exist
a sequence $(\ep_m)_{m \ge 1} $ of positive real numbers converging to $0$, 
$\mf{p} \in \mf{P}$, and 
a sequence $(x_m)_{m \ge 1}$ such that
\[ 
x_ m \in (B_\delta(\mca{K}'_{\mf{p}}) \cap \Pi ( (\wt{\mf{S}}^{\ep_m})^{-1}(0)) ) \setminus \mca{K}_{\mf{p}}
\] 
for every $m$. 
Up to subsequence $(x_m)_m$ has a limit $x$ in $|\mca{K}|$. 
By Lemma \ref{170329_1}, $x \in \bigcup_{\mf{q} \in \mf{P}} (\mca{K}_{\mf{q}} \cap s_{\mf{q}}^{-1}(0))$. 
Also $x \in \bar{B}_\delta(\mca{K}'_{\mf{p}})$, thus by Step 1 we get $x \in \mca{K}''_{\mf{p}}$. 

There exists $\mf{q} \in \mf{P}$ such that $x_m \in \mca{K}_{\mf{q}}$ for infinitely many $m$, 
then we may assume that $x_m \in \mca{K}_{\mf{q}}$ for all $m$, in particular $x \in \mca{K}_{\mf{q}}$. 
Then $\mca{K}_{\mf{q}} \cap \mca{K}''_{\mf{p}} \ne \emptyset$, 
thus at least one of the following three cases holds: 
\begin{itemize} 
\item $\mf{p} \ge \mf{q}$: 
This case $\mca{K}_{\mf{q}} \setminus U_{\mf{p}}$ is compact, thus $d(\mca{K}_{\mf{q}} \setminus U_{\mf{p}}, \mca{K}''_{\mf{p}})>0$. 
Since $\lim_{m \to \infty} x_m = x \in \mca{K}''_{\mf{p}}$, 
we obtain $x_m \in U_{\mf{p}}$ for sufficiently large $m$. 
By $\mca{K}''_{\mf{p}} \subset \mtrg{\mca{K}}_{\mf{p}}$, 
we obtain $x_m \in \mca{K}_{\mf{p}}$ for sufficiently large $m$, contradicting our assumption. 
\item $\mf{p} \le  \mf{q}$ and $\dim U_{\mf{p}} = \dim U_{\mf{q}}$: 
Again $\mca{K}_{\mf{q}} \setminus U_{\mf{p}}$ is compact, and we obtain a contradiction as in the previous case. 
\item $\mf{p} < \mf{q}$ and $\dim U_{\mf{p}} < \dim U_{\mf{q}}$: 
Since $\lim_{m \to \infty} x_m = x \in \mtrg{\mca{K}_{\mf{p}}}$ and 
$x_m \notin \mca{K}_{\mf{p}}$ for every $m \ge 1$, 
we obtain $x_m \notin \ph_{\mf{q}\mf{p}}(U_{\mf{q}\mf{p}})$ for sufficiently large $m$. 
Since the CF-perturbation $\mf{S}^{\ep_m}_{\mf{q}}$ converges to $s_{\mf{q}}$ in compact $C^1$-topology, 
$s_{\mf{q}}(x)=0$ thus 
$D_x s_{\mf{q}} : T_xU_{\mf{q}}/ T_x U_{\mf{p}} \to (\mca{E}_{\mf{q}})_x/(\mca{E}_{\mf{p}})_x$ is an isomorphism, 
we obtain $x_m \notin \Pi ((\mf{S}_{\mf{q}}^{\ep_m})^{-1}(0))$ for sufficiently large $m$, 
contradicting our assumption. 
\end{itemize} 
\end{proof} 

\begin{lem} 
For sufficiently small $\delta>0$ there holds
\[ 
\mf{q} < \mf{p} \implies B_\delta(\mca{K}'_{\mf{p}}) \cap \mca{K}_{\mf{q}} \subset U_{\mf{p}\mf{q}}. 
\]
\end{lem}
\begin{proof}
Suppose that this is not the case. 
Then there exist $\mf{p}$, $\mf{q}$ such that 
$\mf{q} < \mf{p}$ and a sequence $(\delta_m)_m$ converging to $0$ 
and a sequence $(x_m)_m$ such that $x_m \in (B_{\delta_m}(\mca{K}'_{\mf{p}}) \cap \mca{K}_{\mf{q}}) \setminus U_{\mf{p}\mf{q}}$
for every $m$. 
Then the sequence $(x_m)_m$ has a limit point $x \in \mca{K}_{\mf{q}} \setminus U_{\mf{p}\mf{q}}$, 
then $d(x, \mca{K}'_{\mf{p}})>0$, 
contradicting $x_m \in B_{\delta_m}(\mca{K}'_{\mf{p}})\,(\forall m)$. 
\end{proof} 

\subsubsection{Definition and well-definedness}

We start from data $X$, $\wt{\mca{U}}$, $\wt{\omega}$, $\mca{K}$, $\wt{\mf{S}}$ and $\wt{f}$. 
We take the following choices
(our definition of partition of unity is slightly different from that of \cite{FOOO_Kuranishi}, Definition 7.64). 
\begin{itemize}
\item A support system $\mca{K}' < \mca{K}$.
\item A positive real number $\delta$ such that 
\begin{itemize}
\item There exists $\ep>0$ such that 
\[ 
0 < \ep'   < \ep  \implies 
B_{\delta}   ( \mca{K}'_{\mf{p}}) \cap \Pi ( (\wt{\mf{S}}^{\ep'})^{-1}(0)) \subset \mca{K}_{\mf{p}} \quad (\forall \mf{p} \in \mf{P}).
\]
\item $2 \delta < d(\mca{K}'_{\mf{p}}, \mca{K}'_{\mf{q}})$ for any $\mf{p}, \mf{q} \in \mf{P}$ such that $U_{\mf{p}} \cap U_{\mf{q}} = \emptyset$. 
\item $\mf{q} < \mf{p} \implies B_{\delta}(\mca{K}'_{\mf{p}}) \cap \mca{K}_{\mf{q}} \subset U_{\mf{p}\mf{q}}$.
\item  $\mca{K}'_{\mf{p}}(\delta):= \{ x \in \mca{K}_{\mf{p}} \mid d(\mca{K}'_{\mf{p}}, x) \le \delta \} \subset \mtrg{\mca{K}_{\mf{p}}}$ for every $\mf{p} \in \mf{P}$. 
\end{itemize}
\item 
Partition of unity $\chi = (\chi_{\mf{p}})_{\mf{p} \in \mf{P}}$ of $(X, \wt{\mca{U}}, \mca{K}', \delta)$. 
Namely, the following conditions are satisfied: 
\begin{itemize}
\item $\chi_{\mf{p}} : |\mca{K}| \to [0,1]$ is a strongly smooth function for every $\mf{p} \in \mf{P}$.
Namely, for every $\mf{q} \in \mf{P}$, 
$\chi_{\mf{p}}|_{\mca{K}_{\mf{q}}}$ can be extended to a $C^\infty$-function defined on a neighborhood of $\mca{K}_{\mf{q}}$ in $U_{\mf{q}}$. 
\item $\supp \chi_{\mf{p}} \subset B_{\delta}(\mca{K}'_{\mf{p}})$ for every $\mf{p} \in \mf{P}$. 
\item $\sum_{\mf{p} \in \mf{P}} \chi_{\mf{p}} = 1$ on a neighborhood of $\bigcup_{\mf{p} \in \mf{P}} s_{\mf{p}}^{-1}(0) \subset |\mca{K}|$. 
\end{itemize}
\end{itemize}
Then we define 
\[ 
\wt{f}_*(X, \wt{\mca{U}}, \wt{\omega}, \wt{\mf{S}}^\ep):= \sum_{\mf{p} \in \mf{P}} (f_{\mf{p}})_*( \mca{U}_{\mf{p}}, \chi_{\mf{p}} \omega_{\mf{p}} , \mf{S}^\ep_{\mf{p}}).
\]
$ (f_{\mf{p}})_*( \mca{U}_{\mf{p}}, \chi_{\mf{p}} \omega_{\mf{p}} , \mf{S}^\ep_{\mf{p}})$
in the RHS makes sense, since for every $\mf{p} \in \mf{P}$ there holds 
\[ 
\supp \chi_{\mf{p}} \cap U_{\mf{p}} \subset 
B_\delta(\mca{K}'_{\mf{p}}) \cap U_{\mf{p}} \subset 
\mca{K}'_{\mf{p}} (\delta) \subset \mtrg{\mca{K}}_{\mf{p}}
\]
and $\mf{S}_{\mf{p}} \in \mca{S}(\mca{K}_{\mf{p}})$. 
We need to recall the following definition (Definition 7.79 in \cite{FOOO_Kuranishi}): 

\begin{defn}\label{180301_1} 
Let $\mca{A}$ and $\mca{X}$ be sets, and 
$(F_a)_{a \in \mca{A}}$ be a family of maps such that 
$F_a: (0, \ep_a) \to \mca{X}$ for each $a \in \mca{A}$. 
We say that $F_a$ is independent on choices of $a$ in the sense of $\spadesuit$ if the following holds:
\begin{quote} 
$\spadesuit$: For any $a_1, a_2 \in \mca{A}$, there exists $0 < \ep_0 < \min \{\ep_{a_1}, \ep_{a_2}\}$ such that 
$F_{a_1}(\ep) = F_{a_2}(\ep)$ for every $\ep \in (0, \ep_0)$.
\end{quote}
\end{defn} 

\begin{lem}\label{170330_1.5}
The above definition of 
$\wt{f}_*(X, \wt{\mca{U}}, \wt{\omega}, \wt{\mf{S}}^\ep)$
is independent on choices of $\mca{K}'$, $\delta$, $\chi$
in the sense of $\spadesuit$. 
\end{lem}
\begin{proof} 
Let us take two choices $((\mca{K}')^i, \delta^i, \chi^i )_{i=1,2}$. 
We are going to prove 
\[ 
\sum_{\mf{p} \in \mf{P}} (f_{\mf{p}})_*( \mca{U}_{\mf{p}}, \chi^1_{\mf{p}} \omega_{\mf{p}}, \mf{S}^\ep_{\mf{p}}) = 
\sum_{\mf{p} \in \mf{P}} (f_{\mf{p}})_*( \mca{U}_{\mf{p}}, \chi^2_{\mf{p}} \omega_{\mf{p}}, \mf{S}^\ep_{\mf{p}})
\]
when $\ep$ is sufficiently small. 

By our assumption, $\sum_{\mf{p} \in \mf{P}} \chi^1_{\mf{p}} = \sum_{\mf{p} \in \mf{P}} \chi^2_{\mf{p}} = 1$ on a neighborhood of $\bigcup_{\mf{p} \in \mf{P}} s^{-1}_{\mf{p}}(0)$. 
Thus, Lemma \ref{170329_1} implies that, 
for sufficiently small $\ep>0$ and every $\mf{p} \in \mf{P}$, 
there holds $\sum_{\mf{p} \in \mf{P}} \chi^1_{\mf{p}} = \sum_{\mf{p} \in \mf{P}} \chi^2_{\mf{p}} = 1$
on $\mca{K}_{\mf{p}} \cap \Pi ( (\mf{S}^\ep_{\mf{p}})^{-1}(0))$. 
Then we obtain 
\begin{align*} 
\sum_{\mf{p} \in \mf{P}} (f_{\mf{p}})_*( \mca{U}_{\mf{p}}, \chi^1_{\mf{p}} \omega_{\mf{p}}, \mf{S}^\ep_{\mf{p}})&= \sum_{\mf{p}_1, \mf{p}_2} (f_{\mf{p}_1})_* ( \mca{U}_{\mf{p}_1}, \chi^1_{\mf{p}_1} \chi^2_{\mf{p}_2} \omega_{\mf{p}_1} , \mf{S}^\ep_{\mf{p}_1}) \\
\sum_{\mf{p} \in \mf{P}} (f_{\mf{p}})_*( \mca{U}_{\mf{p}}, \chi^2_{\mf{p}} \omega_{\mf{p}}, \mf{S}^\ep_{\mf{p}})&=
\sum_{\mf{p}_1, \mf{p}_2} (f_{\mf{p}_2})_* ( \mca{U}_{\mf{p}_2}, \chi^1_{\mf{p}_1} \chi^2_{\mf{p}_2} \omega_{\mf{p}_2} , \mf{S}^\ep_{\mf{p}_2}). 
\end{align*} 
Then it is sufficient to show 
\[ 
 (f_{\mf{p}_1})_* ( \mca{U}_{\mf{p}_1}, \chi^1_{\mf{p}_1} \chi^2_{\mf{p}_2} \omega_{\mf{p}_1} , \mf{S}^\ep_{\mf{p}_1}) = 
 (f_{\mf{p}_2})_* ( \mca{U}_{\mf{p}_2}, \chi^1_{\mf{p}_1} \chi^2_{\mf{p}_2} \omega_{\mf{p}_2} , \mf{S}^\ep_{\mf{p}_2})
\] 
for any $\mf{p}_1, \mf{p}_2 \in \mf{P}$. 
We may assume either $\mf{p}_1 \le \mf{p}_2$ or $\mf{p}_1 \ge \mf{p}_2$, 
since otherwise $U_{\mf{p}_1} \cap U_{\mf{p}_2} = \emptyset$, thus 
$\chi^1_{\mf{p}_1} \chi^2_{\mf{p}_2} \equiv 0$ by $2\delta < d(\mca{K}'_{\mf{p}_1}, \mca{K}'_{\mf{p}_2})$. 
Also this equality is clear when $\mf{p}_1 = \mf{p}_2$. 
Thus we may assume that $\mf{p}_1 < \mf{p}_2$. 

In the argument below, 
we abbreviate 
$\mf{p}_i$ by $i$. For example $\mca{U}_{\mf{p}_i}$ is abbreviated by $\mca{U}_i$. 
Let us introduce the following notations: 
\begin{itemize} 
\item $\mca{U}_{21}$ denotes $\mca{U}_1|_{U_{21}}  = (\ph_{21})^* \mca{U}_2$.
\item $f_{21}$ denotes $f_1|_{U_{21}} = f_2 \circ \ph_{21}$. 
\item $\omega_{21}$ denotes $\omega_1|_{U_{21}} = (\ph_{21})^* \omega_2$. 
\item $\mf{S}_{21}$ denotes $\mf{S}_1|_{U_{21}} = (\ph_{21})^* \mf{S}_2$. 
\end{itemize} 
Then there holds 
\[ 
(f_1)_* (\mca{U}_1, \chi^1_1 \chi^2_2 \omega_1 , \mf{S}^\ep_1) = (f_{21})_* (\mca{U}_{21}, \chi^1_1 \chi^2_2 \omega_{21}, \mf{S}^\ep_{21})
\] 
since $\supp \chi^2_{\mf{p}_2} \subset B_\delta(\mca{K}'_{\mf{p}_2})$ and 
$B_\delta(\mca{K}'_{\mf{p}_2}) \cap \mca{K}_{\mf{p}_1} \subset U_{\mf{p}_2\mf{p}_1}$.
On the other hand 
\[ 
(f_2)_* (\mca{U}_2, \chi^1_1 \chi^2_2 \omega_2 , \mf{S}^\ep_2) = (f_{21})_* (\mca{U}_{21}, \chi^1_1 \chi^2_2 \omega_{21}, \mf{S}^\ep_{21})
\] 
when $\ep$ is sufficiently small, 
since $\supp \chi^1_{\mf{p}_1} \subset B_\delta(\mca{K}'_{\mf{p}_1})$ and 
$B_\delta(\mca{K}'_{\mf{p}_1}) \cap \Pi ((\wt{\mf{S}}^\ep)^{-1}(0)) \subset \mca{K}_{\mf{p}_1}$
by Lemma \ref{170330_1}. 
\end{proof}

Next we prove the invariance by GG-embedding 
(for the definition of GG-embedding, see Definition 3.24 in \cite{FOOO_Kuranishi}). 
Lemma \ref{170330_3} below is an analogue of Proposition 9.16 in \cite{FOOO_Kuranishi}. 

\begin{lem}\label{170330_3} 
Let us consider $(\wt{\mca{U}}_i, \mca{K}_i, \wt{\mf{S}}_i, \wt{f}_i, \wt{\omega}_i)_{i=1,2}$
and $\wt{\Phi}$ such that 
\begin{itemize} 
\item For each $i \in \{1,2\}$, the tuple $(\wt{\mca{U}}_i, \mca{K}_i, \wt{\mf{S}}_i, \wt{f}_i, \wt{\omega}_i)$ 
satisfies the conditions to define $(\wt{f}_i)_*(X, \wt{\mca{U}}_i, \wt{\omega}_i, \wt{\mf{S}}^\ep_i)$ for sufficiently small $\ep>0$. 
\item $\wt{\Phi}:  \wt{\mca{U}_1} \to \wt{\mca{U}_2}$ is a GG-embedding. 
Namely, $\wt{\Phi} = (\mf{i}, (\Phi_{\mf{p}})_{\mf{p} \in \mf{P}_1})$ where $\mf{i}: \mf{P}_1 \to \mf{P}_2$ is an order preserving map, 
and $\Phi_{\mf{p}} = (\ph_{\mf{p}}, \wh{\ph}_{\mf{p}}):  (\mca{U}_1)_{\mf{p}} \to (\mca{U}_2)_{\mf{i}(\mf{p})}$ is an embedding of K-charts for each $\mf{p} \in \mf{P}_1$, 
such that compatibilities in Definition 3.24 in \cite{FOOO_Kuranishi} are satisfied. 
\item $\mca{K}_1$, $\mca{K}_2$ and $\wt{\mf{S}}_1$, $\wt{\mf{S}}_2$ are compatible with $\wt{\Phi}$ (see Definition 9.3 (3), (4) in \cite{FOOO_Kuranishi}). 
Moreover $\wt{\omega}_1 = \wt{\Phi}^* \wt{\omega}_2$, $\wt{f}_1 = \wt{f}_2 \circ \wt{\Phi}$. 
\end{itemize} 
Then, for sufficiently small $\ep>0$, there holds 
\[ 
(\wt{f}_1)_*(X, \wt{\mca{U}}_1, \wt{\omega}_1, \wt{\mf{S}}^\ep_1) = 
(\wt{f}_2)_*(X, \wt{\mca{U}}_2, \wt{\omega}_2, \wt{\mf{S}}^\ep_2). 
\]
\end{lem}
\begin{proof}

Let us take a support system $\mca{K}'_i$ of $\wt{\mca{U}_i}$ for $i=1,2$ such that: 
\begin{itemize} 
\item $\mca{K}'_i < \mca{K}_i$ for $i=1, 2$. 
\item $\mca{K}'_1$ and $\mca{K}'_2$ are compatible with $\wt{\Phi}$. 
Namely, $\ph_{\mf{p}}((\mca{K}'_1)_{\mf{p}}) \subset \mtrg{(\mca{K}'_2)}_{\mf{i}(\mf{p})}$ for every $\mf{p} \in \mf{P}_1$. 
\end{itemize} 
First we need Lemma \ref{171208_2} below, 
whose proof is almost the same as the proof of Proposition 7.67 in \cite{FOOO_Kuranishi}. 
In the statement, we take a metric on $|\mca{K}_2|$ and 
its pullback to $|\mca{K}_1|$ via a natural embedding map 
$|\mca{K}_1| \to |\mca{K}_2|$.
This embedding map exists since 
$\mca{K}_1$ and $\mca{K}_2$ are compatible with $\wt{\Phi}$. 

\begin{lem}\label{171208_2}
For any $\delta>0$, there exists 
$\chi^+_1  = ((\chi^+_1)_{\mf{p}}) _{\mf{p} \in \mf{P}_1}$
which satisfies the following conditions: 
\begin{itemize}
\item $(\chi^+_1)_{\mf{p}}$ is a strongly smooth map from $|\mca{K}_2|$ to $[0,1]$ for every $\mf{p} \in \mf{P}_1$.
\item $\supp (\chi^+_1)_{\mf{p}} \subset B_\delta( \ph_{\mf{p}} (  (\mca{K}'_1)_{\mf{p}}))$ for every $\mf{p} \in \mf{P}_1$. 
Here $B_\delta$ denotes the $\delta$-neighborhood in $|\mca{K}_2|$. 
\item $\sum_{\mf{p} \in \mf{P}_1} (\chi^+_1)_{\mf{p}} = 1$ on a neighborhood of $\bigcup_{\mf{q} \in \mf{P}_2} s_{\mf{q}}^{-1}(0) \subset |\mca{K}_2|$. 
\end{itemize} 
\end{lem}
\begin{proof}
There exists a natural homeomorphism $\bigcup_{\mf{q} \in \mf{P}_2} s_{\mf{q}}^{-1}(0) \to X$. 
In the following argument we identify these two spaces, and in particular consider $X$ as a subspace of $|\mca{K}_2|$. 
First we need the following fact, which is essentially the same as Lemma 7.66 in \cite{FOOO_Kuranishi}: 
\begin{quote}
($\star$): For any open set $W$ of $|\mca{K}_2|$ containing a compact subset $K$ of $X$, 
there exists a strongly smooth function $g: |\mca{K}_2| \to [0,1]$ that has a compact support
contained in $W$ and $1$ on a neighborhood of $K$. 
\end{quote}

Applying the claim ($\star$) to $K: =\ph_{\mf{p}}((\mca{K}'_1)_{\mf{p}})$ and $W:= B_{\delta}(\ph_{\mf{p}}((\mca{K}'_1)_{\mf{p}}))$, 
we obtain $f_{\mf{p}}: |\mca{K}_2| \to [0, 1]$. 
Applying the claim ($\star$) to $K:= X$ and 
$W:= \bigg\{ x \in |\mca{K}_2| \,  \bigg{|} \sum_{\mf{q} \in \mf{P}_1} f_{\mf{q}}(x)>1/2 \bigg\}$, 
we obtain $g: |\mca{K}_2| \to [0,1]$. 
Finally, for each $\mf{p} \in \mf{P}_1$ 
we define $(\chi^+_1)_{\mf{p}}: |\mca{K}_2| \to [0,1]$ by  
\[ 
(\chi^+_1)_{\mf{p}}(x):= 
\begin{cases} 
g(x) f_{\mf{p}}(x) \biggl( \sum_{\mf{q} \in \mf{P}_1} f_{\mf{q}}(x) \biggr)^{-1} &(x \in W) \\ 
0 &(x \notin W). 
\end{cases}
\]
\end{proof}
For each $\mf{p} \in \mf{P}_1$, we define $(\chi_1)_{\mf{p}}: |\mca{K}_1| \to [0,1]$ by 
\[ 
(\chi_1)_{\mf{p}} |_{(\mca{K}_1)_{\mf{q}}} := (\chi^+_1)_{\mf{i}(\mf{p})} \circ \ph_{\mf{q}}|_{(\mca{K}_1)_{\mf{q}}}. 
\] 
Then $\chi_1 = ((\chi_1)_{\mf{p}})_{\mf{p} \in \mf{P}_1}$ is a partition of unity of 
$(X, \wt{\mca{U}}_1 , \mca{K}'_1, \delta)$. 
We also take a partition of unitiy of 
$(X, \wt{\mca{U}}_2, \mca{K}'_2, \delta)$, which we denote by $\chi_2 = ((\chi_2)_{\mf{q}})_{\mf{q} \in \mf{P}_2}$. 

Now we can complete the proof as follows: 
\begin{align*} 
&(\wt{f}_1)_*(X, \wt{\mca{U}}_1, \wt{\omega}_1,  \wt{\mf{S}}^\ep_1 ) \\
&\quad = \sum_{\mf{p} \in \mf{P}_1}  ( (f_1)_{\mf{p}})_* ((\mca{U}_1)_{\mf{p}}, (\chi_1)_{\mf{p}} (\omega_1)_{\mf{p}}, (\mf{S}^\ep_1)_{\mf{p}}) \\ 
&\quad = \sum_{\mf{p} \in \mf{P}_1}
((f_2)_{\mf{i}(\mf{p})})_*( (\mca{U}_2)_{\mf{i}(\mf{p})}, (\chi^+_1)_{\mf{i}(\mf{p})} (\omega_2)_{\mf{i}(\mf{p})} , (\mf{S}^\ep_2)_{\mf{i}(\mf{p})}) \\ 
&\quad = \sum_{\mf{q} \in \mf{P}_2}   ((f_2)_{\mf{q}} )_*((\mca{U}_2)_{\mf{q}}, (\chi_2)_{\mf{q}} (\omega_2)_{\mf{q}} , (\mf{S}^\ep_2)_{\mf{q}}) \\ 
&\quad = (\wt{f}_2)_*(X, \wt{\mca{U}}_2, \wt{\omega}_2, \wt{\mf{S}}^\ep_2). 
\end{align*} 
The first and fourth equality follows from the definition. 
The second equality holds since
$(\omega_1)_{\mf{p}} = (\ph_{\mf{p}})^*((\omega_2)_{\mf{i}(\mf{p})})\, (\forall \mf{p} \in \mf{P}_1)$ and 
\[ 
\supp (\chi^+_1)_{\mf{p}} \cap \Pi ((\mf{S}^\ep_2)^{-1}_{\mf{i}(\mf{p})}(0)) \subset
B_\delta(\ph_{\mf{p}}((\mca{K}'_1)_{\mf{p}})) \cap \Pi ( (\mf{S}^\ep_2)^{-1}_{\mf{i}(\mf{p})} (0)) 
\subset \ph_{\mf{p}} ((\mca{K}_1)_{\mf{p}})
\] 
when $\delta$ and $\ep$ are sufficiently small
(this can be proved by arguments similar to the proof of Lemma \ref{170330_1}). 
Proof of the third equality is similar to the proof of Lemma \ref{170330_1.5}. 
This completes the proof of Lemma \ref{170330_3}. 
\end{proof} 

Let us state and prove Stokes' formula. 

\begin{prop}\label{171109_1} 
For sufficiently small $\ep>0$, there holds 
\[ 
\partial (\wt{f}_*(X, \wt{\mca{U}}, \wt{\omega}, \wt{\mf{S}}^\ep) ) 
= (-1)^{|\omega|+1}  \wt{f}_*(X, \wt{\mca{U}}, d\wt{\omega}, \wt{\mf{S}}^\ep).
\]
\end{prop}
\begin{proof}
Take $\mca{K}' < \mca{K}$, $\delta>0$ and let $(\chi_{\mf{p}})_{\mf{p} \in \mf{P}}$ be a partition of unity of $(X, \wt{\mca{U}}, \mca{K}', \delta)$.
When $\ep>0$ is sufficiently small, 
\begin{align*} 
&\partial (\wt{f}_*(X, \wt{\mca{U}}, \wt{\omega}, \wt{\mf{S}}^\ep) ) 
- (-1)^{|\omega|+1}  \wt{f}_*(X, \wt{\mca{U}}, d\wt{\omega}, \wt{\mf{S}}^\ep) \\ 
&\quad = 
(-1)^{|\omega|+1} \sum_{\mf{p} \in \mf{P} } (f_{\mf{p}})_* (\mca{U}_{\mf{p}}, d \chi_{\mf{p}} \omega_{\mf{p}} , \mf{S}^\ep_{\mf{p}}) \\ 
&\quad = 
(-1)^{|\omega|+1} \sum_{\mf{p}, \mf{q} \in \mf{P}} (f_{\mf{p}})_* (\mca{U}_{\mf{p}}, \chi_{\mf{q}} d \chi_{\mf{p}} \omega_{\mf{p}}, \mf{S}^\ep_{\mf{p}}) \\
&\quad = 
(-1)^{|\omega|+1}
\sum_{\mf{p}, \mf{q} \in \mf{P}}  (f_{\mf{q}})_* (\mca{U}_{\mf{q}}, \chi_{\mf{q}} d\chi_{\mf{p}} \omega_{\mf{q}}, \mf{S}^\ep_{\mf{q}}) \\
&\quad = 
(-1)^{|\omega|+1}
\sum_{\mf{q} \in \mf{P}} (f_{\mf{q}})_*(\mca{U}_{\mf{q}}, \chi_{\mf{q}}  (\sum_{\mf{p} \in \mf{P}} d\chi_{\mf{p}}) \omega_{\mf{q}}, \mf{S}^\ep_{\mf{q}}) = 0. 
\end{align*} 
The first equality follows from Stokes' formula for K-charts: 
Lemma \ref{170323_1} (iii). 
The second and fifth equality holds since $\sum_{\mf{p} \in \mf{P}} \chi_{\mf{p}} = 1$ on 
$\Pi( (\wt{\mf{S}}^\ep)^{-1}(0))$ when $\ep$ is sufficiently small. 
The third equality holds by the argument similar to the proof of Lemma \ref{170330_1.5}. 
The fourth equality is obvious. 
\end{proof} 

\subsection{K-space} 

Let $(X, \wh{\mca{U}})$ be a compact, oriented K-space with a strongly smooth map 
$\wh{f}: (X, \wh{\mca{U}}) \to \mca{L}_{k+1}$, 
a differential form $\wh{\omega}$ on $(X, \wh{\mca{U}})$ and a CF-perturbation 
$\wh{\mf{S}} = (\wh{\mf{S}}^\ep)_{\ep \in (0,1]}$ on $(X, \wh{\mca{U}})$ 
which is transversal to $0$ and $\ev_0  \circ \wh{f}: (X, \wh{\mca{U}}) \to L$ is strongly submersive with respect to $\wh{\mf{S}}$. 
Under these assumptions we define a de Rham chain 
\[ 
\wh{f}_* (X, \wh{\mca{U}}, \wh{\omega}, \wh{\mf{S}}^\ep) \in C^\dR_*(\mca{L}_{k+1})
\]
for sufficiently small $\ep>0$, and check Stokes' formula and fiber product formula. 
We only consider K-spaces \textit{without} boundaries, 
since generalization to admissible  K-spaces with boundaries (and corners) 
will be straightforward. 

By Lemma 9.10 in \cite{FOOO_Kuranishi}, 
there exist 
\begin{itemize}
\item $\wt{\mca{U}}$: GCS of $X$, 
\item $\wt{\omega}$: differential form on $\wt{\mca{U}}$, 
\item $\wt{f}$: strongly smooth map from $(X, \wt{\mca{U}})$ to $\mca{L}_{k+1}$, 
\item $\mca{K}$: support system of $\wt{\mca{U}}$, 
\item $\wt{\mf{S}}$: CF-perturbation of $(\wt{\mca{U}}, \mca{K})$ which is transversal to $0$, and 
$\ev_0 \circ \wt{f}: (X, \wt{\mca{U}}) \to L$ is strongly submersive with respect to $\wt{\mf{S}}$, 
\item strict KG-embedding $\wh{\Phi}: \wh{\mca{U}}_0 \to \wt{\mca{U}}$, where $\wh{\mca{U}_0}$ is an open substructure of $\wh{\mca{U}}$,  
\end{itemize} 
satisfying the following compatibilities:
\begin{itemize} 
\item $\wh{\omega}|_{\wh{\mca{U}_0}} = \wh{\Phi}^*(\wt{\omega})$, 
\item $\wh{f}|_{\wh{\mca{U}_0}} = \wt{f} \circ \wh{\Phi}$, 
\item $\wh{\mf{S}}|_{\wh{\mca{U}_0}} = \wh{\Phi}^*(\wt{\mf{S}})$. 
\end{itemize} 
Then we define 
\begin{equation}\label{171208_1} 
\wh{f}_*(X, \wh{\mca{U}}, \wh{\omega}, \wh{\mf{S}^\ep}):= \wt{f}_*(X, \wt{\mca{U}}, \wt{\omega}, \wt{\mf{S}^\ep})
\end{equation} 
for sufficiently small $\ep>0$. 
Well-definedness
(i.e. the RHS of (\ref{171208_1}) does not depend on choices of $\wt{\mca{U}}$, $\wt{\omega}$, etc. in the sense of $\spadesuit$)
follows from invariance by GG-embedding (Lemma \ref{170330_3})
and arguments in \cite{FOOO_Kuranishi} Section 9.2
(proof of Proposition 9.16 $\implies$ Theorem 9.14). 
Stokes' formula in this setting (Theorem \ref{170628_2})
follows from Stokes' formula for GCS (Theorem \ref{171109_1}). 

Finally we give a sketch of the proof 
of the fiber product formula (Theorem \ref{170628_3}), 
imitating the proof of Proposition 10.23 in \cite{FOOO_Kuranishi}. 
Namely, we prove 
\begin{equation}\label{171110_1} 
(\wh{f}_{12})_*(X_{12}, \wh{\mca{U}}_{12}, \wh{\omega}_{12}, \wh{\mf{S}}_{12}^\ep) = 
(\wh{f}_1)_*(X_1, \wh{\mca{U}}_1, \wh{\omega}_1, \wh{\mf{S}}_1^\ep) \circ_j
(\wh{f}_2)_*(X_2, \wh{\mca{U}}_2, \wh{\omega}_2, \wh{\mf{S}}_2^\ep)
\end{equation} 
for sufficiently small $\ep>0$, where $(X_{12}, \wh{\mca{U}}_{12})$ denotes the fiber product of 
$(X_1, \wh{\mca{U}}_1)$ and $(X_2, \wh{\mca{U}}_2)$.
For each $i \in \{1, 2\}$, there exist a GCS $\wt{\mca{U}}_i$ on $X_i$ and a KG-embedding 
$\wh{\mca{U}}_i \to \wt{\mca{U}}_i$,
namely a strict KG-embedding 
$(\wh{\mca{U}}_i)_0 \to \wt{\mca{U}}_i$, 
where $(\wh{\mca{U}}_i)_0$ is an open substructure of $\wh{\mca{U}}_i$. 
Let $(\wh{\mca{U}}_{12})_0$ denote the fiber product of 
$(\wh{\mca{U}}_1)_0$ and $(\wh{\mca{U}}_2)_0$. 
One may assume that, for each $i \in \{1, 2\}$
there exist $\wt{\omega}_i$, $\wt{f}_i$, $\wt{\mf{S}}_i$ satisfying compatibilities. 

Let $\chi^i = (\chi^i_{\mf{p}_i})_{\mf{p}_i \in \mf{P}_i}$ be a partition of unity of $(X_i, \wt{\mca{U}}_i)$. 
For each $\mf{p}_i \in \mf{P}_i$, 
$\chi^i_{\mf{p}_i}$ induces a strongly smooth map 
$\wh{\chi}^i_{\mf{p}_i}: (X_i, \wh{\mca{U}}_i) \to [0, 1]$, 
and there holds 
$\sum_{\mf{p}_i \in \mf{P}_i} \wh{\chi}^i_{\mf{p}_i} = 1$. 
Then  (\ref{171110_1}) reduces to proving 
\begin{align}\label{171110_2} 
&(\wh{f}_{12})_*(X_{12}, (\wh{\mca{U}}_{12})_0, \wh{\chi}^1_{\mf{p}_1} \wh{\chi}^2_{\mf{p}_2} \wh{\omega}_{12}, \wh{\mf{S}}^\ep_{12}) =\\ 
&\qquad\qquad (\wh{f}_1)_*(X_1, (\wh{\mca{U}}_1)_0, \wh{\chi}^1_{\mf{p}_1} \wh{\omega}_1, \wh{\mf{S}}^\ep_1)  \circ_j 
(\wh{f}_2)_*(X_2, (\wh{\mca{U}}_2)_0, \wh{\chi}^2_{\mf{p}_2} \wh{\omega}_2, \wh{\mf{S}}^\ep_2) \nonumber
\end{align}
for every $\mf{p}_1 \in \mf{P}_1$ and $\mf{p}_2 \in \mf{P}_2$. 
Now there holds
\[
(\wh{f}_1)_*(X_1, (\wh{\mca{U}}_1)_0, \wh{\chi}^1_{\mf{p}_1} \wh{\omega}_1, \wh{\mf{S}}^\ep_1) = 
((f_1)_{\mf{p}_1})_*( (\mca{U}_1)_{\mf{p}_1}, (\chi_1)_{\mf{p}_1} (\omega_1)_{\mf{p}_1} , (\mf{S}_1)^\ep_{\mf{p}_1}),
\] 
and similar formulas hold for $(\wh{f}_{12})_*(X_{12}, (\wh{\mca{U}}_{12})_0, \wh{\chi}^1_{\mf{p}_1} \wh{\chi}^2_{\mf{p}_2} \wh{\omega}_{12}, \wh{\mf{S}}^\ep_{12})$
and $(\wh{f}_2)_*(X_2, (\wh{\mca{U}}_2)_0, \wh{\chi}^2_{\mf{p}_2} \wh{\omega}_2, \wh{\mf{S}}^\ep_2)$. 
Then (\ref{171110_2}) follows from the fiber product formula for single K-charts (Lemma \ref{171110_3}). 
This completes a sketch of the proof of Theorem \ref{170628_3}.

\section{Proof of $C^0$-approximation lemma (Theorem \ref{170430_2})} 

Recall that $\mca{L}_{k+1}$ is a subspace of $\Pi^{k+1}$, 
which consists of $(\Gamma_0, \ldots, \Gamma_k) \in \Pi^{k+1}$ such that $\ev_1(\Gamma_i) =\ev_0(\Gamma_{i+1})\,(0 \le i \le k-1)$ and $\ev_1(\Gamma_k) = \ev_0(\Gamma_0)$. 
Then, Theorem \ref{170430_2} is reduced to Lemma \ref{170430_3} below. 
We first give a proof of Lemma \ref{170430_3} assuming 
Lemma \ref{171114_1} (which is stated in the proof), 
and proceed to a proof of Lemma \ref{171114_1}. 

\begin{lem}\label{170430_3}
Let $(X, \wh{\mca{U}})$ be a compact K-space and 
$\wh{f}: (X, \wh{\mca{U}}) \to \Pi^\con$ be a strongly continuous map such that 
$\ev_j \circ \wh{f}: (X, \wh{\mca{U}}) \to L$ is strongly smooth for $j=0, 1$. 

Let $Z$ be a closed subset of $X$ and $\wh{g}: (Z, \wh{\mca{U}}|_Z) \to \Pi$
be a strongly smooth map 
(the notion of smooth map from a K-space to $\Pi$ is defined in the same way as Definition \ref{170902_2})
such that 
$\ev_j \circ \wh{g} = \ev_j \circ \wh{f}|_Z$ for  $j=0, 1$ and 
$\wh{g}$ is $\ep$-close to $\wh{f}|_Z$ (with respect to $d_\Pi$). 

If $\ep<\rho_L$, there exist an open substructure $\wh{\mca{U}_0}$ of $\wh{\mca{U}}$ 
and a strongly smooth map $\wh{g'}: (X, \wh{\mca{U}_0}) \to \Pi$ 
such that the following conditions hold: 
\begin{itemize} 
\item $\wh{g'}$ is $\ep$-close to $\wh{f}|_{\wh{\mca{U}_0}}$. 
\item $\ev_j \circ \wh{g'} = \ev_j \circ \wh{f} |_{\wh{\mca{U}_0}}$ for $j=0, 1$. 
\item $\wh{g'} = \wh{g}$ on $\wh{\mca{U}_0}|_Z$. 
\end{itemize} 
\end{lem} 

\begin{proof} 
\textbf{Step 1.} 
There exist 
\begin{itemize} 
\item a GCS $\wt{\mca{U}_Z}$ of $Z$, 
\item a KG-embedding from $\wh{\mca{U}} |_Z$ to $\wt{\mca{U}}_Z$, 
namely  an open substructure $\wh{\mca{U}_{Z,0}}$ of $\wh{\mca{U}}|_Z$ and a strict KG-embedding 
$\Phi_1: \wh{\mca{U}_{Z,0}}  \to \wt{\mca{U}_Z}$, 
\item a strongly continuous map $\wt{f_Z}: (Z, \wt{\mca{U}}_Z) \to \Pi^\con$, 
\item a strongly smooth map $\wt{g}: (Z, \wt{\mca{U}}_Z) \to \Pi$, 
\end{itemize} 
such that 
\begin{itemize}
\item $\Phi_1^* \wt{f}_Z = \wh{f} |_{\wh{\mca{U}_{Z,0}}}$ and $\Phi_1^* \wt{g} = \wh{g}|_{\wh{\mca{U}_{Z,0}}}$, 
\item $\wt{g}$ is $\ep$-close to $\wt{f}_Z$, 
\item $\ev_j \circ \wt{g} = \ev_j \circ \wt{f}_Z$ for $j=0, 1$. 
\end{itemize}

The GCS $\wt{\mca{U}_Z}$ exists by Theorem 3.30 in \cite{FOOO_Kuranishi}. 
By inspecting its proof (Section 11.1 in \cite{FOOO_Kuranishi}), 
one can take $\wt{\mca{U}_Z}$ so that each chart of 
$\wt{\mca{U}_Z}$ is an open subchart of a certain K-chart of $\wh{\mca{U}_Z}$, 
thus one can define $\wt{g}$ and $\wt{f}_Z$ by pulling back $\wh{f}$ and $\wh{g}$. 

\textbf{Step 2.} 
By Proposition 7.52 and Lemma 7.53 in \cite{FOOO_Kuranishi}, there exist 
\begin{itemize} 
\item a GCS $\wt{\mca{U}}$ of $X$.
\item a KG-embedding from $\wh{\mca{U}}$ to $\wt{\mca{U}}$, 
namely and open substructure $\wh{\mca{U}_{0, +}}$ of $\wh{\mca{U}}$ 
and a strict KG-embedding $\Phi_2: \wh{\mca{U}_{0, +}} \to \wt{\mca{U}}$, 
\item an extension from $\wt{\mca{U}_Z}$ to $\wt{\mca{U}}$ (see Definition 7.50 of \cite{FOOO_Kuranishi}), 
namely  an open substructure $\wt{\mca{U}_{Z,0}}$ of $\wt{\mca{U}_Z}$ 
and a strict extension $\Phi_3: \wt{\mca{U}_{Z,0}} \to \wt{\mca{U}}$, 
\item a strongly continuous map $\wt{f}: (X, \wt{\mca{U}}) \to \Pi^\con$
such that $\Phi_2^* \wt{f} = \wh{f} |_{\wh{\mca{U}_{0, +}}}$
and $\Phi_3^* \wt{f} = \wt{f_Z}|_{\wt{\mca{U}_{Z,0}}}$. 
\end{itemize} 

Now we can state a $C^0$-approximation result for GCS 
in  Lemma \ref{171114_1} below, 
which we assume for the moment. 

\begin{lem}\label{171114_1} 
There exist 
\begin{itemize}
\item $\wt{\mca{U}_{Z,00}}$: an open substructure of $\wt{\mca{U}_{Z, 0}}$, 
\item $\wt{\mca{U}_0}$: an open substructure of $\wt{\mca{U}}$, 
\item a strict extension from $\wt{\mca{U}_{Z,00}}$ to $\wt{\mca{U}_0}$, 
\item a strongly smooth map $\wt{g'}: (X, \wt{\mca{U}_0}) \to \Pi$ such that 
\begin{itemize}
\item $\wt{g'}$ is $\ep$-close to $\wt{f}|_{\wt{\mca{U}_0}}$, 
\item $\ev_j \circ \wt{g'} = \ev_j \circ \wt{f}\,(j=0,1)$ on $\wt{\mca{U}_0}$, 
\item $\wt{g'} = \wt{g}$ on $\wt{\mca{U}_{Z,00}}$. 
\end{itemize}
\end{itemize} 
\end{lem}

\textbf{Step 3.} 
Take an open substructure $\wh{\mca{U}_0}$ of $\wh{\mca{U}}$ with a strict KG-embedding $\wh{\mca{U}_0} \to \wt{\mca{U}_0}$
such that $\wh{\mca{U}_0}|_Z \to \wt{\mca{U}_0}$ factors through a strict KG-embedding $\wh{\mca{U}_0}|_Z \to \wt{\mca{U}_{Z,00}}$.
Finally, we can define $\wh{g'}: (X, \wh{\mca{U}_0}) \to \Pi$ by pulling back $\wt{g'}: (X, \wt{\mca{U}_0}) \to \Pi$ 
by $\wh{\mca{U}_0} \to \wt{\mca{U}_0}$. 
\end{proof} 

The rest of this section is devoted to the proof of Lemma \ref{171114_1}. 
First we need to prove Lemmas \ref{171114_2}, \ref{171114_3} and \ref{170513_1}. 

\begin{lem}\label{171114_2}
For any nonempty finite set $I$, there exists a $C^\infty$-map 
\[ 
G: \{ (x_i, t_i )_{i  \in I} \mid x_i \in L, \, t_i \in [0,1], \, \max_{i, j \in I} d_L(x_i, x_j) < \rho_L,  \, \sum_{i \in I} t_i = 1\} \to L 
\] 
such that the following properties hold: 
\begin{enumerate} 
\item[(i):]  If $t_i = \begin{cases} 1 &(i= i_0) \\ 0 &(i \ne i_0) \end{cases}$ for some $i_0 \in I$, then $G (x_i, t_i)_i = x_{i_0}$. 
\item[(ii):] If there exists $y \in L$ and $r \in (0, \rho_L]$ such that $d_L(y, x_i) < r\,(\forall i \in I)$, then 
$d_L(y, G(x_i, t_i)_i) < r$. 
\end{enumerate} 
\end{lem} 
\begin{proof} 
We fix an arbitrary total order on $I$. 
For each $t = (t_i)_{i \in I}$, let $i_0:= \max \{ i \mid t_i \ne 0\}$, 
and for every $\theta \in [0,1]$ let 
\[ 
(t^\theta)_i: = \begin{cases} 
t_i (1-\theta t_{i_0})/(1-t_{i_0}) &(i < i_0) \\ 
\theta t_{i_0} &(i = i_0) \\ 
0 &(i>i_0). 
\end{cases} 
\] 
Now we define $G$ so that it satisfies (i), 
and for any $x = (x_i)_i$, the map 
\[
[0, 1] \to L; \, \theta \mapsto G(x, t^\theta)
\] 
is the shortest geodesic connecting the end points. 
Now $G$ satisfies (ii) since 
any geodesic ball with radius $r \in (0, \rho_L]$ is geodesically convex, 
by the definition of $\rho_L$ (see Section 7.3). 
\end{proof} 

\begin{lem}\label{171114_3} 
For any $(T, \gamma) \in \Pi^\con$ and $\ep>0$, 
there exists $(T', \gamma') \in \Pi$ such that 
$d_\Pi((T,\gamma), (T', \gamma')) < \ep$.
\end{lem}
\begin{proof}
Take $T':= T+\ep/2$, and 
$\gamma': [0,T'] \to L$ so that $d_L(\gamma(sT), \gamma'(sT'))<\ep/2$ for any $s \in [0,1]$, 
which is possible since $C^\infty([0,1], L)$ is dense in $C^0([0,1], L)$ with respect to the $C^0$-topology. 
\end{proof}

\begin{lem}\label{170513_1}
Let $U$ be a $C^\infty$-manifold, 
$f: U \to \Pi^\con$ be a continuous map, 
such that $\ev_j \circ f: U \to L$ is of $C^\infty$ for $j=0, 1$. 

Let $V$ be a submanifold of $U$, 
$g: V \to \Pi$ be a smooth map such that 
$\ev_j \circ g = \ev_j \circ f|_V\,(j=0, 1)$ and 
$g$ is $\ep$-close to $f|_V$. 

Then for any $x \in V$, there exists an open neighborhood $W$ of $x$ in $U$ 
and a smooth map $g': W \to \Pi$ such that 
$g' = g$ on $W \cap V$, 
$\ev_j \circ g' = \ev_j \circ f|_W\,(j=0,1 )$, 
and $g'$ is $\ep$-close to $f|_W$. 
\end{lem} 
\begin{proof}
The last condition  ``$g'$ is $\ep$-close to  $f|_W$'' can be achieved by taking $W$ sufficiently small, 
since $f$ is continuous and $\ep$-closeness is an open condition. 
Thus it is sufficient to define $W$ and a smooth map $g': W \to \Pi$ such that 
$g' = g$ on $W \cap V$ and $\ev_j \circ g'  =\ev_j \circ f|_W\,(j=0, 1)$. 

Let $W$ be a sufficiently small neighborhood of $x$
, so that there exists a $C^\infty$-map $r: W \to W \cap V$ satisfying $r|_{W \cap V} = \id_{W \cap V}$, and 
\[
d_L(\ev_j (f(y)), \ev_j (g \circ r(y))) < r_{\text{inj}}(L) 
\]
for every $y \in W$ and $j \in \{0,1\}$, where $r_{\text{inj}}(L)$ denotes the injectivity radius of $L$. 
We define $\xi_j(y) \in T_{\ev_j(g \circ r(y))}L$ by 
$\exp (\xi_j(y)) = \ev_j (f(y))$. 

We set $g(z):= (T(z), \gamma(z))$ for any $z \in W \cap V$. 
For every $y \in W$ and $j \in \{0, 1\}$, 
we define $\xi^y_j(\theta) \in T_{\gamma(r(y))(\theta)} L \, (0 \le \theta \le T(r(y)))$ so that 
\[ 
\xi^y_0(0) = \xi_0(y), \quad \nabla_\theta \xi^y_0  \equiv 0, \quad 
\xi^y_1(T(r(y)) = \xi_1(y), \quad \nabla_\theta  \xi^y_1 \equiv 0. 
\] 
Taking a $C^\infty$-function $\chi: [0,1] \to [0,1]$ such that  
$\chi \equiv 0$ near $0$ and 
$\chi \equiv 1$ near $1$, 
we define 
$\xi'(y) \in C^\infty( \gamma(r(y))^*TL)$ by 
\[ 
\xi'(y) (\theta):= \chi(\theta/T(r(y))) \cdot \xi^y_1(\theta ) + (1 - \chi (\theta/T(r(y)))) \cdot \xi^y_0(\theta)
\qquad (0 \le \theta \le T(r(y))). 
\]
Finally, we define 
$\gamma'(y): [0, T(r(y))] \to L$ by 
\[ 
\gamma'(y)(\theta):= \exp (\gamma(r(y))(\theta), \xi'(y)(\theta))
\] 
and 
$g'(y):= (T(r(y)), \gamma'(y))$. 
It is easy to check that this map $g'$ is smooth and satisfies required conditions. 
\end{proof}

Now let us start the proof of Lemma \ref{171114_1}. 
Let $\mf{P}$ denote the index set of $\wt{\mca{U}}$, 
and $\mf{P}_Z \subset \mf{P}$ denote the index set of $\wt{\mca{U}_Z}$. 
Let $\{\mca{U}_{\mf{p}}\}_{\mf{p} \in \mf{P}}$ be K-charts of $\wt{\mca{U}}$, and we denote $\mca{U}_{\mf{p}} = (U_{\mf{p}}, \ldots )$ for each $\mf{p} \in \mf{P}$. 
Let $\{\mca{U}^Z_{\mf{p}}\}_{\mf{p} \in \mf{P}_Z}$ be K-charts of $\wt{\mca{U}_{Z,0}}$ and we denote $\mca{U}^Z_{\mf{p}}=(U^Z_{\mf{p}}, \ldots )$ for each $\mf{p} \in \mf{P}_Z$.
We also take support systems $\{ \mca{K}_{\mf{p}}\}_{\mf{p} \in \mf{P}}$ of $\wt{\mca{U}}$ 
and $\{\mca{K}^Z_{\mf{p}}\}_{\mf{p} \in \mf{P}_Z}$ of $\wt{\mca{U}_{Z,0}}$, 
such that 
$(\ph_3)_{\mf{p}} (\mca{K}^Z_{\mf{p}}) \subset \mtrg{\mca{K}}_{\mf{p}}$
for every $\mf{p} \in \mf{P}_Z$,
where $\Phi_3 = (\ph_3, \wh{\ph}_3)$ is a strict extension from $\wt{\mca{U}_{Z,0}}$ to $\wt{\mca{U}}$. 

A subset $\mf{F} \subset \mf{P}$ is called a \textit{filter} 
if $\mf{p}, \mf{q} \in \mf{P}$, $\mf{p} \ge \mf{q}$, $\mf{p} \in \mf{F}$ imply $\mf{q} \in \mf{F}$. 
In particular the empty set is a filter (see Section 12.3 in \cite{FOOO_Kuranishi}). 
Now Lemma \ref{171114_1} is proved by applying Lemma \ref{170507_1} below for $\mf{F}=\mf{P}$, 
and setting 
\[ 
(\mca{U}_0)_{\mf{p}}: = \mca{U}_{\mf{p}}|_{V_{\mf{p}}} \,(\forall \mf{p} \in \mf{P}), \quad
(\mca{U}_{Z,00})_{\mf{p}}: = (\mca{U}_{Z,0})_{\mf{p}}|_{W_{\mf{p}}} \,(\forall \mf{p} \in \mf{P}_Z).
\]

\begin{lem}\label{170507_1} 
For any filter $\mf{F}$ of $\mf{P}$, 
there exist 
$(V_{\mf{p}}, g'_{\mf{p}})_{\mf{p} \in \mf{F}}$ and $(W_{\mf{p}})_{\mf{p} \in \mf{P}_Z}$
such that the following conditions are satisfied: 
\begin{itemize}
\item $V_{\mf{p}}$ is an open neighborhood of $\mca{K}_{\mf{p}}$ in $U_{\mf{p}}$ for every $\mf{p} \in \mf{F}$. 
\item $g'_{\mf{p}}: V_{\mf{p}} \to \Pi$ is a smooth map. 
Moreover, $g'_{\mf{p}}$ is $\ep$-close to $f_{\mf{p}}|_{V_{\mf{p}}}$ and $\ev_j \circ g'_{\mf{p}} = \ev_j \circ f_{\mf{p}}|_{V_{\mf{p}}}$ for $j=0,1$.
\item $(g'_{\mf{p}})_{\mf{p} \in \mf{F}}$ is compatible with coordinate changes. 
Specifically, for any $\mf{p}, \mf{p}' \in \mf{F}$ such that $\mf{p} \ge \mf{p}'$, 
there holds $g_{\mf{p}'} = g_{\mf{p}} \circ \ph_{\mf{p}\mf{p}'}$ on $V_{\mf{p}'} \cap \ph_{\mf{p}\mf{p}'}^{-1}(V_{\mf{p}})$. 
\item $W_{\mf{p}}$ is an open neighborhood of $\mca{K}^Z_{\mf{p}}$ in $U^Z_{\mf{p}}$ for every $\mf{p} \in \mf{P}_Z$. 
\item $(g'_{\mf{p}})_{\mf{p} \in \mf{F}}$ is compatible with $\wt{g}$. Specifically:
\begin{itemize}
\item For any $\mf{p} \in \mf{P}_Z \cap \mf{F}$, there holds $(\ph_3)_{\mf{p}}(W_{\mf{p}}) \subset V_{\mf{p}}$ and $g'_{\mf{p}} = g_{\mf{p}} \circ (\ph_3)_{\mf{p}}$ on $W_{\mf{p}}$. 
\item For any $\mf{q} \in \mf{P}_Z \setminus \mf{F}$ and $\mf{p} \in \mf{F}$ satisfying $\mf{p} < \mf{q}$, there holds 
$g_{\mf{q}} \circ \ph_{\mf{q}\mf{p}} = g'_{\mf{p}}$ on $V_{\mf{p}} \cap (\ph_{\mf{q}\mf{p}})^{-1}(W_{\mf{q}})$. 
\end{itemize} 
\end{itemize} 
\end{lem}
\begin{proof}
The proof is by induction on the cardinality of $\mf{F}$. 
There is nothing to prove when $\mf{F} = \emptyset$.
To discuss the induction step, let $\mf{F}$ be a filter, $\mf{p}_0$ be its maximal element
(namely $\mf{p} \in \mf{F}, \mf{p} \ge \mf{p}_0 \implies \mf{p} = \mf{p}_0$), 
and suppose that there exist $(V_{\mf{p}}, g'_{\mf{p}})_{\mf{p} \in \mf{F} \setminus \{\mf{p}_0\}}$ and $(W_{\mf{q}})_{\mf{q} \in \mf{P}_Z}$
satisfying the conditions in the lemma. 

For each $x \in \mca{K}_{\mf{p}_0}$, we take an open neighborhood $W_x$ of $x$ in $U_{\mf{p}_0}$
and a smooth map $g'_x: W_x \to \Pi$ in the way described below. 
We consider three cases (here we consider $x$ as a point in $|\wt{\mca{U}}| = \bigsqcup_{\mf{p} \in \mf{P}} U_{\mf{p}}/\sim$, and $\mca{K}^Z_{\mf{p}}$, $\mca{K}_{\mf{p}}$ as subspaces of $|\wt{\mca{U}}|$). 
\begin{itemize}
\item[(i):] There exists $\mf{q} \in \mf{P}_Z \setminus (\mf{F} \setminus \{\mf{p}_0\})$ such that $x \in \mca{K}^Z_{\mf{q}}$. 
\item [(ii):] $x \notin \mca{K}^Z_{\mf{q}}$ for any $\mf{q} \in \mf{P}_Z \setminus (\mf{F} \setminus \{\mf{p}_0\})$, but there exists $\mf{q} \in \mf{F} \setminus \{\mf{p}_0\}$ 
such that $x \in \mca{K}_{\mf{q}}$. 
\item[(iii):] $x \notin \mca{K}^Z_{\mf{q}}$ for any $\mf{q} \in \mf{P}_Z \setminus (\mf{F} \setminus \{\mf{p}_0\})$, and $x \notin \mca{K}_{\mf{q}}$ for any $\mf{q} \in \mf{F} \setminus \{\mf{p}_0\}$. 
\end{itemize} 

In case (i), take maximal $\mf{q} \in \mf{P}_Z \setminus (\mf{F} \setminus \{\mf{p}_0\})$ such that $x \in \mca{K}^Z_{\mf{q}}$. 
This condition implies $\mf{q} > \mf{p}_0$ since $U_{\mf{q}} \cap U_{\mf{p}_0} \ne \emptyset$ and $\mf{q} \in \mf{F}$. 
Then take $W_x$ such that: 
\begin{itemize}
\item 
$\ol{W_x} \cap \mca{K}^Z_{\mf{q}'} = \emptyset$ for any $\mf{q}' \in \mf{P}_Z \setminus (\mf{F} \setminus \{\mf{p}_0\})$ which does \textit{not} satisfy $\mf{q}' \le \mf{q}$
(note this condition implies that $\mca{K}^Z_{\mf{q}} \cap \mca{K}^Z_{\mf{q}'} = \emptyset$), 
\item 
$W_x \subset (\ph_{\mf{q}\mf{p}_0})^{-1}(W_{\mf{q}})$. 
\end{itemize} 
Then we define $g'_x: W_x \to \Pi$ by 
$g'_x: = g_{\mf{q}} \circ \ph_{\mf{q}\mf{p}_0}|_{W_x}$. 

In case (ii), take maximal $\mf{q} \in \mf{F} \setminus \{\mf{p}_0\}$ such that $x \in \mca{K}_{\mf{q}}$.
This condition implies $\mf{q} < \mf{p}_0$. 
Then take $W_x$ such that: 
\begin{itemize}
\item $\ol{W_x} \cap \mca{K}_{\mf{q}'} = \emptyset$ for any $\mf{q}' \in \mf{F} \setminus \{\mf{p}_0\}$ which does \textit{not} satisfy $\mf{q}' \le \mf{q}$
(note this condition implies that $\mca{K}_{\mf{q}} \cap \mca{K}_{\mf{q}'} = \emptyset$), 
\item $\ol{W_x} \cap \mca{K}^Z_{\mf{q}'} = \emptyset$ for any $\mf{q}' \in \mf{P}_Z \setminus (\mf{F} \setminus \{\mf{p}_0\})$, 
\item $(\ph_{\mf{p}_0\mf{q}})^{-1}(W_x) \subset V_{\mf{q}}$. 
\end{itemize} 
When $W_x$ is sufficiently small, Lemma \ref{170513_1} shows that there exists a smooth map $g'_x: W_x \to \Pi$ such that 
\begin{itemize} 
\item $g'_x$ is $\ep$-close to $f_{\mf{p}_0}|_{W_x}$.
\item $\ev_j \circ g'_x = \ev_j \circ f_{\mf{p}_0}|_{W_x}$ for $j=0, 1$. 
\item $g'_x \circ \ph_{\mf{p}_0\mf{q}} = g'_{\mf{q}}$ on $(\ph_{\mf{p}_0\mf{q}})^{-1}(W_x)$. 
\end{itemize} 

In case (iii), take $W_x$ so that 
$\ol{W_x} \cap \mca{K}_{\mf{q}} = \emptyset$ for every $\mf{q} \in \mf{F} \setminus \{\mf{p}_0\}$, and 
$\ol{W_x} \cap \mca{K}^Z_{\mf{q}} = \emptyset$ for every $\mf{q} \in \mf{P}_Z \setminus (\mf{F} \setminus \{\mf{p}_0\})$. 
When $W_x$ is sufficiently small, 
Lemmas \ref{171114_3} and \ref{170513_1} (applied to $V= \{x\}$)
show that there exists a smooth map $g'_x: W_x \to \Pi$ such that 
\begin{itemize} 
\item $g'_x$ is $\ep$-close to $f_{\mf{p}_0}|_{W_x}$. 
\item $\ev_j \circ g'_x = \ev_j \circ f_{\mf{p}}|_{W_x}$ for $j=0, 1$. 
\end{itemize} 

Since $\mca{K}_{\mf{p}_0}$ is compact, 
one can take finitely many points $\{x_i\}_{i \in I}$ such that 
$\{W_{x_i}\}_{i \in I}$ covers $\mca{K}_{\mf{p}_0}$. 
For each $i \in I$, we take 
$g'_{x_i}: W_{x_i} \to \Pi$ and denote it as 
$g'_{x_i} = (\gamma_i, T_i)$. 
Let us take a $C^\infty$-function $\chi_i: U_{\mf{p}_0} \to \R_{\ge 0}$ for each $i \in I$, 
such that $\supp \chi_i \subset W_{x_i}$
and $\sum_{i \in I} \chi_i = 1$ on a neighborhood of $\mca{K}_{\mf{p}_0}$, which we denote by $V_{\mf{p}_0}$. 
Then we define $g'_{\mf{p}_0}: V_{\mf{p}_0} \to \Pi$ by 
$g'_{\mf{p}_0} := (T, \gamma)$, such that 
$T: = \sum_{i \in I} \chi_i T_i$, 
and for every $y \in V_{\mf{p}_0}$
\[
\gamma(y): [0, T(y)] \to L ; \quad 
\theta \mapsto 
G (\gamma_i(y)(T_i(y)\theta/T(y)), \chi_i(y))_{i \in I} 
\]
where $G$ is defined in Lemma \ref{171114_2}. 
If $i \in I$ satisfies $y \in V_{x_i}$, 
then
$g'_{x_i}(y)$ is $\ep$-close to $f_{\mf{p}_0}(y)$, 
and $\ep < \rho_L$, thus 
$g'_{\mf{p}_0}(y)$ is $\ep$-close to $f_{\mf{p}_0}(y)$. 

Now we can finish the proof by replacing $V_{\mf{p}}$ with a smaller neighborhood of $\mca{K}_{\mf{p}}$ for each $\mf{p} \in \mf{F} \setminus \{\mf{p}_0\}$, 
and $W_{\mf{q}}$ with a smaller neighborhood of $\mca{K}^Z_{\mf{q}}$ for each $\mf{q} \in \mf{P}_Z$. 
\end{proof}

\section{Some basic notions in the theory of Kuranishi structures}

Here we recall some basic notions
in the theory of Kuranishi structures (abbreviated as K-structures), 
mainly to fix notations used throughout this paper. 
When we use notions which are not recalled here, 
we directly refer to \cite{FOOO_Kuranishi}. 
Throughout this section, 
$X$ denotes a separable, metrizable topological space. 

\textbf{Kuranishi chart (K-chart)} 

A \textit{K-chart} of $X$ is a tuple 
$\mca{U}=(U, \mca{E}, s, \psi)$ such that: 
\begin{itemize}
\item $U$ is a $C^\infty$-manifold, 
\item $\mca{E}$ is a $C^\infty$-vector bundle on $U$, 
\item $s$ is a $C^\infty$-section of $\mca{E}$, 
\item $\psi: s^{-1}(0) \to X$ is a homeomorphism onto an open set of $X$. 
\end{itemize}
$\dim \mca{U}: = \dim U - \rk \mca{E}$ is called the dimension of $\mca{U}$. 
An orientation of $\mca{U}$ is a pair of orientations of $U$ and $\mca{E}$. 
A K-chart at $p \in X$ is a K-chart $\mca{U}=(U, \mca{E}, s, \psi)$ such that
$p \in \Image \psi$. We denote $o_p:= \psi^{-1}(p) \in s^{-1}(0)$. 

\begin{rem}\label{170912_1} 
In the standard definition (see Definition 3.1 in \cite{FOOO_Kuranishi}), 
one assumes that 
$U$ is an orbifold and $\mca{E}$ is an orbibundle. 
However, 
in the present paper we are working with pseudo-holomorphic disks \textit{without} sphere bubbles, 
thus we do not need to take quotients by finite group actions on moduli spaces, 
hence it is sufficient to work with vector bundles on manifolds. 
\end{rem}

\textbf{Embedding of K-charts} 

Let $\mca{U}_i = (U_i, \mca{E}_i, s_i, \psi_i)\,(i=1, 2)$ be K-charts of $X$. 
An \textit{embedding} of K-charts $\Phi: \mca{U}_1 \to \mca{U}_2$ 
is a pair $\Phi = (\ph, \wh{\ph})$ such that: 
\begin{itemize} 
\item $\ph: U_1 \to U_2$ is an embedding of $C^\infty$-manifolds, 
\item $\wh{\ph}: \mca{E}_1 \to \mca{E}_2$ is an embedding of $C^\infty$-vector bundles over $\ph$, 
\item $\wh{\ph} \circ s_1 = s_2 \circ \ph$, 
\item $\psi_2 \circ \ph = \psi_1$ on $s_1^{-1}(0)$, 
\item for any $x \in s_1^{-1}(0)$, the covariant derivative
\begin{equation}\label{171119_1} 
D_{\ph(x)} s_2: 
\frac{ T_{\ph(x)} U_2}{(D_x\ph)(T_x U_1)} \to \frac{ (\mca{E}_2)_{\ph(x)}}{\wh{\ph}((\mca{E}_1)_x)} 
\end{equation} 
is an isomorphism. 
\end{itemize} 
When K-charts $\mca{U}_1$ and $\mca{U}_2$ are oriented, 
we say that $\Phi = (\ph, \wh{\ph})$ is orientation preserving, 
if an isomorphism 
\[ 
\det TU_2 \otimes (\det TU_1)^\vee \cong 
\det \mca{E}_2 \otimes (\det \mca{E}_1)^\vee
\] 
induced by (\ref{171119_1}) preserves orientations.

\textbf{Coordinate changes}

Let $\mca{U}_i = (U_i, \mca{E}_i, s_i, \psi_i) \,(i=1, 2)$ be K-charts of $X$. 
A \textit{coordinate change} in weak sense (resp. strong sense) 
from $\mca{U}_1$ to $\mca{U}_2$ is a triple 
$\Phi_{21} = (U_{21}, \ph_{21}, \wh{\ph}_{21})$ 
satisfying (i) and (ii) (resp. (i), (ii) and (iii)): 
\begin{itemize} 
\item[(i):] $U_{21}$ is an open subset of $U_1$, 
\item[(ii):] $(\ph_{21}, \wh{\ph}_{21})$ is an embedding of  K-charts $\mca{U}_1|_{U_{21}} \to \mca{U}_2$. 
\item[(iii):] $\psi_1(s_1^{-1}(0) \cap U_{21}) = \Image \psi_1 \cap \Image \psi_2$. 
\end{itemize} 

\textbf{Kuranishi structure (K-structure)} 

A \textit{K-structure}  $\wh{\mca{U}}$ of $X$ (of dimension $d$) consists of 
\begin{itemize} 
\item a K-chart (of dimension $d$) $\mca{U}_p = (U_p, \mca{E}_p, s_p, \psi_p)$ at $p$ for every $p \in X$, 
\item a coordinate change in weak sense 
$\Phi_{pq} = (U_{pq}, \ph_{pq}, \wh{\ph}_{pq}): \mca{U}_q \to \mca{U}_p$ 
for every $p \in X$ and $q \in \Image (\psi_p)$, 
\end{itemize}
such that 
\begin{itemize} 
\item $o_q \in U_{pq}$ for every $q \in \Image \psi_p$, 
\item for every $p \in X$, $q \in \Image \psi_p$ and $r \in \psi_q(s_q^{-1}(0) \cap U_{pq})$, 
there holds 
$\Phi_{pr}|_{U_{pqr}} = \Phi_{pq} \circ \Phi_{qr}|_{U_{pqr}}$
where $U_{pqr} := \ph^{-1}_{qr}(U_{pq}) \cap U_{pr}$. 
\end{itemize} 
The pair $(X, \wh{\mca{U}})$ is called a space with Kuranishi structure
(abbreviated as a \textit{K-space})
of dimension $d$. 
We say $\wh{\mca{U}}$ is oriented if each K-chart $\mca{U}_p$ is oriented and 
$\ph_{pq}$ preserves orientations for every $p \in X$ and $q  \in \Image \psi_p$. 

\begin{rem} 
In Definition 3.11 \cite{FOOO_Kuranishi} the notion of a relative K-space $(X, Z; \wh{\mca{U}})$ is introduced, 
however in this paper we only need the absolute case ($X=Z$). 
\end{rem} 

\textbf{Strongly continuous/smooth maps from K-spaces} 

Let $(X, \wh{\mca{U}})$ be a K-space and $Y$ be a topological space. 
A \textit{strongly continuous map} 
$\wh{f}: (X, \wh{\mca{U}}) \to Y$ 
assigns a continuous map 
$f_p: U_p \to Y$ for every $p \in X$ such that 
$f_p \circ \ph_{pq} = f_q$ on $U_{pq}$ for every $p \in X$ and $q \in \Image \psi_p$. 
When $Y$ has a structure of a $C^\infty$-manifold, 
$\wh{f}$ is called 
\textit{strongly smooth}
 if $f_p: U_p \to Y$ is of $C^\infty$ for every $p \in X$.
Moreover, 
$\wh{f}$ is called 
\textit{weakly submersive}, 
if $f_p$ is a submersion for every $p \in X$. 

\textbf{Good coordinate system (GCS)}

Finally we recall the definition of a 
\textit{good coordinate system} (GCS). 
A GCS $\wt{\mca{U}}$ of $X$ consists of 
\[ 
( (\mf{P}, \le) , \{ \mca{U}_{\mf{p}}\}_{\mf{p} \in \mf{P}} , \{ \Phi_{\mf{p}\mf{q}}\}_{\mf{q} \le \mf{p}})
\] 
such that: 
\begin{itemize} 
\item $(\mf{P}, \le)$ is a finite partially ordered set. 
\item $\mca{U}_{\mf{p}} = (U_{\mf{p}}, \mca{E}_{\mf{p}}, s_{\mf{p}}, \psi_{\mf{p}})$ is a K-chart for each $\mf{p} \in \mf{P}$, and $\bigcup_{\mf{p} \in \mf{P}} \Image \psi_{\mf{p}} = X$. 
\item $\Phi_{\mf{p}\mf{q}}$ is a coordinate change $\mca{U}_{\mf{q}} \to \mca{U}_{\mf{p}}$ in strong sense. 
\item If $\mf{r} \le \mf{q} \le \mf{p}$, 
there holds $\Phi_{\mf{p}\mf{r}}|_{U_{\mf{p}\mf{q}\mf{r}}}  = \Phi_{\mf{p}\mf{q}} \circ \Phi_{\mf{q}\mf{r}}|_{U_{\mf{p}\mf{q}\mf{r}}}$, 
where $U_{\mf{p}\mf{q}\mf{r}}: = \ph^{-1}_{\mf{q}\mf{r}}(U_{\mf{p}\mf{q}}) \cap U_{\mf{p}\mf{r}}$. 
\item If $\Image \psi_{\mf{p}} \cap \Image \psi_{\mf{q}} \ne \emptyset$, then either $\mf{p} \le \mf{q}$ or $\mf{q} \le \mf{p}$ holds. 
\item Let us define a relation $\sim$ on $\bigsqcup_{\mf{p} \in \mf{P}} U_{\mf{p}}$ as follows: 
$x \sim y$ if and only if one of the following holds: 
\begin{itemize}
\item $\mf{p} = \mf{q}$ and $x=y$. 
\item $\mf{p} \le \mf{q}$ and $y = \ph_{\mf{q}\mf{p}}(x)$. 
\item $\mf{q} \le \mf{p}$ and $x = \ph_{\mf{p}\mf{q}}(y)$. 
\end{itemize} 
Then the relation $\sim$ is an equivalence relation, and the quotient 
$\bigg(\bigsqcup_{\mf{p} \in \mf{P}} U_{\mf{p}}\bigg)/\sim$, 
equipped with the quotient topology, is Hausdorff. 
\end{itemize} 

The definitions of strongly continuous/smooth maps naturally extend to spaces with GCS. 
Finally, there exists a natural notion of embeddings from a K-structure to a GCS (KG-embedding; see Definition 3.29 in \cite{FOOO_Kuranishi}). 
For any K-structure $\wh{\mca{U}}$ on a compact space $X$, 
there exist a GCS $\wt{\mca{U}}$ and a KG-embedding $\wh{\mca{U}} \to \wt{\mca{U}}$; 
see Theorem 3.30 in \cite{FOOO_Kuranishi}.

\end{document}